\documentclass[11pt,letterpaper]{article}
\usepackage[utf8]{inputenc}
\usepackage{tikz}
\usepackage{bm}
\usepackage{mathtools}
\usepackage{tikz-cd}
\usepackage{mathrsfs}
\usepackage{cite}
\usepackage{amsmath}
\numberwithin{equation}{section}
\usepackage{hyperref}
\usepackage{lineno}
\usepackage{mathrsfs}
\usepackage{amsfonts}
\usepackage{amsthm}
\usepackage{amssymb}
\theoremstyle{definition}
\newtheorem{define}{Definition}[section]
\newtheorem{example}[define]{Example}
\newtheorem{construction}{Construction}[section]
\theoremstyle{remark}
\newtheorem{remark}[define]{Remark}
\theoremstyle{plain}
\newtheorem{theo}[define]{Theorem}
\newtheorem{lemma}[define]{Lemma}
\newtheorem{prop}[define]{Proposition}
\newtheorem{cor}[define]{Corollary}
\newcommand{\X}{\mathscr X}
\newcommand{\C}{\mathscr C}
\newcommand{\E}{\mathscr E}
\newcommand{\G}{\mathscr G}
\newcommand{\A}{\mathscr A}
\newcommand{\B}{\mathscr B}
\newcommand{\F}{\mathscr F}
\newcommand{\D}{\mathscr D}

\newcommand{\K}{\mathfrak K}

\newcommand{\cat}{\mathbf{Cat}}
\newcommand{\twocat}{\mathbf{2Cat}}
\newcommand{\elt}{\mathbf{Elt}}

\newcommand{\dom}{\mathrm{dom}}
\newcommand{\cod}{\mathrm{cod}}
\newcommand{\src}{\mathrm{src}}
\newcommand{\tgt}{\mathrm{tgt}}
\newcommand{\sq}{\mathbf{Sq}}
\newcommand{\lax}{\mathbf{Lax}}

\makeatletter
\newcommand{\laxlim@}[2]{%
  \vtop{\m@th\ialign{##\cr
    \hfil$#1\operator@font lax$\hfil\cr
    \noalign{\nointerlineskip\kern1.5\ex@}#2\cr
    \noalign{\nointerlineskip\kern-\ex@}\cr}}%
}
\newcommand{\laxlim}{%
  \mathop{\mathpalette\laxlim@{\leftarrowfill@\textstyle}}\nmlimits@
}
\makeatother
\makeatletter
\newcommand{\laxcolim@}[2]{%
  \vtop{\m@th\ialign{##\cr
    \hfil$#1\operator@font lax$\hfil\cr
    \noalign{\nointerlineskip\kern1.5\ex@}#2\cr
    \noalign{\nointerlineskip\kern-\ex@}\cr}}%
}
\newcommand{\laxcolim}{%
  \mathop{\mathpalette\laxcolim@{\rightarrowfill@\textstyle}}\nmlimits@
}
\makeatother
\usepackage{amssymb}
\usepackage[left=2cm,right=2cm,top=2cm,bottom=2cm]{geometry}
\author{Michael Lambert}
\title{Discrete 2-Fibrations}
\begin{document}
\maketitle

\begin{abstract}
This paper is concerned with developing a 2-dimensional analogue of the notion of an ordinary discrete fibration.  A definition is proposed, and it is shown that such discrete 2-fibrations correspond via a 2-equivalence to certain category-valued 2-functors.  The ultimate goal of the paper is to show that discrete 2-fibrations are 2-monadic over a slice of the 2-category of categories.
\end{abstract}

\tableofcontents

\section{Introduction}

This opening section starts off by summarizing some motivation and background on categorifying the notion of a discrete fibration.  Some specifics and notation are then recalled on fibrations, cleavages, discrete fibrations and the representation theorems relating fibrations to certain category-valued functors.  Finally, a brief breakdown on the various sections of the paper is given.

\subsection{Introductory Discussion}

This paper is concerned with identifying a notion of ``discrete 2-fibration," categorifying in dimension 2 the idea of a discrete fibration.  Roughly, a discrete fibration is a functor having a lifting property with respect to morphisms in the target category whose codomains are strictly in the image of the functor.  The main representation theorem for discrete fibrations over a fixed base category is that they correspond via an equivalence of categories to set-valued contravariant functors on the base category.  In some ways the point is that the uniqueness clause of the lifting property causes the ``fibers" of any discrete fibration to be sets rather than categories with potentially non-trivial morphisms.  Taking these sets as the values of the corresponding functor gives one direction of the equivalence, whereas the other direction is given by a canonical ``category of elements" construction.  It is a version of these correspondences that this paper seeks to work out in dimension 2.  The only immediate clue for a starting point is that the category of sets is the base structure for category theory.  Thus, roughly speaking, insofar as the 2-category of categories is viewed as a base structure for 2-categories, the idea is to identify the fibration concept corresponding to contravariant category-valued 2-functors indexed by a base 2-category. 

One important question, however, concerns the way in which such a 2-fibration concept would be considered discrete.  For recall that discrete fibrations are, from one point of view, a fragment of a theory of fibrations.  Fibrations are functors with lifing properties via special ``cartesian morphisms" whose fibers are not necessarily just sets.  The representation theorem is that certain fibrations with specified or canonically chosen cartesian morphisms over a fixed base correspond via an equivalence of 2-categories to certian contravariant category-valued functors on the base.  Every discrete fibration is a (split) fibration and the representation theorem for discrete fibrations is a fragment of that for (split) fibrations.  Thus, just as discrete fibrations are fibrations and the representation theorem for the latter is a special case of the result for the former, we should plan to isolate a notion of discrete 2-fibration such that at least: every discrete 2-fibration is some kind of 2-fibration; and any representation result for the former is canonically a special case of one for the former.  

Thankfully the literature on 2-fibrations gives a well-established starting place that turns out to fit into a higher-dimension version of the pattern suggested above.  The notion of 2-fibration considered here was introduced in a form in \cite{Hermida} and reworked in \cite{Buckley} to give a representation theorem of the following form: namely, that 2-fibrations over a fixed base 2-category, so defined, correspond via an equivalence of 3-categories to contravariant 2-category-valued 2-functors on the base 2-category.  Thus, our initial idea that discrete 2-fibrations should correspond to category-valued functors will fit the pattern suggested at the lower level.  For at the low level (1) sets are locally discrete with respect to categories; (2) the category of sets is a base structure for categories; and (3) the 2-category of categories is an ambient forum for category theory.  The categorification will have (1)' categories are locally discrete with respect to 2-categories: (2)' the category of categories is a base structure for 2-categories; and (3)' the 3-category of 2-categories is an ambient forum for 2-category theory.  Thus, given that discrete 2-fibrations should correspond to category-valued functors on a 2-category, the question is: what aspect of the definition of 2-fibration needs to be adjusted to ensure this discreteness relative to 2-categories?

Roughly speaking, the definition of a 2-fibration from the references is that it is a 2-functor with fibration-like lifting properties of both arrows and 2-cells via certain specified cartesian morphisms, but that is also a split (op)fibration locally.  The correspondence for a representation theorem is achieved, as in the lower-dimensional case, by a category of elements construction and a suitable pseudo-inverse, explicitly constructed in \cite{Buckley}.  Now, our examination of the fibration properties of the category of elements construction, when applied to category-valued 2-functors, reveals that the canonical projection from the category of elements is essentially a 2-fibration which is locally a discrete (op)fibration.  It turns out that asking a 2-functor to be a split fibration and locally a discrete (op)fibration is precisely what gives the discrete 2-fibration concept and the restricted equivalence that it is the first object of this paper to exhibit.  Exhibiting this correspondence directly and showing how the discreteness assumption is used throughout is the content of \S 3 of the paper.

One technical point should be noted at this stage.  This is that the representation correspondences, as developed here, treat the notion of contravariance in dimension 2 as the dual on 1-cells, but not necessarily on 2-cells.  The reference \cite{Buckley} considers functors on the ``$coop$" dual to correspond to the notion of 2-fibration, resulting in a definition of 2-fibration as having the fibration-like lifting property globally and being a split fibration locally.  Owing to the fact that the canonical representable functors for any object in a fixed 2-category are not defined on the 2-dimensional `$co$' dual, we view the notion of contravariance as defined only on the ``$op$" dual, resulting in the definition of 2-fibration as being a split opfibration locally instead.  Perhaps this convention is ultimately unjustified.  However, the importance of the canonical representable functors is, to us, highlighted by the fact that they act as units for what appears to be a tensor product of category-valued 2-functors as constructed in a previous paper \cite{Me2}.  This construction boosts the 1-categorical tensor product of presheaves into the 2-categorical setting by exhibiting a category computing the colimit of a category-valued functor on a 2-category with a given weight 2-functor that formally resembles a  tensor-hom adjunction.  Just as the canonical representables are units for the tensor product in dimension 1, the canonical representable 2-functors are  units for the tensor in dimension 2.  Just as 1-dimsional representables correspond to certain (op)fibrations via the representation theorems for discrete (op)fibrations, we think the 2-dimensional representables should correspond to 2-(op)fibration in the higher version of this duality.

These considerations bring us to discuss the second overall objective of the paper.  Discrete fibrations over a fixed base category are well-known to be monadic over a certain slice of the category of sets.  The importance of this result shows up in the following development.  That is, the tensor product extension of a presheaf along the Yoneda embedding is left exact if, and only if, the associated category of elements is filtered (see Ch VII of \cite{MM} for a textbook account).  R. Diaconescu gave an elementary version of this result in \cite{DiaconescuThesis} and \cite{DiaconescuChangeOfBase} replacing the category of sets by an arbitrary topos and working in internal category theory.  Of course the idea of a ``base-valued functor" really has no strict internal analogue.  However, presheaf data can be captured by looking at the corresponding algebras, which turn out to be just morphisms of the topos over a fixed base admitting a certain action from the arrows of the base category.  In this way, Diaconescu gave an elementary version of the filtering axioms and showed that a given algebra is filtered if, and only if, its internal tensor product extension is left exact.  Given that there is known to be a tensor product of category-valued 2-functors, the question is about the nature of filtering conditions on the corresponding 2-category of elements that are equivalent to the exactness of the tensor product extension.  This question was answered in \cite{DDS} where a notion of 2-filteredness was proposed.  Any ``elementary" version of the results obtained there in some kind of 2-topos along the pattern of Diaconescu's work would first require an repackaging of the base-valued 2-functor as some kind of algebra, ready for elementary generalization.  It is the second overall goal of this paper to show that the concept of discrete 2-fibration proposed here is 2-monadic over a slice of the 2-category of categories.  This notion appears in \S 4 and could be appropriate for elementary axiomatization.

Some subtleties are encountered in this development, however.  Fibrations over a fixed base category were shown \cite{GrayFiberedCofibered} to be algebras for a monad given by an action of the cotensor of the base category with the two-element ordinal category, that is, an action of the ``arrow category" of the base category.  Since a discrete 2-fibration is at least a split fibration, we thus expect the action to be from some kind of 2-dimensional cotensor or arrow 2-category constructed from the base.  There is of course such an arrow 2-category, consisting of arrows, commutative squares and pairs of certain 2-cells satisfying a compatibility condition, directly adding higher structure to the ordinary arrow category in dimension 1.  This is the cotensor in the 3-category of 2-categories, 2-functors, 2-natural transformations and modifications.  However, because a discrete 2-fibration is also locally a discrete opfibration, it admits an action from a more general structure, namely a ``lax arrow category" where the commutative squares are replaced by squares with not just an isomorphism but instead a mere 2-cell pointing in the correct direction.  And indeed it turns out that to obtain a discrete 2-fibration from a functor admitting an action of some kind of 2-categorical arrow structure, both the commutative squares and the globular structure of the base 2-category are required.  For the squares give the transition functors and the globular structure gives the transition 2-cells.  It appears that the only way to pack all this information into a 1-category is to take the underlying 1-category of this ``lax arrow category."  One expects, then, that 2-fibrations are precisely the algebras for actions the whole 2-category and not just the underlying structure.

The peculiarity of this development is that the lax arrow category is not a cotensor in 2-categories.  For it is universal not among 2-natural transformations but among certain \textit{lax} natural transformations.  Accordingly, a good portion of the paper is devoted to building up a 3-dimensional ambient categorical structure to describe this universality, namely, a setting whose objects are 2-categories, whose morphisms are 2-functors, whose 2-cells are lax natural transformations, and whose 3-cells are modifications.  This can be obtianed as a category enriched in the 1-category of 2-categories and lax functors.  Then the lax arrow category enjoys a universal property in this setting very much like the ordinary  cotensor in more familiar 2- and 3-categories.  It is thus this 3-categorical setting that we anticipate is an appropriate ambient forum for studying 2-fibrations and discrete 2-fibrations.  In fact, this is already suggested by the theory of ordinary fibrations.  For the category of elements construction associated to a category-valued functor fits into a certain lax comma square; moreover taking the strict comma square in this situation would result in asking for strict equalities instead of arbitrary vertical morphisms, which is first of all not the usual construction and second generally regarded as ``evil" since equality in categories is not invariant under equivalence.  The question, then, is whether there is a more general 2-fibration concept to be obtained in this setting as a result of considering the lax structure.  For now, however, we merely introduce the setting and stick to the notions compatible with the existing literature.

\subsection{Fibrations and Discrete Fibrations}

Here are recalled some specifics and notation on ordinary fibrations and discrete fibrations.  From these specifics, we can set out the desiderata guiding the results of the paper.  Throughout let $\C$ denote a small category.  Recall first the following standard definition.

\begin{define} \label{discr fibn defn} A \textbf{discrete fibration} over $\C$ is a functor $F\colon \F\to\C$ such that for each morphism $f\colon C\to FX$ with $ X\in \F$, there is a unique morphism $Y\to X$ of $\F$ above $f$.  A functor $E\colon \E\to\C$ is a discrete opfibration if $E^{op}$ is a discrete fibration.  A morphism of discrete fibrations $F\colon \F\to\C$ and $G\colon \G\to\C$ is a functor $H\colon \F\to\G$ such that $GH=F$ holds.  Let $\mathbf{DFib}(\C)$ denote the category of discrete fibrations over $\C$ and $\mathbf{DOpf}(\C)$ denote the category of discrete opfibrations over $\C$.
\end{define}

For each set-valued functor $F\colon \C^{op}\to\mathbf{Set}$, there is an associated category of elements, or ``Grothendieck semi-direct product," detailed for example in \S II.6 and \S III.7 of \cite{MacLane}, yielding a discrete fibration
\[ \Pi\colon \elt(E)\to \C.
\]
The source category has as objects pairs $(C,x)$ with $C\in \C_0$ and $x\in FC$ and as morphisms $(C,x) \to (D,y)$ those morphisms $f\colon C\to D$ of $\C$ with $x =Ff(y)$.

\begin{theo}[Representation Theorem I] \label{disc fib set-functor duality thm} The category of elements construction is one half of an equivalence of categories
\[ \mathbf{DFib}(\C)\simeq [\C^{op},\mathbf{Set}]. 
\]
between discrete fibrations and presheaves on $\C$.
\end{theo}
\begin{proof}  The pseudo-inverse sends a discrete fibration $F\colon \F\to \C$ to the functor $\C^{op} \to\mathbf{Set}$ whose action on $C\in \C_0$ is to take the fiber of $F$ above $C$.  Notice that these fibers must be discrete categories by the uniqueness assumption.  \end{proof}

\begin{remark}[Desiderata 1] \label{desiderata disc fib duality theorem} It is the result of Theorem \ref{disc fib set-functor duality thm} that this paper seeks to categorify.  Thus, needed is a category of elements construction and an appropriate pseudo-inverse construction.  The developments of \S 2 in \cite{Buckley} will be adapted for this purpose.  The main result is given as Theorem \ref{theorem duality for disc 2-fibrations}.  Needed is the half-step categorification of the idea of a discrete fibration, namely, that of a fibration.
\end{remark}

\begin{define}  A functor $F\colon \F \to \C$ is a \textbf{fibration} if for each $x\colon X\to FA$ there is an arrow $f\colon B\to A$ of $\F$ such that $Ff=x$ and having the property that whenever $h\colon C\to A$ makes a commutative triangle $xu=Fh$ as below there is a unique $F$-lift $\hat u\colon C\to B$ over $u$ making a commutative triangle in $\F$ as indicated in the following picture
$$
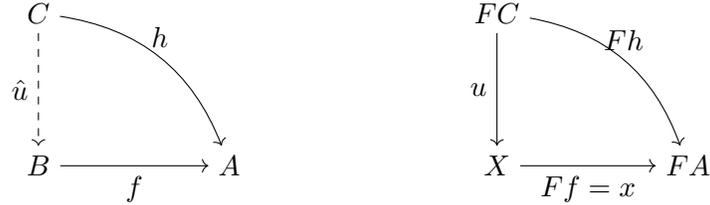
\begin{tikzpicture}
\node(1){$C$};
\node(2)[node distance=.8in, below of=1]{$B$};
\node(3)[node distance=1in, right of=2]{$A$};
\node(4)[node distance=2.4in, right of=1]{$FC$};
\node(5)[node distance=.8in, below of=4]{$X$};
\node(6)[node distance=1in, right of=5]{$FA$};
\draw[->,dashed](1) to node[left]{$\hat u$}(2);
\draw[->,bend left](1) to node[above]{$h$}(3);
\draw[->](2) to node[below]{$f$}(3);
\draw[->](4) to node[left]{$u$}(5);
\draw[->,bend left](4) to node[above]{$Fh$}(6);
\draw[->](5) to node[below]{$Ff=x$}(6);
\end{tikzpicture}$$
Such a morphism $f$ is \textbf{cartesian} over $x$.  A morphism of $\F$ is $F$-vertical if its image under $F$ is an identity.  The \textbf{fiber} of $F$ over an object $C\in \C$ is the subcategory of $\F$ of objects and vertical morphisms over $C$ via $F$.  A functor $E\colon \E\to\C$ is an \textbf{opfibration} if $E^{op}$ is a fibration; in this case the morphisms of $\E$ having the special lifting property are called ``opcartesian."
\end{define}

A cleavage $\phi$ for a fibration specifies a cartesian morphism in $\F$ for each such $f\colon X\to F A$ in $\C$.  Denote the chosen cartesian morphism by $\phi(f, A)\colon f^*A\to A$.  A fibration with a cleavage is said to be ``cloven.''  Notice that each discrete fibration is a cloven fibration.  An opfibration with chosen opcartesian morphisms is said to be ``opcloven" or to be equipped with an ``opcleavage."  A chosen opcartesian morphism above $f\colon FA\to X$ is denoted by $\phi(f,A)\colon A\to f_!A$.

\begin{remark} \label{cleavage not functorial}  In general a cleavage $\phi$ for a fibration $F\colon \F\to\C$ need not be functorial.  That is, given composable arrows $f\colon X\to Y$ and $g\colon Y \to FB$ of $\C$, there is a diagram of chosen cartesian arrows in $\F$ of the form
$$\begin{tikzpicture}
\node(1){$f^*g^*B$};
\node(2)[node distance=1.4in, right of=1]{$g^*B$};
\node(3)[node distance=.8in, below of=1]{$(gf)^*B$};
\node(4)[node distance=.8in, below of=2]{$B.$};
\node(5)[node distance=.7in, right of=1]{$$};
\node(6)[node distance=.4in, below of=5]{$=$};
\draw[->](1) to node [above]{$\phi(f,g^*B)$}(2);
\draw[->,dashed](1) to node [left]{$\cong$}(3);
\draw[->](2) to node [right]{$\phi(g,B)$}(4);
\draw[->](3) to node [below]{$\phi(gf,B)$}(4);
\end{tikzpicture}$$
The dashed arrow exists since a composition of cartesian morphisms is again cartesian.  It is an isomorphism by the uniqueness aspect of the definition.  But in general this isomorphism is not an identity.  When every such isomorphism is an identity, the fibration $F\colon \F\to\C$ is said to be split.  The difference between cloven and split fibrations over a base category is essentially the difference between category-valued pseudo-functors and 2-functors indexed by the base.
\end{remark}

In this paper, only split fibrations will be considered.  Thus, let $\mathbf{Fib}(\C)$ denote the 2-category of split fibrations over $\C$ with splitting-preserving functors as morphisms and natural transformations with vertical components as the 2-cells.  Dually, $\mathbf{Opf}(\C)$ is the 2-category of split opfibrations over $\C$ with appropriate morphisms and 2-cells.

Now, start with a 2-functor $F\colon \C^{op}\to\mathbf{Cat}$.  Denote the image of $f\colon C\to D$ in $\C$ by $f^*\colon ED\to EC$.  As in the discrete case, there is an associated fibration arising as a category of elements construction
\[ \Pi \colon \elt(F) \to \C.
\]
The source category has objects pairs $(C,X)$ with $X\in FC$ and as morphisms $(C,X) \to (D,Y)$ pairs $(f,u)$ where $f\colon C\to D$ and $u\colon X\to f^*Y$ is a morphism of $FC$.  Units and composition are well-known, but described for example in \S B1.3 of \cite{Elephant}.

\begin{theo}[Representation Theorem II] \label{opfibration psdfunctor correspondence} The category of elements construction is one-half of an equivalence of 2-categories
\[ \mathbf{Fib}(\C) \simeq [\C^{op},\mathbf{Cat}]
\]
between split fibrations over $\C$ and contravariant category-valued 2-functors on $\C$.
\end{theo}
\begin{proof}  Again the pseudo-inverse sends a split fibration $F$ to the 2-functor that associates to each $C\in \C_0$ the fiber of $F$ over it.  For more see Theorem B1.3.5 of \cite{Elephant} for example.  \end{proof}

\begin{remark}[Desiderata 2]\label{desiderata restrict equiv}  Every discrete fibration is a split fibration.  Additionally, the category of elements construction for a category-valued functor on $\C^{op}$ applied to one taking discrete categories as values reduces to the category of elements construction for presheaves.  Thus, the equivalence in Theorem \ref{opfibration psdfunctor correspondence} restricts to that of Theorem \ref{disc fib set-functor duality thm} as in  the commutative diagram
$$\begin{tikzpicture}
\node(1){$\mathbf{Fib}(\C)$};
\node(2)[node distance=1.2in, right of=1]{$[\C^{op},\mathbf{Cat}]$};
\node(3)[node distance=.7in, below of=1]{$\mathbf{DFib}(\C)$};
\node(4)[node distance=.7in, below of=2]{$[\C^{op},\mathbf{Set}].$};
\node(5)[node distance=.5in, right of=1]{$$};
\node(6)[node distance=.4in, below of=5]{$$};
\draw[->](1) to node [above]{$\simeq$}(2);
\draw[->](3) to node [left]{$\mathrm{incl}$}(1);
\draw[->](4) to node [right]{$\mathrm{incl}$}(2);
\draw[->](3) to node [below]{$\simeq$}(4);
\end{tikzpicture}$$
The notion of a 2-fibration has been studied by Hermida \cite{Hermida} and Buckley \cite{Buckley}.  In particular Buckley gives a categorification of Theorem \ref{opfibration psdfunctor correspondence} as a 3-equivalence between 2-fibrations and 2-category-valued functors indexed by the base 2-category.  Thus, we have the following requirement on development of subsequent theory: that any equivalence between discrete 2-fibrations and category-valued functors should be a restriction of that between 2-fibrations and 2-category-valued functors.
\end{remark}

\subsection{Overview of Contents}

The outline of the sections of the paper is as follows.  As a preliminary, \S \ref{section: 3-categories} sets out the required background in 2- and 3-dimensional categories.  The real point of the section is to summarize the universal properties of various 2-categorical comma objects in suitable 3-categorical settings.  The section also introduces the appropriate ambient 3-categorical setting for the work of the paper.  Following this, \S \ref{section: disc 2-fibs} introduces the main concept of the paper, namely, so-called ``discrete 2-fibrations" and proves the representation result in Theorem \ref{theorem duality for disc 2-fibrations} as a higher analogue of Theorem \ref{disc fib set-functor duality thm}.  In \S \ref{section: monadicity}, the monadicity results are presented.  That is, it is proved in Theorem \ref{MAIN THM2 Monadicity} that 2-categories of discrete 2-fibrations are 2-monadic over appropriate slices of 2-categories of categories.  A summary of the contents of various subsection appears at the opening of each section.

\section{The 3-Categorical Setting}

This section is mostly background on 2- and 2-dimensional categories, culminating in the last subsection, which gives our initial proposal for the 3-dimensional forum for studying the universal constructions relating to 2-fibrations.

\label{section: 3-categories}

\subsection{Higher Categories and Size Issues}

Throughout 2- and certain 3-dimensional categories play an essential role.  For the most part, conventions on 2-categories in the form required here are summarized in a previous paper, namely, \cite{Me2}.  Some notational subtleties discussed below will differ.  In any case, roughly, the required material on 2-categories is Chapters I,2 and I,3 of Gray's \cite{GrayFormalCats}.  Other references are Chapter 7 of \cite{Handbook1}, Chapter B1 of \cite{Elephant}, and the paper of \cite{KS}.  Here we deal with (1) some conventions concerning foundational set-theoretic issues arising in examples; and (2) with stage-setting with particular 2- and 3-dimensional categories that will recur throughout the work.

\subsubsection{Universe Assumption}

Some care concerning matters of size is required in several of the examples later in the paper.  Firstly, by the term `set' is meant a set or collection, naively construed.  The most important assumption for the paper is that there is a sufficient number of so-called ``universes."  By the term is meant a universe in the sense of \cite{SGA4}, that is, a set of sets that is transitive under membership and closed under pairing, powersets, and unions indexed by elements of the universe.  The \textbf{universe assumption} is that every set belongs to some universe.

Throughout $\mathbf{Set}$ denotes a fixed category of distinguished ``small" sets, that is, sets that are members of some given, fixed universe.  Thus, a given arbitrary set is potentially large, i.e., not in $\mathbf{Set}$.  By the axiom of universes, there is a 2-category $\mathbf{Cat}$ of categories containing $\mathbf{Set}$ as a member and admitting an inclusion $\mathbf{Set}\to\mathbf{Cat}$ viewing a set as a (locally) discrete category.  Such a 2-category $\mathbf{Cat}$ can be constructed from any 2-category of $\mathbf{Set}$-small categories by expanding its set of elements by the universe axiom to include $\mathbf{Set}$ and thus all small sets by transitivity.  Similarly, there is a 3-category $\mathbf{2Cat}$ of 2-categories containing $\mathbf{Cat}$ as a member and admitting an inclusion $\mathbf{Cat}\to \mathbf{2Cat}$ viewing a category as a locally discrete 2-category.  Again this is constructed by the universe axiom.  Whenever it is important to distinguish between, for example, a 2-category of $\mathbf{Set}$-small categories on the one hand and a 2-category of categories containing $\mathbf{Set}$ as a member, write $\mathbf{Cat}$ for the former and $\mathbf{CAT}$ for the latter.  Adopt a similar convention for 3-categories of 2-categories.

\subsubsection{Specific 2- and 3-Categories}

Throughout we shall assume the basic machinery of enriched categories as in the early chapters of \cite{Enriched}.  Thus, a 2-category $\K$ is a $|\mathbf{Cat}|$-enriched category where $|\mathbf{Cat}|$ is the 1-category of (small) categories and functors.  In general the notation `$|\cdot |$' as in `$|\cat|$' will mean the $n-1$ dimensional structure obtained by discarding the top-dimensional cells of the structure enclosed in the `$|\cdot|$'.  So, for example, `$||\mathbf{2Cat}||$' denotes the 1-category of 2-categories and 2-functors obtained from the 3-category $\mathbf{2Cat}$ by forgetting the 3-cells and the 2-cells.  Consider the following further examples.

\begin{example}  Let $\mathbf{DFib}$ denote the 2-category of discrete fibrations $F\colon \F\to\C$, whose morphisms $F\to G$ are pairs of functors $(H,K)$ making commutative squares
$$\begin{tikzpicture}
\node(1){$\F$};
\node(2)[node distance=1in, right of=1]{$\G$};
\node(3)[node distance=.7in, below of=1]{$\C$};
\node(4)[node distance=.7in, below of=2]{$\B$};
\node(5)[node distance=.5in, right of=1]{$$};
\node(6)[node distance=.4in, below of=5]{$$};
\draw[->](1) to node [above]{$H$}(2);
\draw[->](1) to node [left]{$F$}(3);
\draw[->](2) to node [right]{$G$}(4);
\draw[->](3) to node [below]{$K$}(4);
\end{tikzpicture}$$
and whose 2-cells $(H,K) \Rightarrow (L,M)$ are pairs of natural transformations $(\alpha,\beta)$ with $\alpha\colon H\Rightarrow L$ and $\beta\colon K\Rightarrow M$ satisfying the compatibility condition $G\ast \alpha = \beta\ast F$.  Notice that there is thus a projection 2-functor 
\begin{equation}
\cod\colon \mathbf{DFib}\to\mathbf{Cat}
\end{equation}
taking a discrete fibration to its codomain category and extended suitably to morphisms and 2-cells.  Identify the fiber of $\cod$ over $\C\in \mathbf{Cat}_0$ with the category $\mathbf{DFib}(\C)$ of Definition \ref{discr fibn defn}.  The development for a 2-category $\mathbf{DOpf}$ of discrete opfibrations is analogous but dual.
\end{example}

\begin{example}[``Lax Arrow Category," or 2-Category of Squares] \label{2cat of squares example} Given a 2-category $\mathfrak C$, the 2-category of squares $\mathbf{Sq}(\mathfrak C)$ has as objects morphisms $f\colon A\to B$ of $\mathfrak C$; as arrows $f\to g$ those 2-cells 
$$\begin{tikzpicture}
\node(1){$A$};
\node(2)[node distance=1in, right of=1]{$C$};
\node(3)[node distance=.7in, below of=1]{$B$};
\node(4)[node distance=.7in, below of=2]{$D$};
\node(5)[node distance=.5in, right of=1]{$$};
\node(6)[node distance=.3in, below of=5]{$\alpha$};
\node(7)[node distance=.43in, below of=5]{$\Rightarrow$};
\draw[->](1) to node [above]{$h$}(2);
\draw[->](1) to node [left]{$f$}(3);
\draw[->](2) to node [right]{$g$}(4);
\draw[->](3) to node [below]{$k$}(4);
\end{tikzpicture}$$
of $\mathfrak C$; and finally as 2-cells pairs of 2-cells $(\gamma,\delta)$ for which there is an equality of composite 2-cells as in the diagram
$$\begin{tikzpicture}
\node(1){$A$};
\node(2)[node distance=1in, right of=1]{$C$};
\node(3)[node distance=.8in, below of=1]{$B$};
\node(4)[node distance=.8in, below of=2]{$D$};
\node(5)[node distance=.5in, right of=1]{$$};
\node(6)[node distance=.34in, below of=5]{$\alpha$};
\node(7)[node distance=.47in, below of=5]{$\Rightarrow$};
\node(8)[node distance=.25in, above of=5]{$\gamma \Uparrow$};
\node(9)[node distance=2.5in, right of=8]{$A$};
\node(10)[node distance=1in, right of=9]{$C$};
\node(11)[node distance=.8in, below of=9]{$B$};
\node(12)[node distance=.8in, below of=10]{$D.$};
\node(13)[node distance=.5in, right of=9]{$$};
\node(14)[node distance=.32in, below of=13]{$\beta$};
\node(15)[node distance=.45in, below of=13]{$\Rightarrow$};
\node(16)[node distance=1in, below of=13]{$\Uparrow\delta$};
\node(17)[node distance=1.5in, right of=7]{$=$};
\draw[->](1) to node [below]{$h$}(2);
\draw[->](1) to node [left]{$f$}(3);
\draw[->](2) to node [right]{$g$}(4);
\draw[->](3) to node [below]{$k$}(4);
\draw[->,bend left=100](1) to node [above]{$m$}(2);
\draw[->](9) to node [above]{$m$}(10);
\draw[->,bend right=100](11) to node [below]{$k$}(12);
\draw[->](9) to node [left]{$f$}(11);
\draw[->](10) to node [right]{$g$}(12);
\draw[->](11) to node [above]{$n$}(12);
\end{tikzpicture}$$
The domain of the arrow $\alpha$ above is thus $f$, while the codomain is $g$.  Denote ``source" and ``target" 2-functors by $\src \colon \mathbf{Sq}(\mathfrak C)\to\mathfrak C$ and $\tgt\colon \mathbf{Sq}(\mathfrak C)\to\mathfrak C$, respectively.  The source 2-functor takes an object (i.e. a morphism) to its domain, takes a morphism (i.e. a square as above) to the morphism $h$, and a 2-cell pair $(\gamma, \delta)$ to $\gamma$.  Target is defined analogously.  The 2-category $\mathbf{Sq}(\mathfrak C)$ is basically the 2-comma category $1_{\mathfrak C}/1_{\mathfrak C}$ in \S I,2.5 of \cite{GrayFormalCats}.  It is also a fragment of the double category of quintets associated to a 2-category of \cite{Ehresmann}.  Neither of these names will be used, however.
\end{example}

\begin{example} \label{algebra for monad in 2cat eg} Let $\K$ denote a 2-category and $t\colon B\to B$ a monad in $\K$ as in \S 3.1 of \cite{KS}.  Define a 2-category $t\mathbf{Alg}$ of $t$-algebras in $\K$.  The objects are $t$-algebras $(s,\nu)$ where $s\colon A\to B$ and $\nu\colon ts\Rightarrow s$ satisfying (3.2) of the reference.  A morphism $(s,\nu)\to (r,\lambda)$ is a pair $(g,\sigma)$ where $g\colon A\to C$ is a $\K$-morphism and $\sigma\colon s\Rightarrow rg$ is a 2-cells satisfying the equation
\[ (\lambda\ast g)( t\ast \sigma) = \sigma \nu,
\]
basically the appropriate adaptation of (3.3) of the reference allowing the domain of the $t$-algebras to vary.  The equation says that $\sigma$ is a morphism of the $t$-algebras $s$ and $gr$.  A 2-cell $(g,\sigma)\Rightarrow (h,\tau)$ is one $\alpha \colon g\Rightarrow h$ such that $(r\ast \alpha)\sigma = \tau$ holds.  Notice that by fixing the domain object in the $t$-algebra $s\colon A\to B$, this defines a category $t\mathbf{Alg}(A)$ of algebras with domain $A$.  Since algebra structure is preserved by precomposition with any morphism and homomorphisms are preserved by precomposition with arbitrary 2-cells, these categories yield a functor $t\mathbf{Alg}(-)\colon \mathfrak K^{op}\to\mathbf{Cat}$.
\end{example}

\begin{example} \label{lax slice example} For any 2-category $\mathfrak C$, the \textbf{lax slice} over an object $C\in\mathfrak C_0$ consists of arrows $f\colon B\to C$ as its objects, with morphisms $f\to g$ those pairs consisting of $h\colon B\to D$ and a 2-cell $\alpha\colon f\Rightarrow gh$, and finally with 2-cells those $\theta\colon h\Rightarrow k$ satisfying the compatibility condition 
$$\begin{tikzpicture}
\node(1){$B$};
\node(2)[node distance=.5in, right of=1]{$$};
\node(3)[node distance=.5in, right of=2]{$D$};
\node(4)[node distance=1in, below of=2]{$C$};
\node(5)[node distance=.2in, above of=2]{$\Uparrow\theta$};
\node(6)[node distance=.35in, below of=2]{$\alpha$};
\node(7)[node distance=.5in, below of=2]{$\Rightarrow$};
\node(8)[node distance=3in, right of=1]{$B$};
\node(9)[node distance=3in, right of=3]{$D$};
\node(10)[node distance=3in, right of=4]{$C$};
\node(11)[node distance=3in, right of=6]{$\beta$};
\node(12)[node distance=3in, right of=7]{$\Rightarrow$};
\node(13)[node distance=1.5in, right of=6]{$=$};
\draw[->,bend left=100](1) to node [above]{$k$}(3);
\draw[->](1) to node [below]{$h$}(3);
\draw[->](1) to node [left]{$f$}(4);
\draw[->](3) to node [right]{$g$}(4);
\draw[->](8) to node [above]{$k$}(9);
\draw[->](8) to node [left]{$f$}(10);
\draw[->](9) to node [right]{$g$}(10);
\end{tikzpicture}$$
The lax slice will be denoted by `$\mathfrak C/C$'.  Of course, there are various candidates for 2-categorical slices.  That is, the morphisms could alternatively have the 2-cell $\alpha$ be an isomorphism or an equality.  The latter would perhaps be called the ``strict slice" 2-category.  Whereas the lax slice will occur in a number of examples, the strict slice will be used throughout \S 4 in considering monadicity.  The two notions will not be used in the same context, so the same notation `$\mathfrak C/C$' will stand for each and the meaning will be explicitly stated whenever it arises.
\end{example}

A final generic example we now treat in more detail.

\begin{construction}[Lax Comma Category; Cf. \S I,2.5 of \cite{GrayFormalCats}] \label{lax comma cat construction} The \textbf{lax comma category} of a 2-functor $F\colon \mathfrak A\to\mathfrak B$ over another $G\colon\mathfrak C\to\mathfrak B$, denoted by $F/G$, has as objects triples $(A,f,C)$ with $f\colon FA\to GC$ an arrow of $\mathfrak B$; and with morphisms $(A,f,C)\to (B,g,D)$ those triples $(h,k,\alpha)$ making a cell in $\mathfrak B$ of the form
$$\begin{tikzpicture}
\node(1){$FA$};
\node(2)[node distance=1in, right of=1]{$FB$};
\node(3)[node distance=.8in, below of=1]{$GC$};
\node(4)[node distance=.8in, below of=2]{$GD$};
\node(5)[node distance=.5in, right of=1]{$$};
\node(6)[node distance=.34in, below of=5]{$\alpha$};
\node(7)[node distance=.47in, below of=5]{$\Rightarrow$};
\draw[->](1) to node [above]{$Fh$}(2);
\draw[->](1) to node [left]{$f$}(3);
\draw[->](2) to node [right]{$g$}(4);
\draw[->](3) to node [below]{$Gk$}(4);
\end{tikzpicture}$$
and whose 2-cells $(h,k,\alpha) \Rightarrow (m,n,\beta)$ are pairs $(\gamma,\delta)$ satisfying the 2-cell ``cylinder" equality
$$\begin{tikzpicture}
\node(1){$FA$};
\node(2)[node distance=1in, right of=1]{$FB$};
\node(3)[node distance=.8in, below of=1]{$GC$};
\node(4)[node distance=.8in, below of=2]{$GD$};
\node(5)[node distance=.5in, right of=1]{$$};
\node(6)[node distance=.34in, below of=5]{$\alpha$};
\node(7)[node distance=.47in, below of=5]{$\Rightarrow$};
\node(8)[node distance=.25in, above of=5]{$\gamma \Uparrow$};
\node(9)[node distance=2.5in, right of=8]{$FA$};
\node(10)[node distance=1in, right of=9]{$FB$};
\node(11)[node distance=.8in, below of=9]{$GC$};
\node(12)[node distance=.8in, below of=10]{$GD.$};
\node(13)[node distance=.5in, right of=9]{$$};
\node(14)[node distance=.32in, below of=13]{$\beta$};
\node(15)[node distance=.45in, below of=13]{$\Rightarrow$};
\node(16)[node distance=1in, below of=13]{$\Uparrow\delta$};
\node(17)[node distance=1.5in, right of=7]{$=$};
\draw[->](1) to node [below]{$Fh$}(2);
\draw[->](1) to node [left]{$f$}(3);
\draw[->](2) to node [right]{$g$}(4);
\draw[->](3) to node [below]{$Gk$}(4);
\draw[->,bend left=100](1) to node [above]{$Fm$}(2);
\draw[->](9) to node [above]{$Fm$}(10);
\draw[->,bend right=100](11) to node [below]{$Gk$}(12);
\draw[->](9) to node [left]{$f$}(11);
\draw[->](10) to node [right]{$g$}(12);
\draw[->](11) to node [above]{$Gn$}(12);
\end{tikzpicture}$$
Somewhat paradoxically take the target of the square $\alpha$ in the first display as the arrow $k$ and the source as $h$.  Ordinary composition of squares is by pasting those with matching domain and codomain; a further ``horizontal" composition of squares is by pasting squares with matching source and target.  Now, vertical composition of 2-cells is given by vertical  composition of 2-cells in $\mathfrak B$ whereas horizontal composition is given by horizontal composition in $\mathfrak B$.  A further composition of 2-cells is given by stacking cylinders.  Source and target again define 2-functors $\src \colon F/G\to \mathfrak A$ and $\tgt\colon F/G\to \mathfrak B$.  Notice as a special case with $F=G=1_{\mathfrak B}$, the 2-category of squares $\mathbf{Sq}(\mathfrak B)$ of Example \ref{2cat of squares example} is recovered.  In the special case that only commutative squares in $\mathfrak B$ are taken in the definition of the morphisms of $F/G$, we have the \textbf{2-comma category}, which in general is given by the same notation, but will be accompanied by clarification in context whenever it is used.  Notice that the 2-comma category is accompanied by a canonically given 2-natural transformation $\beta\colon F\circ \src\Rightarrow G\circ\tgt$ defined on a component $(A,f,C)$ by $\beta_{(A,f,C)}=\beta_f$.  This is universal in the following sense.
\end{construction}

\begin{prop}[Universal Property of 2-Comma Category; Cf. \S 1 p.108 of \cite{StreetFibrations}] \label{up of 2-comma} Given 2-functors $F\colon \mathfrak A\to \mathfrak B$ and $G\colon \mathfrak C\to\mathfrak B$, the 2-comma category $F/G$ of Construction \ref{lax comma cat construction} is 1-, 2- and 3-dimensionally universal in the following sense.
\begin{enumerate}
\item  Given 2-functors $H\colon \mathfrak D\to\mathfrak A$ and $K\colon \mathfrak D\to \mathfrak C$ and any 2-cell $\theta\colon FH\Rightarrow GK$, there is a unique 2-functor $U\colon \mathfrak D\to F/G$ such that $\theta = \lambda \ast U$.
\item Given 2-natural transformations $\xi$ and $\zeta$ satisfying the equality of 2-cells in the diagram
$$\begin{tikzpicture}
\node(1){$\mathfrak D$};
\node(2)[node distance=.6in, right of=1]{$\Uparrow \xi$};
\node(3)[node distance=.6in, right of=2]{$\mathfrak C$};
\node(4)[node distance=.5in, above of=2]{$F/G$};
\node(5)[node distance=.5in, below of=2]{$F/G$};
\node(6)[node distance=.6in, right of=5]{$\Uparrow\lambda$};
\node(7)[node distance=.6in, right of=6]{$\mathfrak B$};
\node(8)[node distance=.5in, below of=6]{$\mathfrak A$};
\node(9)[node distance=1in, right of=7]{$\mathfrak D$};
\node(10)[node distance=.6in, right of=9]{$\Uparrow\zeta$};
\node(11)[node distance=.6in, right of=10]{$\mathfrak A$};
\node(12)[node distance=.5in, above of=10]{$F/G$};
\node(13)[node distance=.5in, below of=10]{$F/G$};
\node(14)[node distance=.6in, right of=12]{$\Uparrow\lambda$};
\node(15)[node distance=.6in, right of=14]{$\mathfrak A$};
\node(16)[node distance=.5in, above of=14]{$\mathfrak C$};
\node(17)[node distance=.5in, right of=7]{$=$};
\draw[->](1) to node [above]{$V\;\;\;$}(4);
\draw[->](1) to node [below]{$U\;\;\;$}(5);
\draw[->](4) to node [above]{$\;\;\;\tgt$}(3);
\draw[->](5) to node [below]{$\;\;\;\tgt$}(3);
\draw[->](5) to node [below]{$\src\;\;\;$}(8);
\draw[->](3) to node [above]{$\;\;\;G$}(7);
\draw[->](8) to node [below]{$\;\;\;F$}(7);
\draw[->](9) to node [above]{$V\;\;\;$}(12);
\draw[->](9) to node [below]{$U\;\;\;$}(13);
\draw[->](12) to node [above]{$\;\,\src$}(11);
\draw[->](13) to node [below]{$\;\;\;\src$}(11);
\draw[->](12) to node [above]{$\tgt\;\;\;\;$}(16);
\draw[->](11) to node [below]{$\;\;\;F$}(15);
\draw[->](16) to node [above]{$\;\;\;G$}(15);
\end{tikzpicture}$$
there is a unique 2-natural transformation $\omega\colon U\Rightarrow V$ such that the equations $\src\ast \omega = \xi$ and $\tgt\ast\omega = \zeta$ each hold.
\item Given 2-natural transformations $\omega \colon U \Rightarrow V$ and $\chi\colon U \Rightarrow V$ between $U,V\colon \mathfrak D\rightrightarrows F/G$ as above with modfications $m\colon \src\ast \omega \Rrightarrow \src\ast \chi$ and $n\colon \tgt\ast \omega \Rrightarrow \tgt\ast \chi$ satisfying 
\begin{equation} \label{compat} (\lambda\ast V)Fm = Gn(\lambda\ast U)
\end{equation}
there is a unique modification $l\colon \omega\Rrightarrow \chi$ satisfying $\tgt\ast l = n$ and $\src\ast l = m$.
\end{enumerate}
These properties characterize $F/G$ up to isomorphism in $\mathbf{2Cat}$.
\end{prop}
\begin{proof}  The proofs of the first two universality conditions are the same as for the ordinary 1-comma category in $\mathbf{Cat}$, since the underlying 1-category of the 2-comma category is essentially the 1-comma category of the underlying functors.  Thus, we have only to prove the third condition.  But the compatibility condition \ref{compat} just means that for each $D\in\mathfrak D_0$ there is an equality of composite 2-cells in $\mathfrak B$ of the form
$$\begin{tikzpicture}
\node(1){$\cdot$};
\node(2)[node distance=1.2in, right of=1]{$\cdot$};
\node(3)[node distance=.8in, below of=1]{$\cdot$};
\node(4)[node distance=.8in, below of=2]{$\cdot$};
\node(5)[node distance=.6in, right of=1]{$\Downarrow Fm_D$};
\node(6)[node distance=.9in, below of=5]{$$};
\node(7)[node distance=.5in, below of=5]{$$};
\node(8)[node distance=1.6in, right of=7]{$=$};
\node(9)[node distance=2in, right of=2]{$\cdot$};
\node(10)[node distance=1.2in, right of=9]{$\cdot$};
\node(11)[node distance=.8in, below of=9]{$\cdot$};
\node(12)[node distance=.8in, below of=10]{$\cdot$};
\node(13)[node distance=.6in, right of=9]{$$};
\node(14)[node distance=.2in, below of=13]{$$};
\node(15)[node distance=.6in, below of=14]{$\Downarrow Gn_D$};
\draw[->](2) to node [right]{$VD$}(4);
\draw[->](1) to node [left]{$UD$}(3);
\draw[->,bend left](1) to node [above]{$\src\,\omega_D$}(2);
\draw[->,bend right](1) to node [below]{$\src\,\chi_D$}(2);
\draw[->,bend right](3) to node [below]{$\tgt\,\chi_D$}(4);
\draw[->,bend left](9) to node [above]{$\src\,\omega_D$}(10);
\draw[->](9) to node [left]{$UD$}(11);
\draw[->,bend left](11) to node [above]{$\tgt\,\omega_D$}(12);
\draw[->,bend right](11) to node [below]{$\tgt\,\chi_D$}(12);
\draw[->](10) to node [right]{$VD$}(12);
\end{tikzpicture}$$
But this is plainly the form of a 2-cell in $F/G$ and thus gives the $D$-component of the purported modification $l$.  Compatibility follows by naturality and the modification condition.  \end{proof}

\begin{remark}[Non-Elementary Statement of Universal Property] \label{remark on enriched cotensor} Recall from \S 3.7 of \cite{Enriched} that the cotensor of an object $C$ in a $\mathscr V$-category $\mathcal B$ by an object $X\in \mathscr V_0$ is an object $X\pitchfork C$ of $\mathcal B$ for which there is a $\mathscr V$-natural isomorphism with counit
\begin{equation} \label{enriched cotensor defn} \mathcal B(B, X\pitchfork C)\cong [X,\mathcal B(B,C)] \qquad X\to\mathcal B(X\pitchfork C,C)
\end{equation}
for any $B\in \mathcal B_0$.  Viewing $\mathbf{Cat}$ as $\mathbf{Set}$-enriched, for any small category $\C$, the usual arrow category $\mathbf{Arr}(\C) = \mathbf{Cat}(\mathbf 2,\C)$ is the cotensor of $\C$ with the ordinal category $\mathbf 2$.  Similarly, the usual arrow 2-category $\mathbf{2Cat}(\mathbf 2,\mathfrak C)$ is the cotensor of $\mathfrak C$ with $\mathbf 2$ in $\mathbf{2Cat}$.  Now, the universal property of the 2-comma category says, essentially, that $1_{\mathfrak B}/1_{\mathfrak B}$, the comma category of $1_{\mathfrak B}$ with itself, is the cotensor of $\mathfrak B$ with $\mathbf 2$ in the sense that there is an isomorphism
\[ \mathbf{2Cat}(\mathfrak A,1_{\mathfrak B}/1_{\mathfrak B})\cong |\cat|(\mathbf 2,\mathbf{2Cat}(\mathfrak A,\mathfrak B))
\]
induced by composition with the canonical 2-natural transformation $\beta$ of Construction \ref{lax comma cat construction}.  This is probably easiest to see from the definitions using the fact that the hom-category on the right is isomorphic to the arrow 2-category $\mathbf{2Cat}(\mathfrak A,\mathfrak B)^{\mathbf 2}$ as presented in Construction \ref{lax comma cat construction}.  Note in particular that $1_{\mathfrak B}/1_{\mathfrak B}$ is isomorphic to the 2-arrow category $\mathfrak B^{\mathbf 2}$.
\end{remark}

More examples of 2-categories will be known to the reader and others will be introduced in the course of the paper.  For now, $||\mathbf{2Cat}||$ denotes the ordinary category of 2-categories and 2-natural transformations.  It is strict cartesian monoidal.

\begin{define} \label{3cat defn} A \textbf{3-category} is a $||\mathbf{2Cat}||$-enriched category.
\end{define}

\begin{remark}  This means that a 3-category $\mathcal A$ is a set of objects $\mathcal A_0$ together with ``hom 2-categories" $\mathcal A(A,B)$ for any $A,B\in \mathcal A_0$ together with appropriate composition and identity 2-functors satisfying the usual associativity and unit diagrams.
\end{remark}

\begin{example}  2-categories, 2-functors, 2-natural transformations, and their modifications comprise the 3-category, $\mathbf{2Cat}$.
\end{example}

\begin{example}  For any 2-category $\mathfrak A$, 2-functors $\mathfrak A^{op}\to\mathbf{2Cat}$, 2-natural transformations, modifications and ``perturbations" between them (as in \cite{CohTricats}) form a 3-category denoted by $[\mathfrak A^{op},\mathbf{2Cat}]$.
\end{example}

\subsection{Lax Natural Transformations}

The notions of 2-functor, 2-natural transformation, and modification are well-known and together comprise the data for the 3-categorical structure giving $\mathbf{2Cat}$.  Perhaps less well-known is the idea of a lax natural transformation, which for completeness is recalled here.

\begin{define}[Cf. \S I,2.4 of \cite{GrayFormalCats}] \label{LAXNATTRANSF} A \textbf{lax-natural transformation} $\alpha\colon F \Rightarrow G$ of 2-functors $F,G\colon \mathfrak K\rightrightarrows \mathfrak L$ consists of a family of arrows $\alpha_A\colon FA\to GA$ of $\mathfrak L$ indexed over the objects $A\in\mathfrak K_0$ together with, for each arrow $f$ of $\mathfrak K$, a distinguished 2-cell
$$\begin{tikzpicture}
\node(1){$FA$};
\node(2)[node distance=1in, right of=1]{$FB$};
\node(3)[node distance=.8in, below of=1]{$GA$};
\node(4)[node distance=.8in, below of=2]{$GB$};
\node(5)[node distance=.5in, right of=1]{$$};
\node(6)[node distance=.34in, below of=5]{$\alpha_f$};
\node(7)[node distance=.46in, below of=5]{$\Rightarrow$};
\draw[->](1) to node [above]{$Ff$}(2);
\draw[->](1) to node [left]{$\alpha_A$}(3);
\draw[->](2) to node [right]{$\alpha_B$}(4);
\draw[->](3) to node [below]{$Gf$}(4);
\end{tikzpicture}$$ 
satisfying the following two compatibility conditions.  
\begin{enumerate}
\item For any composable arrows $f$ and $g$ of $\mathfrak K$, there is an equality of 2-cells
$$\begin{tikzpicture}
\node(1){$FA$};
\node(2)[node distance=1in, right of=1]{$FB$};
\node(3)[node distance=.8in, below of=1]{$GA$};
\node(4)[node distance=.8in, below of=2]{$GB$};
\node(5)[node distance=.5in, right of=1]{$$};
\node(6)[node distance=.34in, below of=5]{$\alpha_f$};
\node(7)[node distance=.46in, below of=5]{$\Rightarrow$};
\node(8)[node distance=1in, right of=2]{$FC$};
\node(9)[node distance=.8in, below of=8]{$GC$};
\node(10)[node distance=1in, right of=6]{$\alpha_g$};
\node(11)[node distance=1in, right of=7]{$\Rightarrow$};
\node(12)[node distance=2in, right of=8]{$FA$};
\node(13)[node distance=1in, right of=12]{$FB$};
\node(14)[node distance=.8in, below of=12]{$GA$};
\node(15)[node distance=.8in, below of=13]{$GB$};
\node(16)[node distance=.5in, right of=12]{$$};
\node(17)[node distance=.34in, below of=16]{$\alpha_{gf}$};
\node(18)[node distance=.46in, below of=16]{$\Rightarrow$};
\node(19)[node distance=1.5in, right of=11]{$=$};
\draw[->](1) to node [above]{$Ff$}(2);
\draw[->](1) to node [left]{$\alpha_A$}(3);
\draw[->](2) to node [right]{$\alpha_B$}(4);
\draw[->](3) to node [below]{$Gf$}(4);
\draw[->](2) to node [above]{$Fg$}(8);
\draw[->](4) to node [below]{$Gg$}(9);
\draw[->](8) to node [right]{$\alpha_C$}(9);
\draw[->](12) to node [above]{$Ff$}(13);
\draw[->](12) to node [left]{$\alpha_A$}(14);
\draw[->](13) to node [right]{$\alpha_B$}(15);
\draw[->](14) to node [below]{$Gf$}(15);
\end{tikzpicture}$$ 
\item For any 2-cell $\theta \colon f\Rightarrow g$ of $\mathfrak K$, there is an equality of 2-cells as depicted in the diagram
$$\begin{tikzpicture}
\node(1){$FA$};
\node(2)[node distance=1.2in, right of=1]{$FB$};
\node(3)[node distance=.8in, below of=1]{$GA$};
\node(4)[node distance=.8in, below of=2]{$GB$};
\node(5)[node distance=.6in, right of=1]{$\Uparrow F\theta$};
\node(6)[node distance=.74in, below of=5]{$\Rightarrow$};
\node(7)[node distance=.63in, below of=5]{$\alpha_f$};
\node(8)[node distance=1.6in, right of=7]{$=$};
\node(9)[node distance=2in, right of=2]{$FA$};
\node(10)[node distance=1.2in, right of=9]{$FB$};
\node(11)[node distance=.8in, below of=9]{$FA$};
\node(12)[node distance=.8in, below of=10]{$FB$};
\node(13)[node distance=.6in, right of=9]{$\alpha_g$};
\node(14)[node distance=.11in, below of=13]{$\Rightarrow$};
\node(15)[node distance=.8in, below of=13]{$\Uparrow G\theta$};
\draw[->](2) to node [right]{$\alpha_B$}(4);
\draw[->](1) to node [left]{$\alpha_A$}(3);
\draw[->,bend left](1) to node [above]{$Fg$}(2);
\draw[->,bend right](1) to node [below]{$Ff$}(2);
\draw[->,bend right](3) to node [below]{$Gf$}(4);
\draw[->,bend left](9) to node [above]{$Fg$}(10);
\draw[->](9) to node [left]{$\alpha_A$}(11);
\draw[->,bend left](11) to node [above]{$Gg$}(12);
\draw[->,bend right](11) to node [below]{$Gf$}(12);
\draw[->](10) to node [right]{$\alpha_B$}(12);
\end{tikzpicture}$$
\end{enumerate}
A lax-natural transformation is \textbf{pseudo natural} if the cells $\alpha_f$ are invertible.  If they are identities, the transformation is \textbf{2-natural}.
\end{define}

\begin{example} \label{universal lax nat transf for cotensor} Between the source and target 2-functors
\[ \src,\tgt\colon \mathbf{Sq}(\mathfrak B) \rightrightarrows\mathfrak B
\]
from the 2-category of squares from Example \ref{2cat of squares example}, there is a lax natural transformation $\beta\colon \src \Rightarrow \tgt$ given on a component $f\colon B\to C$ by $\beta_f:=f$.  If $\alpha\colon f\to g$ is an arrow of $\mathbf{Sq}(\mathfrak B)$ with source $h$ and target $k$ as in the first display of Example \ref{2cat of squares example}, then take the coherence cell in $\mathfrak B$ to be $\alpha$ itself.  Then the axioms for lax naturality are satisfied by the definitions of composition for 1- and 2-cells in $\mathbf{Sq}(\mathfrak B)$.  Of course this is a special case of a more general situation for the lax comma category of a 2-functor $F$ over another $G$.  That is, $F/G$ and its projections $\src\colon F/G\to \mathfrak A$ and $\tgt\colon F/G\to \mathfrak C$ admit a lax natural transformation $\beta\colon F\circ \src \Rightarrow G\circ \tgt$ defined on components in an analgous way.  This will make $F/G$ into a lax comma object in $\mathbf{Lax}$, as seen in the development below.
\end{example}

\begin{lemma} \label{comma 1-dim universal lemma} The lax natural transformation $\beta\colon F\circ \src \Rightarrow G\circ \tgt$ of Example \ref{universal lax nat transf for cotensor} is 1-dimensionally universal amid lax natural transformations in the sense that for any lax natural transformation 
\[\alpha\colon FH\Rightarrow GK
\]
there is a unique 2-functor depicted as the dashed arrow in the diagram
$$\begin{tikzpicture}
\node(1){$F/G$};
\node(2)[node distance=1in, right of=1]{$\mathfrak C$};
\node(3)[node distance=.8in, below of=1]{$\mathfrak A$};
\node(4)[node distance=.8in, below of=2]{$\mathfrak B$};
\node(5)[node distance=.5in, right of=1]{$$};
\node(6)[node distance=.35in, below of=5]{$\beta$};
\node(7)[node distance=.6in, left of=1]{$$};
\node(8)[node distance=.5in, above of=7]{$\mathfrak D$};
\node(9)[node distance=.5in, below of=5]{$\Rightarrow$};
\draw[->](1) to node [above]{$\tgt$}(2);
\draw[->](1) to node [left]{$\src$}(3);
\draw[->](2) to node [right]{$G$}(4);
\draw[->](3) to node [below]{$F$}(4);
\draw[->,bend right](8) to node [left]{$H$}(3);
\draw[->,bend left](8) to node [above]{$K$}(2);
\draw[->,dashed](8) to node [above]{$\;U$}(1);
\end{tikzpicture}$$
making two commutative triangles and satisfying $\beta\ast U =\alpha$.
\end{lemma}
\begin{proof}  Define $\mathfrak D\to F/G$ in the following way.  On objects take $C\mapsto \alpha_C\colon FC\to GC$.  On a morphism $f\colon C\to D$, take $Uf$ to be the square
$$\begin{tikzpicture}
\node(1){$FC$};
\node(2)[node distance=1in, right of=1]{$FD$};
\node(3)[node distance=.8in, below of=1]{$GC$};
\node(4)[node distance=.8in, below of=2]{$GD$};
\node(5)[node distance=.5in, right of=1]{$$};
\node(6)[node distance=.34in, below of=5]{$\alpha_f$};
\node(7)[node distance=.46in, below of=5]{$\Rightarrow$};
\draw[->](1) to node [above]{$Ff$}(2);
\draw[->](1) to node [left]{$\alpha_A$}(3);
\draw[->](2) to node [right]{$\alpha_B$}(4);
\draw[->](3) to node [below]{$Gf$}(4);
\end{tikzpicture}$$ 
coming with $\alpha$.  The assignment on a 2-cell $\theta\colon f\Rightarrow g$ of $\mathfrak D$ is then evident and is well-defined by the compatibility condition for $\alpha$.  That $U$ is functor follows also by the compatibility conditions for $\alpha$.  It is clear that by construction $U$ makes the two traingles commute and satisfies the equation $\beta\ast U=\alpha$.  \end{proof}

\begin{remark}  In fact $F/G$ is furthermore 2- and 3-dimensionally universal in a precise sense recalled in the next subsection.  For now, we return to generalities on lax natural transformations.
\end{remark}

\begin{lemma}  For any 2-categories $\mathfrak A$ and $\mathfrak B$, the 2-functors between them, with lax natural transformations and modifications, are the data of a 2-category, denoted here by $\mathbf{Lax}(\mathfrak A,\mathfrak B)$.
\end{lemma}
\begin{proof}  Composition of lax natural transformations is well-defined.  Taking $\alpha\colon F\Rightarrow G$ and $\beta\colon G\Rightarrow H$, declare $(\beta\alpha)_A:=\beta_A\alpha_A$, as expected, and take the 2-cell for lax naturality at an arrow $f\colon A\to B$ to be the juxtaposition of 2-cells 
$$\begin{tikzpicture}
\node(1){$FA$};
\node(2)[node distance=1.2in, right of=1]{$FB$};
\node(3)[node distance=.8in, below of=1]{$GA$};
\node(4)[node distance=.8in, below of=2]{$GB$};
\node(5)[node distance=.6in, right of=1]{$$};
\node(6)[node distance=.34in, below of=5]{$\alpha_f$};
\node(7)[node distance=.46in, below of=5]{$\Rightarrow$};
\node(8)[node distance=.8in, below of=3]{$HA$};
\node(9)[node distance=.8in, below of=4]{$HB$};
\node(10)[node distance=.8in, below of=6]{$\beta_{f}$};
\node(11)[node distance=.8in, below of=7]{$\Rightarrow$};
\draw[->](1) to node [above]{$Ff$}(2);
\draw[->](1) to node [left]{$\alpha_A$}(3);
\draw[->](2) to node [right]{$\alpha_B$}(4);
\draw[->](3) to node [fill=white]{$Gf$}(4);
\draw[->](3) to node [left]{$\beta_{A}$}(8);
\draw[->](4) to node [right]{$\beta_{B}$}(9);
\draw[->](8) to node [below]{$Hf$}(9);
\end{tikzpicture}$$ 
The lax naturality conditions are satisfied because they are satisfied by $\alpha$ and $\beta$.  Vertical composition of modifications is given by vertical composition of 2-cells in $\mathfrak B$; similarly, horizontal composition is given by horizontal composition in $\mathfrak B$.  The interchange law follows by interchange in $\mathfrak B$.  \end{proof}

\begin{remark}[On Notation]  Perhaps the notation `$\mathbf{Lax}$' is ultimately not well-chosen, given that, upon seeing it out of context, one might suppose it to indicate some $n$-category of $(n-1)$-categories with lax functors, lax transformations, and other higher cells.  Our justification for using it anyway is threefold.  Having defined it in this paper, we intend that it be interpreted in the way defined above.  Second, there is no other lax concept under consideration in the paper, so it is unambiguous in this context.  Finally, the obvious alternative is to subscript the notation `$\mathbf{2Cat}$' with indicators such as `$lax$', or perhaps just `$l$' to indicate the inclusion of lax natural transformations; however, this again could be read out of context for lax functors, so we are back to the original objection.  Moreover, subscripting does not seem to emphasize notationally the seriousness of the shift in perspective initiated by the change of the base category to include lax natural transformations.  Ultimately, our perspective is that `$\mathbf{2Cat}$' should be defined to indicate (some) lax structure, since this appears to us to be more fundamental and moreover in concert with the ``laxification" that occurs with the dimensional jump from $\mathbf{Set}$ to $\mathbf{Cat}$.  However,  the use of `$\mathbf{2Cat}$' is well-established as indicating only strictness, so we do not want to prematurely coopt the notation, potentially raising ire and causing unnescessary additional confusion.
\end{remark}

\subsection{The Ambient 3-Dimensional Setting}

This subsection presents the 3-categorical forum for the work of the paper.  Basically, it seems that it is usual to take $\mathbf{2Cat}$, or some "up-to-isomorphism" weakening, as the base 3-dimensional categorical structure in or over which to do work on 2-categories.  Here we construct an alternative setting, taking lax natural transformations as the primary notion.  The main results of the section show that together with 2-functors and modifications these form a 3-category; and that cotensors with $\mathbf 2$ in this 3-category are given by 2-categories of squares.

\begin{define}[Lax Functor] \label{lax functor} A \textbf{(normalized) lax functor} between 2-categories $F\colon \mathfrak A\to\mathfrak B$ makes object, arrow, and 2-cells assignments $A\mapsto FA$, $f\mapsto Ff$, and $\alpha \mapsto F\alpha$ and comes equipped with coherence 2-cells $\phi_{f,g}\colon FgFf\Rightarrow F(gf)$ for any two composable arrows $f$ and $g$, all satisfying the following conditions.
\begin{enumerate}
\item $F$ strictly preserves domains, codomains, sources and targets, identity morphisms and 2-cells, and vertical composition of 2-cells.
\item For composable arrows $f\colon A\to B$, $g\colon B\to C$ and $h\colon C\to D$ of $\mathfrak K$, there is an equality of composite coherence 2-cells
$$\begin{tikzpicture}
\node(1){$FB$};
\node(2)[node distance=1.2in, right of=1]{$FC$};
\node(3)[node distance=1in, below of=1]{$FA$};
\node(4)[node distance=1in, below of=2]{$FD$};
\node(5)[node distance=.4in, right of=1]{$$};
\node(6)[node distance=.5in, right of=5]{$$};
\node(7)[node distance=.3in, below of=5]{$\Downarrow \phi_{f,g}$};
\node(8)[node distance=.7in, below of=6]{$\Downarrow \phi_{gf,h}$};
\node(9)[node distance=2in, right of=2]{$FB$};
\node(10)[node distance=1.2in, right of=9]{$FC$};
\node(11)[node distance=1in, below of=9]{$FA$};
\node(12)[node distance=1in, below of=10]{$FD.$};
\node(13)[node distance=.35in, right of=9]{$$};
\node(14)[node distance=.7in, below of=13]{$\Downarrow \phi_{f,hg}$};
\node(15)[node distance=.5in, right of=13]{$$};
\node(16)[node distance=.3in, below of=15]{$\Downarrow \phi_{g,h}$};
\node(17)[node distance=1in, right of=2]{$$};
\node(18)[node distance=.5in, below of=17]{$=$};
\draw[->](1) to node [above]{$Fg$}(2);
\draw[->](3) to node [left]{$Ff$}(1);
\draw[->](2) to node [right]{$Fh$}(4);
\draw[->](3) to node [below]{$F(hgf)$}(4);
\draw[->](3) to node [sloped, fill=white]{$F(gf)$}(2);
\draw[->](9) to node [above]{$Fg$}(10);
\draw[->](9) to node [sloped, fill=white]{$F(hg)$}(12);
\draw[->](11) to node [left]{$Ff$}(9);
\draw[->](10) to node [right]{$Fh$}(12);
\draw[->](11) to node [below]{$F(hgf)$}(12);
\end{tikzpicture}$$
\item For horizontally composable 2-cells $\alpha\colon f\Rightarrow g$ and $\beta\colon h\Rightarrow k$ with $f,g\colon A\rightrightarrows B$ and $h,k\colon B\rightrightarrows C$, there is an equality of 2-cells
$$\begin{tikzpicture}
\node(1){$FA$};
\node(2)[node distance=.5in, right of=1]{$\Downarrow F\alpha$};
\node(3)[node distance=.5in, right of=2]{$FB$};
\node(4)[node distance=.5in, right of=3]{$\Downarrow F\beta$};
\node(5)[node distance=.5in, right of=4]{$FC$};
\node(6)[node distance=.4in, below of=3]{$\Downarrow \phi_{g,k}$};
\node(7)[node distance=1.5in, right of=5]{$FA$};
\node(8)[node distance=.75in, right of=7]{$$};
\node(9)[node distance=.75in, right of=8]{$FC$};
\node(10)[node distance=.3in, above of=8]{$\Downarrow \phi_{f,h}$};
\node(10)[node distance=.3in, below of=8]{$\Downarrow F(\beta\ast\alpha)$};
\node(11)[node distance=.75in, right of=5]{$=$};
\draw[->, bend left](1) to node [above]{$Ff$}(3);
\draw[->,bend right](1) to node [below]{$Fg$}(3);
\draw[->,bend left](3) to node [above]{$Fh$}(5);
\draw[->,bend right](3) to node [below]{$Fk$}(5);
\draw[->, bend right=100](1) to node [below]{$F(kg)$}(5);
\draw[->,bend left=100](7) to node [above]{$FgFf$}(9);
\draw[->](7) to node [fill=white]{$F(gf)$}(9);
\draw[->,bend right=100](7) to node [below]{$F(kg)$}(9);
\end{tikzpicture}$$
\end{enumerate}
Let $|\mathbf{2Cat}_{lax}|$ denote the 1-category of 2-categories and lax functors between them.  
\end{define}
\begin{remark}  The notation `$|\mathbf{2Cat}_{lax}|$' is used with out double `$|$' because 2-categories with lax functors cannot be made into a 3-category with any reasonable notion of transformation.  See \cite{ProblemWithLax}.
\end{remark}

\begin{remark}  The category $|\mathbf{2Cat}_{lax}|$ is cartesian monoidal with ordinary cartesian products of 2-categories giving the product.  Thus, the following makes sense.
\end{remark}

\begin{define}  A \textbf{lax 3-category} is a category enriched in $|\mathbf{2Cat}_{lax}|$.
\end{define}

Although certainly every ordinary 3-category is an obvious (and trivial) example, the data for the main example of this paper is given in the following development.

\begin{construction}[Data for 2-Categories with Lax Natural Transformations] \label{constructions for LAX}  The hom-categories for a lax 3-category will be the 2-categories $\mathbf{Lax}(\mathfrak A,\mathfrak B)$.  Required for enrichment are composition and identity morphisms in $|\mathbf{2Cat}_{lax}|$.  First fix 2-categories $\mathfrak A$, $\mathfrak B$, and $\mathfrak C$.  Construct what will be a lax functor 
\begin{equation}
-\otimes-\colon \mathbf{Lax}(\mathfrak B,\mathfrak C)\times \mathbf{Lax}(\mathfrak A,\mathfrak B)\to \mathbf{Lax}(\mathfrak A,\mathfrak C)
\end{equation}
in the following way.  On objects, i.e. composable pairs $(G,F)$ of 2-functors, take $G\otimes F$ to be the ordinary composition $GF$.  For horizontally composable lax natural transformations $\alpha\colon F\Rightarrow H$ and $\beta\colon G\Rightarrow K$ with $F,H\colon \mathfrak A\rightrightarrows \mathfrak B$ and $G,K\colon \mathfrak B\rightrightarrows \mathfrak C$, take $\beta\otimes\alpha$ to have components 
\[ (\beta\otimes\alpha)_A:=\beta_{HA}G\alpha_A
\] 
indexed over $A\in\mathfrak A_0$.  Given a morphism $f\colon A\to B$ of $\mathfrak A$, the 2-cell $(\beta\otimes\alpha)_f$ required for lax naturality is then the composite 2-cell
$$\begin{tikzpicture}
\node(1){$GFA$};
\node(2)[node distance=1.2in, right of=1]{$GFB$};
\node(3)[node distance=.8in, below of=1]{$GHA$};
\node(4)[node distance=.8in, below of=2]{$GHB$};
\node(5)[node distance=.6in, right of=1]{$$};
\node(6)[node distance=.34in, below of=5]{$G\alpha_f$};
\node(7)[node distance=.46in, below of=5]{$\Rightarrow$};
\node(8)[node distance=.8in, below of=3]{$KHA$};
\node(9)[node distance=.8in, below of=4]{$KHB$};
\node(10)[node distance=.8in, below of=6]{$\beta_{Hf}$};
\node(11)[node distance=.8in, below of=7]{$\Rightarrow$};
\draw[->](1) to node [above]{$GFf$}(2);
\draw[->](1) to node [left]{$G\alpha_A$}(3);
\draw[->](2) to node [right]{$G\alpha_B$}(4);
\draw[->](3) to node [fill=white]{$GHf$}(4);
\draw[->](3) to node [left]{$\beta_{HA}$}(8);
\draw[->](4) to node [right]{$\beta_{HB}$}(9);
\draw[->](8) to node [below]{$KHf$}(9);
\end{tikzpicture}$$ 
The conditions for lax naturality are satisfied since they are satisfied by the $\alpha_f$ and $\beta_{Hf}$ over morphisms $f\colon A\to B$.  Given further lax natural transformations $\gamma\colon F\Rightarrow H$ and $\delta\colon G\Rightarrow K$ and two modifications $m\colon \alpha \Rrightarrow \gamma$ and $n\colon \beta\Rrightarrow \delta$, define what will be the component of a modification $n\otimes m$ as the horizontal composite 
\[ (n\otimes m)_A:=n_{HA}\ast Gm_A
\]
over $A\in\mathfrak A_0$.  That this is well-defined, that is, satisfies the modification condition, is just a result of the fact that both $m$ and $n$ are modifications and that $G$ is functorial on 2-cells.
\end{construction}

\begin{lemma}  The assignments of Construction \ref{constructions for LAX} make
\begin{equation}
-\otimes-\colon \mathbf{Lax}(\mathfrak B,\mathfrak C)\times \mathbf{Lax}(\mathfrak A,\mathfrak B)\to \mathbf{Lax}(\mathfrak A,\mathfrak C)
\end{equation}
into a lax functor of 2-categories in the sense of Definition \ref{lax functor}.
\end{lemma}
\begin{proof}  Preservation of domains, codomains, sources, targets and 1- and 2-cell identities are all straightforward to check.  The compatibility cells, however, need to be exhibited and the two conditions of Definition \ref{lax functor} need to be checked.  Given further lax natural transformations $\delta\colon K\Rightarrow M$ and $\gamma\colon H\Rightarrow L$, required is a compatibility cell, that is, a modification $(\delta\otimes \gamma)(\beta\otimes\alpha) \Rrightarrow \delta\beta\otimes \gamma\alpha$.  Unraveling both sides at objects $A\in\mathfrak A_0$, we see that this amounts to giving 2-cells
\[ \delta_{LA}K(\gamma_A)\beta_{HA}G(\alpha_A)\Rightarrow \delta_{LA}\beta_{LA}G\gamma_AG\alpha_A
\]
satisfying the appropriate compatibility condition for a modification.  But there is an evident choice by taking the cell to be $\beta_{\gamma_A}$ as in the diagram
$$\begin{tikzpicture}
\node(1){$GFA$};
\node(2)[node distance=1in, right of=1]{$GHA$};
\node(3)[node distance=.8in, right of=2]{$\Downarrow \beta_{\gamma_A}$};
\node(4)[node distance=.8in, right of=3]{$KLA$};
\node(5)[node distance=1in, right of=4]{$MLA$};
\node(6)[node distance=.5in, above of=3]{$KHA$};
\node(7)[node distance=.5in, below of=3]{$GLA$};
\draw[->](1) to node [above]{$G\alpha_A$}(2);
\draw[->](2) to node [sloped,above]{$\beta_{HA}$}(6);
\draw[->](6) to node [sloped,above]{$K\gamma_A$}(4);
\draw[->](4) to node [above]{$\delta_{LA}$}(5);
\draw[->](2) to node [sloped,below]{$G\gamma_A$}(7);
\draw[->](7) to node [sloped,below]{$\beta_{LA}$}(4);
\end{tikzpicture}$$
The modification condition for a fixed arrow $f\colon A\to B$ of $\mathfrak A$ follows by the second lax naturality condition for the associated cell $\gamma_f$.  That the associativity compatibility condition for lax functoriality on 1-cells is satisfied follows by the lax naturality of $\beta$ and the definition of composition of lax natural transformations.

Preservation of vertical composition of 2-cells (that is, modifications) is straightforward using the fact that all the 2-functors involved strictly preserve vertical composition of 2-cells.  However, the compatibility condition in Definition \ref{lax functor} for horizontal composition is less clear.  Thus, suppose that further modifications $p\colon \tau\Rrightarrow \chi$ and $l\colon \sigma\Rrightarrow \rho$ are given between lax natural transformations $\tau,\chi\colon K\Rightarrow M$ and $\sigma,\rho\colon H\Rightarrow L$.  One computes $(p\ast n)\otimes (l\ast m)$ on the one hand and $(p\otimes l)\ast (n\otimes m)$ on the other, and adding in the coherence cells as defined above, the compatibility condition will follow from an equality of the 2-cells
$$\begin{tikzpicture}
\node(1){$GHA$};
\node(2)[node distance=1.4in, right of=1]{$GLA$};
\node(3)[node distance=1in, below of=1]{$KHA$};
\node(4)[node distance=1in, below of=2]{$KLA$};
\node(5)[node distance=.5in, right of=1]{$$};
\node(6)[node distance=.25in, above of=5]{$Gl_A\Uparrow$};
\node(7)[node distance=.5in, below of=5]{$\Uparrow \beta_{\sigma_A}$};
\node(8)[node distance=.9in, right of=7]{$\mathclap{\substack{n_{LA} \\ \Rightarrow }}$};
\node(9)[node distance=2in, right of=7]{$=$};
\node(10)[node distance=2in, right of=2]{$GHA$};
\node(11)[node distance=1.4in, right of=10]{$GLA$};
\node(12)[node distance=1in, below of=10]{$KHA$};
\node(13)[node distance=1in, below of=11]{$KLA$};
\node(14)[node distance=3.8in, right of=7]{$\Uparrow \delta_{\rho_A}$};
\node(15)[node distance=.75in, below of=14]{$\Uparrow Kl_A $};
\node(16)[node distance=.9in, left of=14]{$\mathclap{\substack{n_{HA} \\ \Rightarrow }}$};
\draw[->](1) to node [below]{$G\sigma_A$}(2);
\draw[->,bend left=90](1) to node [above]{$G\rho_A$}(2);
\draw[->](1) to node [left]{$\beta_{HA}$}(3);
\draw[->,bend left](2) to node [right]{$\delta_{LA}$}(4);
\draw[->,bend right](2) to node [left]{$\beta_{LA}$}(4);
\draw[->](3) to node [below]{$K\sigma_A$}(4);
\draw[->](10) to node [above]{$G\rho_A$}(11);
\draw[->,bend left](10) to node [right]{$\delta_{HA}$}(12);
\draw[->,bend right](10) to node [left]{$\beta_{HA}$}(12);
\draw[->](11) to node [right]{$\delta_{LA}$}(13);
\draw[->](12) to node [above]{$K\rho_A$}(13);
\draw[->,bend right=90](12) to node [below]{$K\sigma_A$}(13);
\end{tikzpicture}$$
That this equality does in fact hold is now easy to establish, first using the modification condition for $n$ at $\sigma_A$ and then by the using the second lax naturality condition for $\delta$ at the 2-cell $l_A$.  \end{proof}

\begin{theo}  The lax composition functors as in Construction \ref{constructions for LAX} make 2-categories, 2-functors, lax natural transformations and modifications into a lax 3-category, denoted by $\mathbf{Lax}$. 
\end{theo}
\begin{proof}  It remains to check the pentagonal associativity condition and the identity conditions.  But these are now easy.  For 2-categories $\mathfrak A$, $\mathfrak B$, $\mathfrak C$ and $\mathfrak D$, the associativity condition asserts that there is an equality of lax functors
$$\begin{tikzpicture}
\node(1){$(\mathbf{Lax}(\mathfrak C,\mathfrak D)\times \mathbf{Lax}(\mathfrak B,\mathfrak C))\times \mathbf{Lax}(\mathfrak A,\mathfrak B)$};
\node(2)[node distance=3.5in, right of=1]{$\mathbf{Lax}(\mathfrak C,\mathfrak D)\times (\mathbf{Lax}(\mathfrak B,\mathfrak C)\times \mathbf{Lax}(\mathfrak A,\mathfrak B))$};
\node(3)[node distance=.8in, below of=1]{$\mathbf{Lax}(\mathfrak B,\mathfrak D)\times \mathbf{Lax}(\mathfrak B,\mathfrak C)$};
\node(4)[node distance=.8in, below of=2]{$\mathbf{Lax}(\mathfrak C,\mathfrak D)\times \mathbf{Lax}(\mathfrak A,\mathfrak C).$};
\node(5)[node distance=1.75in, right of=1]{$$};
\node(6)[node distance=.6in, below of=5]{$=$};
\node(7)[node distance=.6in, below of =6]{$\mathbf{Lax}(\mathfrak A,\mathfrak D)$};
\draw[->](1) to node [above]{$\cong$}(2);
\draw[->](1) to node [left]{$\otimes\times 1$}(3);
\draw[->](2) to node [right]{$1\times\otimes$}(4);
\draw[->](3) to node [below]{$\otimes$}(7);
\draw[->](7) to node [below]{$\otimes$}(4);
\end{tikzpicture}$$
But this follows readily.  For composition of 2-functors is strictly associative.  At the level of 1-cells, it is a direct computation using the definition of $\otimes$.  Take 2-functors $F,K\colon \mathfrak A\rightrightarrows \mathfrak B$, $G,L\colon \mathfrak B\rightrightarrows \mathfrak C$ and $H,M\colon \mathfrak C\rightrightarrows \mathfrak D$ with lax natural transformations $\alpha\colon F\Rightarrow K$, $\beta\colon G\Rightarrow L$ and $\gamma\colon H\Rightarrow L$.  On the one hand, computing around the counterclockwise direction of the diagram, we have
\[ ((\gamma\otimes \beta)\otimes \alpha)_A = (\gamma\otimes\beta)_{KA} HG\alpha_A = \gamma_{LKA}H(\beta_{KA})H(G(\alpha_A)
\]
and around the clockwise direction on the other hand
\[ (\gamma\otimes(\beta\otimes \alpha))_A=\gamma_{LKA}H(\beta\otimes \alpha)_A = H(\beta_{KA}\alpha_A).
\]
These results are evidently the same since $H$ is a 2-functor.  The computation at the 2-cell level is analogous.  The identity conditions also follow by direct inspection. \end{proof}

\begin{remark} The reason for our focus on $\mathbf{Lax}$ over $\mathbf{2Cat}$ is summarized in the next result, showing that the lax comma construction, hence the squares 2-category of Example \ref{2cat of squares example}, is universal amid lax natural transformations and not just 2-natural transformations.  In fact this universality is much like the universality of the cotensor with $\mathbf 2$ in $\mathbf{Lax}$.  And the importance is that 2-fibrations and discrete 2-fibrations will be precisely the algebras for certains actions of $\mathbf{Sq}(\mathfrak B)$ but not $\mathbf{2Cat}(\mathbf 2,\mathfrak B)$.  
\end{remark}

Now, specifically, the universal property of $F/G$ in $\mathbf{Lax}$ is a ``laxification" of the universal property of the 2-comma category, proved in Proposition \ref{up of 2-comma}.  As already remarked, what is here called the ``lax comma category" is called the ``2-comma category" in \S I,2.5 of \cite{GrayFormalCats}.  The development in \S I,5.2 of the reference does not appear to completely describe the universality enjoyed by this construction.  The seemingly complete statement is the following.

\begin{prop}[Universal Property of Lax Comma Category; Cf. \S I,5.2 of \cite{GrayFormalCats} and \S 1 p.108 of \cite{StreetFibrations}] \label{up of lax comma} Given 2-functors $F\colon \mathfrak A\to \mathfrak B$ and $G\colon \mathfrak C\to\mathfrak B$, the lax comma category $F/G$ of Construction \ref{lax comma cat construction} is 1-, 2- and 3-dimensionally universal in the following sense.
\begin{enumerate}
\item  Given 2-functors $H\colon \mathfrak D\to\mathfrak A$ and $K\colon \mathfrak D\to \mathfrak C$ and any lax transformation $\theta\colon FH\Rightarrow GK$, there is a unique 2-functor $U\colon \mathfrak D\to F/G$ such that $\theta = \lambda \ast U$.
\item Given lax natural transformations $\xi$ and $\zeta$ together with a modification $m$ as in the diagram
$$
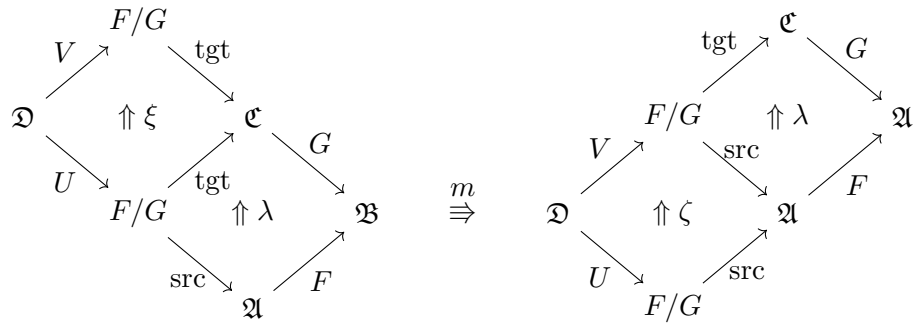
\begin{tikzpicture}
\node(1){$\mathfrak D$};
\node(2)[node distance=.6in, right of=1]{$\Uparrow \xi$};
\node(3)[node distance=.6in, right of=2]{$\mathfrak C$};
\node(4)[node distance=.5in, above of=2]{$F/G$};
\node(5)[node distance=.5in, below of=2]{$F/G$};
\node(6)[node distance=.6in, right of=5]{$\Uparrow\lambda$};
\node(7)[node distance=.6in, right of=6]{$\mathfrak B$};
\node(8)[node distance=.5in, below of=6]{$\mathfrak A$};
\node(9)[node distance=1in, right of=7]{$\mathfrak D$};
\node(10)[node distance=.6in, right of=9]{$\Uparrow\zeta$};
\node(11)[node distance=.6in, right of=10]{$\mathfrak A$};
\node(12)[node distance=.5in, above of=10]{$F/G$};
\node(13)[node distance=.5in, below of=10]{$F/G$};
\node(14)[node distance=.6in, right of=12]{$\Uparrow\lambda$};
\node(15)[node distance=.6in, right of=14]{$\mathfrak A$};
\node(16)[node distance=.5in, above of=14]{$\mathfrak C$};
\node(17)[node distance=.5in, right of=7]{$\Rrightarrow$};
\node(18)[node distance=.12in, above of=17]{$m$};
\draw[->](1) to node [above]{$V\;\;\;$}(4);
\draw[->](1) to node [below]{$U\;\;\;$}(5);
\draw[->](4) to node [above]{$\;\;\;\tgt$}(3);
\draw[->](5) to node [below]{$\;\;\;\tgt$}(3);
\draw[->](5) to node [below]{$\src\;\;\;$}(8);
\draw[->](3) to node [above]{$\;\;\;G$}(7);
\draw[->](8) to node [below]{$\;\;\;F$}(7);
\draw[->](9) to node [above]{$V\;\;\;$}(12);
\draw[->](9) to node [below]{$U\;\;\;$}(13);
\draw[->](12) to node [above]{$\;\,\src$}(11);
\draw[->](13) to node [below]{$\;\;\;\src$}(11);
\draw[->](12) to node [above]{$\tgt\;\;\;\;$}(16);
\draw[->](11) to node [below]{$\;\;\;F$}(15);
\draw[->](16) to node [above]{$\;\;\;G$}(15);
\end{tikzpicture}$$
there is a lax natural transformation $\omega\colon U\Rightarrow V$ such that the equations $\tgt\ast \omega = \xi$ and $\src\ast\omega = \zeta$ each hold.
\item Given lax natural transformations $\omega \colon U \Rightarrow V$ and $\chi\colon U \Rightarrow V$ between $U,V\colon \mathfrak D\rightrightarrows F/G$ as above with modfications $m\colon \src\ast \omega \Rrightarrow \src\ast \chi$ and $n\colon \tgt\ast \omega \Rrightarrow \tgt\ast \chi$ satisfying 
\begin{equation} \label{compat} (\lambda\ast V)Fm = Gn(\lambda\ast U)
\end{equation}
there is a unique modification $l\colon \omega\Rrightarrow \chi$ satisfying $\tgt\ast l = n$ and $\src\ast l = m$.
\end{enumerate}
These properties characterize $F/G$ up to isomorphism in $\mathbf{Lax}$.
\end{prop}
\begin{proof}  The 1-dimensional aspect of the universal property was proved in Lemma \ref{comma 1-dim universal lemma}.  The proof of the 3-dimensional aspect is the same as in the proof of Proposition \ref{up of 2-comma} with suitable adaptations for lax naturality.  Thus, we prove the second condition, that is, the 2-dimensional aspect of the universal property.  Thus, given the data of $\xi$, $\zeta$ and $m$, we need to construction $\omega\colon U\Rightarrow V$.  The component of the modification $m$ at say $D\in \mathfrak D_0$ is a 2-cell
$$\begin{tikzpicture}
\node(1){$\cdot$};
\node(2)[node distance=1in, right of=1]{$\cdot$};
\node(3)[node distance=.8in, below of=1]{$\cdot$};
\node(4)[node distance=.8in, below of=2]{$\cdot$};
\node(5)[node distance=.5in, right of=1]{$$};
\node(6)[node distance=.34in, below of=5]{$m_D$};
\node(7)[node distance=.47in, below of=5]{$\Rightarrow$};
\draw[->](1) to node [above]{$F\zeta_D$}(2);
\draw[->](1) to node [left]{$UD$}(3);
\draw[->](2) to node [right]{$VD$}(4);
\draw[->](3) to node [below]{$G\xi_D$}(4);
\end{tikzpicture}$$
of $\mathfrak B$.  Thus, define the component of $\omega$ at $D\in \mathfrak D_0$ to be the arrow $(\zeta_D,\xi_D,m_D)$ of $F/G$.  Given an arrow $g\colon C\to D$ of $\mathfrak D$, there should be a lax naturality cell from $(\zeta_C,\xi_C,m_C)$ to $(\zeta_D,\xi_D,m_D)$, the arrows of which will be the 2-cells $Ug$ and $Vg$.  The actual 2-cell in $F/G$ is given by the lax naturality cells $\zeta_g$ and $\xi_g$.  This does define the required lax naturality 2-cell in $F/G$ by the equality
$$\begin{tikzpicture}
\node(1){$\cdot$};
\node(2)[node distance=1in, right of=1]{$\cdot$};
\node(3)[node distance=.8in, below of=1]{$\cdot$};
\node(4)[node distance=.8in, below of=2]{$\cdot$};
\node(5)[node distance=.5in, right of=1]{$$};
\node(6)[node distance=.34in, below of=5]{$m_C$};
\node(7)[node distance=.47in, below of=5]{$\Rightarrow$};
\node(8)[node distance=.3in, above of=2]{$\zeta_g \Uparrow$};
\node(9)[node distance=3in, right of=8]{$\cdot$};
\node(10)[node distance=1in, right of=9]{$\cdot$};
\node(11)[node distance=.8in, below of=9]{$\cdot$};
\node(12)[node distance=.8in, below of=10]{$\cdot$};
\node(13)[node distance=.5in, right of=9]{$$};
\node(14)[node distance=.32in, below of=13]{$Ug$};
\node(15)[node distance=.45in, below of=13]{$\Rightarrow$};
\node(16)[node distance=.3in, below of=12]{$\Uparrow\xi_g$};
\node(17)[node distance=2.5in, right of=7]{$=$};
\node(18)[node distance=1in, right of=2]{$\cdot$};
\node(19)[node distance=.8in, below of=18]{$\cdot$};
\node(20)[node distance=1in, right of=6]{$Vg$};
\node(21)[node distance=1in, right of=7]{$\Rightarrow$};
\node(22)[node distance=1in, right of=10]{$\cdot$};
\node(23)[node distance=.8in, below of=22]{$\cdot$};
\node(24)[node distance=1in, right of=14]{$m_g$};
\node(25)[node distance=1in, right of=15]{$\Rightarrow$};
\node(26)[node distance=.6in, above of=2]{$\cdot$};
\node(27)[node distance=.6in, below of=12]{$\cdot$};
\draw[->](1) to node [below]{$\zeta_C$}(2);
\draw[->](1) to node [left]{$UC$}(3);
\draw[->](2) to node [right]{$VC$}(4);
\draw[->](3) to node [below]{$\xi_C$}(4);
\draw[->](1) to node [above,sloped]{$\src\,Ug$}(26);
\draw[->](26) to node [above]{$\zeta_D$}(18);
\draw[->](9) to node [above]{$\src\,Ug$}(10);
\draw[->](11) to node [below]{$\xi_C$}(27);
\draw[->](27) to node [below,sloped]{$\tgt\,Vg$}(23);
\draw[->](9) to node [left]{$UC$}(11);
\draw[->](10) to node [right]{$UD$}(12);
\draw[->](11) to node [above]{$\tgt\,Ug$}(12);
\draw[->](2) to node [below]{$\src\,Vg$}(18);
\draw[->](18) to node [right]{$VD$}(19);
\draw[->](4) to node [below]{$\tgt\,Vg$}(19);
\draw[->](10) to node [above]{$\zeta_D$}(22);
\draw[->](22) to node [right]{$UD$}(23);
\draw[->](12) to node [above]{$\xi_D$}(23);
\end{tikzpicture}$$
which holds by the modification condition for $m$ since the other cells on either side are the coherence cells for the composite lax natural transformations.  The compatibility conditions for lax naturality as in Definition \ref{LAXNATTRANSF} are satisfied since the same conditions are satisfied by $\xi$ and $\zeta$.  The lax natural transformation $\omega$ is by construction evidently the unique one with the desired properties. \end{proof}

\begin{remark}  Thus, the proposition says that the lax comma category for $F=G=1_{\mathfrak B}$ is a representing object in the sense that there is an isomorphism
\[  \lax(\mathfrak A,\sq(\mathfrak B))\cong \lax(\mathbf 2,\lax(\mathfrak A,\mathfrak B))
\]
of 2-categories.  This is an analogue of the non-elementary statement of the universal property of the 2-comma object for $1_{\mathfrak B}$ in Remark \ref{remark on enriched cotensor}.
\end{remark}

\section{Discrete 2-Fibrations}

\label{section: disc 2-fibs}

This section sets out the main definition of the paper, namely, that of a discrete 2-fibration.  The ultimate object is to state and prove the main result, Theorem \ref{theorem duality for disc 2-fibrations}, showing that discrete 2-fibrations correspond, roughly speaking, to category-valued 2-functors indexed by 2-categories.  This result will be reconciled with Buckley's notion of a 2-fibration in the last subsection.  Along the way, we give one justification for the correctness of our definition in the form of a ``Chevalley condition" categorifying that for ordinary discrete fibrations.

\subsection{Definition}

As a preliminary, recall that for any functor $F\colon \C^{op}\to\mathbf{Set}$, the associated category of elements fits into a comma square
$$\begin{tikzpicture}
\node(1){$\elt(F)^{op}$};
\node(2)[node distance=1in, right of=1]{$\C^{op}$};
\node(3)[node distance=.7in, below of=1]{$\mathbf 1$};
\node(4)[node distance=.7in, below of=2]{$\mathbf{Set}$};
\node(5)[node distance=.5in, right of=1]{$$};
\node(6)[node distance=.35in, below of=5]{$\Rightarrow$};
\draw[->](1) to node [above]{$\Pi^{op}$}(2);
\draw[->](1) to node [left]{$$}(3);
\draw[->](2) to node [right]{$F$}(4);
\draw[->](3) to node [below]{$\ast$}(4);
\end{tikzpicture}$$
where $\ast\colon \mathbf 1\to\mathbf{Set}$ denotes the inclusion of the one-element set.  Notice that $\Pi^{op}$ is a discrete opfibration, hence that $\Pi$ is a discrete fibration.  And the content of Theorem \ref{disc fib set-functor duality thm} is that a functor $P\colon \F\to\C$ is a discrete fibration if, and only if, $P^{op}$ is isomorphic over $\C^{op}$ to the opposite of the category of elements of a canonically constructed functor $F_P\colon \C^{op}\to\mathbf{Set}$.  Going up a level and switching to category-valued functors, the category of elements for a 2-functor $F\colon \C^{op}\to\cat$ fits into not a 2-comma square, but instead a \emph{lax} comma square of the same form
$$\begin{tikzpicture}
\node(1){$\elt(F)^{op}$};
\node(2)[node distance=1in, right of=1]{$\C^{op}$};
\node(3)[node distance=.7in, below of=1]{$\mathbf 1$};
\node(4)[node distance=.7in, below of=2]{$\cat$};
\node(5)[node distance=.5in, right of=1]{$$};
\node(6)[node distance=.35in, below of=5]{$\Rightarrow$};
\draw[->](1) to node [above]{$\Pi^{op}$}(2);
\draw[->](1) to node [left]{$$}(3);
\draw[->](2) to node [right]{$F$}(4);
\draw[->](3) to node [below]{$\ast$}(4);
\end{tikzpicture}$$
with $\ast\colon \mathbf 1\to\cat$ denoting the inclusion of the one-object category.  An analogous interpretation of the fibration-functor duality result Theorem \ref{opfibration psdfunctor correspondence} can be made here, characterizing split fibrations in terms of projections in lax comma squares of this form.  

Note that if one were taking the former result as a guide for where to look for fibration properties of projections from various comma squares associated to functors valued in $\cat$, one would technically end up with choice of taking a 2-comma, iso-comma or lax-comma.  But of course the 2-comma is not categorically sound, since it would result in morphisms $f\colon (C,X) \to (D,Y)$ with $f^*X=Y$ holding strictly in $FC$, which is generally considered to be ``evil."  Thus, iso and lax would be appropriate cases to study.  However, Theorem \ref{opfibration psdfunctor correspondence} shows that lax is ultimately the fruitful choice.  And indeed the following development shows that the 2-category of elements for a 2-functor $F\colon\mathfrak B^{op}\to\cat$ on a 2-category $\mathfrak B$ turns out to be a generalization of this construction.

\begin{construction}[2-Category of Elements; Originally from \cite{BirdThesis}; general version in \S 2.2.1 \cite{Buckley}] \label{2cat of elts discrete case} For any 2-functor $F\colon \mathfrak B^{op}\to\cat$ on a 2-category $\mathfrak B$, the \textbf{2-category of elements} of $E$ is the 2-category whose
\begin{enumerate}
\item objects are pairs $(B,X)$ with $B\in \mathfrak B_0$ and $X\in FB$;
\item arrows are pairs $(f,u)\colon (B,X)\to (C,Y)$ with $f\colon B\to C$ in $\mathfrak B$ and $u\colon X\to f^*Y$ in the fiber $FB$;
\item and whose 2-cells $\colon (f,u)\Rightarrow (g,v)$ are those 2-cells $\alpha\colon f\Rightarrow g$ in $\mathfrak B$ making a commutative triangle
$$\begin{tikzpicture}
\node(1){$f^*Y$};
\node(2)[node distance=.4in, below of=1]{$$};
\node(3)[node distance=.4in, below of=2]{$g^*Y$};
\node(4)[node distance=1.2in, left of=2]{$X$};
\node(5)[node distance=.4in, left of=2]{$=$};
\draw[->](4) to node [above]{$u$}(1);
\draw[->](1) to node [right]{$\alpha^*_Y$}(3);
\draw[->](4) to node [below]{$v$}(3);
\end{tikzpicture}$$ 
of the category $FB$.
\end{enumerate}
Denote this 2-category by $\elt(E)$.  There is an evident projection 2-functor $\Pi\colon \elt(E)\to \mathfrak B$.
\end{construction}

\begin{prop} For any 2-functor $F\colon \mathfrak B^{op}\to\mathbf{Cat}$, the opposite of the 2-category of elements as above presents the lax comma object $\elt(F^{op}) \cong \ast/F$ with a universal lax natural transformation
$$\begin{tikzpicture}
\node(1){$\elt(F)^{op}$};
\node(2)[node distance=1in, right of=1]{$\mathfrak B^{op}$};
\node(3)[node distance=.7in, below of=1]{$\mathbf 1$};
\node(4)[node distance=.7in, below of=2]{$\mathbf{Cat}$};
\node(5)[node distance=.5in, right of=1]{$$};
\node(6)[node distance=.35in, below of=5]{$\Rightarrow$};
\draw[->](1) to node [above]{$\Pi^{op}$}(2);
\draw[->](1) to node [left]{$$}(3);
\draw[->](2) to node [right]{$F$}(4);
\draw[->](3) to node [below]{$\ast$}(4);
\end{tikzpicture}$$
where $\ast\colon  \mathbf 1\to\mathbf{Cat}$ denotes the map sending the unique element of $\mathbf 1$ to the terminal category.
\end{prop}
\begin{proof}  Straightforward check from the constructions.  \end{proof}

Since every discrete fibration appears, up to isomorphism, as the projection from a category of elements, we examine the fibration properties of the projection $\Pi\colon \elt(F)\to\mathfrak B$ to codify the 2-dimensional analogue. 

\begin{prop} \label{disc 2-fib lemma} Let $F\colon \mathfrak B^{op}\to\mathbf{Cat}$ denote a 2-functor.  The projection 2-functor $\Pi\colon\elt(F)\to \mathfrak B$ from the 2-category of elements (Construction \ref{2cat of elts}) has the following fibration properties. 
\begin{enumerate}
\item The ordinary functor $|\Pi|\colon |\elt(F)|\to |\mathfrak B|$ of underlying 1-categories is a split fibration.
\item Locally $\Pi$ is a discrete opfibration.
\end{enumerate} 
\end{prop}
\begin{proof}  Since at the level of 1-categories, the 2-category of elements is the same as the ordinary 1-category of elements, the first point has been established.  For the discrete opfibration claim, start with a morphism $(f,u)\colon (C,X) \to (D,Y)$ and a cell $\alpha\colon f\Rightarrow g\colon C\rightrightarrows D$ from the image of $(f,u)$ under $\Pi$ in $\mathfrak B$.  The required lift is the cell
$$\begin{tikzpicture}
\node(1){$(C,X)$};
\node(2)[node distance=.8in, right of=1]{$\Downarrow (\alpha,1)$};
\node(3)[node distance=.8in, right of=2]{$(D,Y)$};
\draw[->,bend left](1) to node [above]{$(f,u)$}(3);
\draw[->,bend right,dashed](1) to node [below]{$(g,\alpha^*_Y\circ u)$}(3);
\end{tikzpicture}$$
This evidently is over $\alpha$ via the projection $\Pi$.  And it is the unique such morphism since the values of $F$ are ordinary categories, hence discrete as 2-categories.  \end{proof}

\begin{remark}  The 2-category of elements construction for a 2-functor $E\colon \mathfrak B\to\mathfrak{Cat}$ will be a split opfibration at the level of underlying 1-categories and a discrete fibration locally.
\end{remark}

\begin{define} \label{disc 2-fibration defn} A \textbf{split discrete 2-fibration} is a 2-functor $E\colon \mathfrak E\to\mathfrak C$ such that  
\begin{enumerate}
\item the underlying functor $|E|\colon |\mathfrak E|\to|\mathfrak B|$ is a split fibration;
\item $E$ itself is locally a discrete fibration, in that each functor $E\colon \mathfrak E(X,Y)\to \mathfrak C(EX,EY)$ is a discrete fibration.
\end{enumerate}
Dually, a discrete 2-opfibration is a 2-functor $E\colon \mathfrak E\to\mathfrak B$ whose underlying functor of 1-categories is a split opfibration and which is locally a discrete opfibration.
\end{define}

\subsubsection{Examples}

\begin{example}  Any split fibration $F\colon \F\to\C$ is a split discrete 2-fibration.  Dually, any split opfibration is a split discrete 2-opfibration.
\end{example}

\begin{example} \label{domain projection example} The domain projection $d_0\colon \mathfrak C/C\to\mathfrak C$ from the lax slice of a 2-category $\mathfrak C/C$ (from Example \ref{lax slice example}) back to $\mathfrak C$ is a discrete 2-fibration.  Notice that each $\mathfrak C/C$ is the fiber of the 2-functor $\tgt\colon \sq(\mathfrak C)\to \mathfrak C$.  This means, of course, that $d_0$ is a fragment of the source 2-functor.  Discrete 2-fibrations isomorphic to those of the form $d_0$ are said to be \textbf{representable}.
\end{example}

\begin{example} \label{codomain example}  The 2-functor 
\[  \mathrm{cod}\colon \mathbf{DOpf} \to\mathbf{Cat}
\]
sending a discrete opfibration (with small fibers) to its codomain and extended suitably to morphisms and 2-cells is a discrete 2-fibration.
\end{example}

\begin{example} \label{algebra for monad in K discrete example} Let $\K$ denote a 2-category and $t\colon B\to B$ a 2-monad in $\K$.  Let $t\mathbf{Alg}$ denote the 2-category of $t$-algebras as in Example \ref{algebra for monad in 2cat eg}.  The forgetful 2-functor
\[ \Pi\colon t\mathbf{Alg}\to\K
\]
is a discrete 2-fibration in the sense of Definition \ref{disc 2-fibration defn}.  This is because, as observed in the reference, pulling back by a morphism or by a 2-cell preserves $t$-algebra structure, but changes the domain or the target.
\end{example}

\begin{example}[Cf. ``Families" of \cite{Buckley} as Example \ref{family 2fibn} below] \label{small fam example} Let $\mathfrak B$ denote a (small) 2-category.  Note that 2-functors $\C\to\mathfrak B$ from a 1-category $\C$ amount to 1-functors $\C\to |\mathfrak B|$ since $\C$ has no nontrivial 2-cells.  Take $\mathbf{fam}(\mathfrak B)$ to denote the 2-category whose objects are pairs $(\C,F)$ where $F\colon \C\to\mathfrak B$ is a functor.  The morphisms are of the same form as those in $\mathbf{Fam}(\mathfrak B)$ from Example \ref{family 2fibn} below.  The 2-cells are also essentially the same, but with the difference that the modification in the definition must be an identity.  The projection
\[ \Pi\colon \mathbf{fam}(\mathfrak B) \to \mathbf{Cat}
\]
is then a discrete 2-fibration as in Definition \ref{disc 2-fibration defn}.  Note that there is an inclusion $\mathbf{fam}(\mathfrak B)\to\mathbf{Fam}(\mathfrak B)$ commuting with the projections to $\mathfrak B$.  Since $\Pi\colon \mathbf{fam}(\mathfrak B)\to\mathbf{Cat}$ is \emph{a fortiori} a 2-fibration, this might be seen as another main example of an inclusion of a sub-2-fibration insofar as such a concept is of interest.
\end{example}

\subsection{The Representation Theorem}

Here is presented a direct proof of the first main result of the paper, namely, that discrete 2-fibrations correspond via the category of elements construction to category-valued 2-functors indexed by the base 2-category.  The candidate for the pseudo-inverse is developed first and subsequently shown to be the correct construction via lax pullback squares. 

\subsubsection{The Pseudo-Inverse}

The following development presents an adaptation of \S2.2.3 in \cite{Buckley} suitable to the 2-cell conventions discussed in the Introduction.  Throughout let $P\colon \mathfrak E\to\mathfrak B$ denote a discrete 2-fibration.  Recall that $\mathfrak E_B$ denote the fiber of $P$ above $B\in\mathfrak B_0$ consisting of the objects, arrows and 2-cells of $\mathfrak E$ above $B$ and the various identities associated to it.

\begin{lemma}  Each fiber $\mathfrak E_B$ is an ordinary category.
\end{lemma}
\begin{proof}  If $\theta\colon u\Rightarrow v$ is a 2-cell between arrows $u,v\colon X\rightrightarrows Y$ of $\mathfrak E_B$, then $P\theta = 1_{1_B}$ holds by definition so that since $P$ is locally a discrete opfibration $\theta$ must be $1_{u}$. \end{proof}

\begin{construction}[Pseudo-Inverse]  Define correspondences $F_P\colon \mathfrak B^{op}\to\mathbf{Cat}$ amounting to a 2-functor in the following way.  First take on objects
\[ F_PB:=\mathfrak E_B
\]
namely, the fiber category of $P$ over $B\in\mathfrak B_0$.  A morphism $f\colon B\to C$ of $\mathfrak B$ defines a ``transition functor" $f^*\colon \mathfrak E_C\to\mathfrak E_B$ in the following way.  First $f^*X$ for $X\in EC_0$ is the domain of the chosen cartesian morphism over $f$, namely, the arrow of $\mathfrak E$
\[ \phi(f,X)\colon f^*X\to X
\]
specified by the splitting $\phi$.  The arrow assignment and 2-cell assignments for $f^*$ are given by the 1- and 2-dimensional lifting properties of the chosen cartesian arrows coming with the splitting.  Such $f^*$ is a functor by uniqueness of lifts.  Finally, given $\alpha\colon f\Rightarrow g$ in $\mathfrak B$, there is an associated transition 2-cell $\alpha^*\colon f^*\Rightarrow g^*$ given in the following way.  The component on $X\in EC_0$ is the dashed arrow in the diagram
$$\begin{tikzpicture}
\node(1){$f^*X$};
\node(2)[node distance=1.4in, right of=1]{$X$};
\node(3)[node distance=1in, below of=1]{$g^*X$};
\node(4)[node distance=1in, below of=2]{$X.$};
\node(5)[node distance=.7in, right of=1]{$$};
\node(6)[node distance=.3in, below of=5]{$\Downarrow\tilde\alpha$};
\draw[->](1) to node [above]{$\phi(f,X)$}(2);
\draw[->, dashed](1) to node [left]{$(\alpha^*)_X$}(3);
\draw[->](2) to node [right]{$1$}(4);
\draw[->](3) to node [below]{$\phi(g,X)$}(4);
\draw[->,bend right=80](1) to node [below]{$\alpha_!\phi(f,X)$}(2);
\end{tikzpicture}$$
The 2-cell $\tilde\alpha$ is the chosen 2-cell $\phi(\alpha, \phi(f,X))$ coming with the local opsplitting; since its target is over $g$, the dashed arrow making a commutative triangle exists by the 1-cell lifting property of $\phi(g,X)$.   Notice that this means the equation
\begin{equation}
\label{eqn defn target lifted 2-cell} \phi(g,X)\alpha^*_X=\alpha_!\phi(f,X)
\end{equation}  
holds.  This choice of component over $X\in EC_0$ defines a natural transformation, as required, by preservation properties of the splitting and uniqueness of lifts.  This is the hard part in the constructions of \cite{Buckley}.  However, it is made straightforward here by the uniqueness properties of the discrete opfibrations.
\end{construction}

\begin{lemma} \label{2-naturality proof} Each 2-cell $\alpha\colon f\Rightarrow g$ of $\mathfrak B$ induces a 2-natural transition 2-cell $\alpha^*\colon f^*\Rightarrow g^*$ between the transition 2-functors $f^*, g^* \colon \mathfrak E_C\to\mathfrak E_B$.
\end{lemma}
\begin{proof}  Ordinary naturality follows by an equality of 2-cells as in the diagram
$$\begin{tikzpicture}
\node(1){$f^*X$};
\node(2)[node distance=1in, right of=1]{$X$};
\node(3)[node distance=.8in, below of=1]{$f^*Y$};
\node(4)[node distance=.8in, below of=2]{$Y$};
\node(5)[node distance=.8in, below of=3]{$g^*Y$};
\node(6)[node distance=.8in, below of=4]{$Y$};
\node(7)[node distance=.5in, right of=1]{$$};
\node(8)[node distance=1.2in, below of=7]{$\Downarrow\tau$};
\node(9)[node distance=2in, right of=2]{$f^*X$};
\node(10)[node distance=1in, right of=9]{$X$};
\node(11)[node distance=.8in, below of=9]{$g^*X$};
\node(12)[node distance=.8in, below of=10]{$X$};
\node(13)[node distance=.8in, below of=11]{$g^*Y$};
\node(14)[node distance=.8in, below of=12]{$Y$};
\node(15)[node distance=.5in, right of=9]{$$};
\node(16)[node distance=.4in, below of=15]{$\Downarrow \eta$};
\node(17)[node distance=1in, right of=4]{$=$};
\draw[->](1) to node [above]{$\phi(f,X)$}(2);
\draw[->](1) to node [left]{$f^*u$}(3);
\draw[->](2) to node [right]{$u$}(4);
\draw[->](3) to node [above]{$\phi(f,Y)$}(4);
\draw[->](3) to node [left]{$\alpha^*_Y$}(5);
\draw[->](4) to node [right]{$1$}(6);
\draw[->](5) to node [below]{$\phi(g,Y)$}(6);
\draw[->](9) to node [above]{$\phi(f,X)$}(10);
\draw[->](9) to node [left]{$\alpha^*_X$}(11);
\draw[->](10) to node [right]{$1$}(12);
\draw[->](11) to node [below]{$\phi(g,X)$}(12);
\draw[->](11) to node [left]{$g^*u$}(13);
\draw[->](12) to node [right]{$u$}(14);
\draw[->](13) to node [below]{$\phi(g,Y)$}(14);
\end{tikzpicture}$$
where $\tau=\phi(\alpha,\phi(f,Y))$ and $\eta = \phi(\alpha,\phi(f,X))$.  That the diagrams are equal follows from the fact that $\tau\ast f^*u$ and $u\ast\eta$ are both lifts of $\alpha$ with domain $u\phi(f,X)$.  But this means that the targets must be equal.  Since $\phi(g,Y)$ is cartesian, the equation $\alpha^*_Yf^*u= g^*u\alpha^*_X$ must hold, as required.  \end{proof}

\begin{prop}  The assignments for $F_P\colon \mathfrak B^{op}\to\mathbf{Cat}$ make it a 2-functor.
\end{prop}
\begin{proof}  The assignments are well-defined by inspection and Lemma \ref{2-naturality proof} above.  That $F_P$ is functorial on 1-cells follows by the same argument in \S 2.4.4 on the bottom of p. 1042 of \cite{Buckley}.  Vertical composition of 2-cells $\alpha\colon f\Rightarrow g$ and $\beta\colon g\Rightarrow h$ between 1-cells $B\to C$ follows from a purported equality of 2-cells as in the diagram
$$\begin{tikzpicture}
\node(1){$f^*X$};
\node(2)[node distance=1in, right of=1]{$X$};
\node(3)[node distance=.8in, below of=1]{$g^*X$};
\node(4)[node distance=.8in, below of=2]{$Y$};
\node(5)[node distance=.8in, below of=3]{$h^*X$};
\node(6)[node distance=.8in, below of=4]{$X$};
\node(7)[node distance=.5in, right of=1]{$$};
\node(8)[node distance=.4in, below of=7]{$\Downarrow\eta$};
\node(9)[node distance=2.5in, right of=8]{$f^*X$};
\node(10)[node distance=1in, right of=9]{$X$};
\node(11)[node distance=.8in, below of=9]{$h^*X$};
\node(12)[node distance=.8in, below of=10]{$X$};
\node(13)[node distance=.8in, below of=8]{$\Downarrow\tau$};
%\node(14)[node distance=.8in, below of=12]{$Y$};
\node(15)[node distance=.5in, right of=9]{$$};
\node(16)[node distance=.4in, below of=15]{$\Downarrow\rho$};
\node(17)[node distance=1in, right of=4]{$=$};
\draw[->](1) to node [above]{$\phi(f,X)$}(2);
\draw[->](1) to node [left]{$\alpha^*_X$}(3);
\draw[->](2) to node [right]{$1$}(4);
\draw[->](3) to node [above]{$\phi(g,X)$}(4);
\draw[->](3) to node [left]{$\beta^*_X$}(5);
\draw[->](4) to node [right]{$1$}(6);
\draw[->](5) to node [below]{$\phi(h,X)$}(6);
\draw[->](9) to node [above]{$\phi(f,X)$}(10);
\draw[->](9) to node [left]{$(\beta\alpha)^*_X$}(11);
\draw[->](10) to node [right]{$1$}(12);
\draw[->](11) to node [below]{$\phi(h,X)$}(12);
\end{tikzpicture}$$
since $\phi(h,X)$ is cartesian.  But the 2-cells on either side of the equality are in fact equal again by the fact that $P$ is locally a discrete opfibration and the cells on either side are over the vertical composite $\beta\alpha$ with the same source $\phi(f,X)$.

Horizontal composition of 2-cells is also respected.  Given $\alpha\colon  f\Rightarrow g$ and $\beta\colon h\Rightarrow k$ with $f,g\colon A\rightrightarrows B$ and $h,k\colon B\rightrightarrows C$, it needs to be seen that $\beta^*\ast\alpha^*=(\beta\ast\alpha)^*$ holds.  Checking on a component at $X\in \mathfrak E_C$, this follows from an equality of 2-cells as in
$$\begin{tikzpicture}
\node(1){$f^*h^*X$};
\node(2)[node distance=1.2in, right of=1]{$h^*X$};
\node(3)[node distance=.8in, below of=1]{$g^*h^*X$};
\node(4)[node distance=.8in, below of=2]{$h^*X$};
\node(5)[node distance=.8in, below of=3]{$g^*k^*X$};
\node(6)[node distance=.8in, below of=4]{$k^*X$};
\node(7)[node distance=.6in, right of=1]{$$};
\node(8)[node distance=.4in, below of=7]{$\Downarrow\eta$};
\node(9)[node distance=3.5in, right of=8]{$f^*X$};
\node(10)[node distance=1in, right of=9]{$X$};
\node(11)[node distance=.8in, below of=9]{$h^*X$};
\node(12)[node distance=.8in, below of=10]{$X$};
\node(13)[node distance=.8in, below of=8]{$$};
\node(14)[node distance=1.1in, right of=13]{$\Downarrow\tau$};
\node(15)[node distance=.5in, right of=9]{$$};
\node(16)[node distance=.4in, below of=15]{$\Downarrow\rho$};
\node(17)[node distance=1.8in, right of=4]{$=$};
\node(18)[node distance=1in, right of=4]{$X$};
\node(19)[node distance=.8in, below of=18]{$X$};
\draw[->](1) to node [above]{$\phi(f,h^*X)$}(2);
\draw[->](1) to node [left]{$\alpha^*_{h^*X}$}(3);
\draw[->](2) to node [right]{$1$}(4);
\draw[->](3) to node [above]{$\phi(g,h^*X)$}(4);
\draw[->](3) to node [left]{$g^*\beta^*_X$}(5);
\draw[->](4) to node [right]{$\beta^*_X$}(6);
\draw[->](5) to node [below]{$\phi(h,k^*X)$}(6);
\draw[->](9) to node [above]{$\phi(hf,X)$}(10);
\draw[->](9) to node [left]{$(\beta\ast\alpha)^*_X$}(11);
\draw[->](10) to node [right]{$1$}(12);
\draw[->](11) to node [below]{$\phi(kg,X)$}(12);
\draw[->](4) to node [above]{$\phi(h,X)$}(18);
\draw[->](18) to node [right]{$1$}(19);
\draw[->](6) to node [below]{$\phi(k,X)$}(19);
\end{tikzpicture}$$
by the splitting equation $\phi(k,X)\phi(\alpha,\phi(f,h^*X) = \phi(kg,X)$.  But again the composites on each side are over $\beta\ast \alpha$ with the same source by the splitting equations.  The rest of the proof follows the pattern of straightforward adaptations of the argument of \S 2.4.4 on pp. 1042-43 in \cite{Buckley} using the various splitting assumptions.  
\end{proof}

\subsubsection{Main Theorem}

The pseudo-inverse from the previous development is what is required to prove the main theorem, namely, that discrete 2-fibrations over a base $\mathfrak B$ correspond to contravariant category-valued 2-functors on the base.  That the correspondence $\elt(-)$ is essentially surjective up to isomorphism can first be seen using lax comma squares.

\begin{construction} \label{lax comma square construction} Let $P\colon \mathfrak E\to\mathfrak B$ denote a discrete 2-fibration.  Construct a lax natural transformation
$$\begin{tikzpicture}
\node(1){$\mathfrak E^{op}$};
\node(2)[node distance=1in, right of=1]{$\mathfrak B^{op}$};
\node(3)[node distance=.8in, below of=1]{$\mathbf 1$};
\node(4)[node distance=.8in, below of=2]{$\cat$};
\node(5)[node distance=.5in, right of=1]{$$};
\node(6)[node distance=.33in, below of=5]{$\lambda$};
\node(7)[node distance=.46in, below of=5]{$\Rightarrow$};
\draw[->](1) to node [above]{$P^{op}$}(2);
\draw[->](1) to node [left]{$$}(3);
\draw[->](2) to node [right]{$F_P$}(4);
\draw[->](3) to node [below]{$\ast$}(4);
\end{tikzpicture}$$
in the following way.  The component corresponding to $X\in\mathfrak E_0$ is given by $X\colon \mathbf 1\to \mathfrak E_{PX}$.  Given $u\colon X\to Y$ of $\mathfrak E$, the corresponding lax naturality square is in fact a triangle as a left in the display
$$\begin{tikzpicture}
\node(1){$\mathfrak E_{PX}$};
\node(2)[node distance=.4in, below of=1]{$$};
\node(3)[node distance=.4in, below of=2]{$\mathfrak E_{PY}$};
\node(4)[node distance=1.2in, left of=2]{$\mathbf 1$};
\node(5)[node distance=.4in, left of=2]{$\Downarrow \tilde u$};
\node(6)[node distance=1.8in, right of=3]{$P(u)^*Y$};
\node(7)[node distance=1.4in, right of=6]{$Y$};
\node(8)[node distance=.8in, above of=6]{$X$};
\draw[->](4) to node [above]{$X$}(1);
\draw[->](3) to node [right]{$P(u)^*$}(1);
\draw[->](4) to node [below]{$Y$}(3);
\draw[->](6) to node [below]{$\phi(Pu,Y)$}(7);
\draw[->,dashed](8) to node [left]{$\tilde u$}(6);
\draw[->](8) to node [above]{$u$}(7);
\end{tikzpicture}$$ 
where $\tilde u$ is the unique lift of $1$ in $\mathfrak E_{PX}$ appearing on the right.  That such data does constitute a lax natural transformation follows by the construction of $F_P$ and by the assumed splitting equations for $P$.
\end{construction}

\begin{prop}  For any discrete 2-fibration $P\colon \mathfrak E\to\mathfrak B$, the 2-functor $F_P\colon \mathfrak B^{op}\to\cat$ fits into a lax comma square of the form
$$\begin{tikzpicture}
\node(1){$\mathfrak E^{op}$};
\node(2)[node distance=1in, right of=1]{$\mathfrak B^{op}$};
\node(3)[node distance=.8in, below of=1]{$\mathbf 1$};
\node(4)[node distance=.8in, below of=2]{$\cat$};
\node(5)[node distance=.5in, right of=1]{$$};
\node(6)[node distance=.33in, below of=5]{$\lambda$};
\node(7)[node distance=.46in, below of=5]{$\Rightarrow$};
\draw[->](1) to node [above]{$P^{op}$}(2);
\draw[->](1) to node [left]{$$}(3);
\draw[->](2) to node [right]{$F_P$}(4);
\draw[->](3) to node [below]{$\ast$}(4);
\end{tikzpicture}$$
with $\lambda$ as in Construction \ref{lax comma square construction} above.
\end{prop}
\begin{proof}  One could verify the three aspects of the universal property of the lax comma object.  Alternatively, from the explicit description of $\ast/F_P$ it is straightforward to see that it is isomorphic as a 2-category to $\mathfrak E^{op}$ over $\mathfrak B^{op}$ via projections.  The slight subtleties come in at the level of morphisms and 2-cells, but the assignments are given by the cartesian lifting properties enjoyed by the distinguished morphisms and 2-cells of $\mathfrak E$, with uniqueness ensuring that the assignments are functorial and bijections.  \end{proof}

\begin{cor}  For any discrete 2-fibration $P\colon \mathfrak E\to\mathfrak B$, the canonical map $\mathfrak E\to\elt(F_P)$ is an isomorphism over $\mathfrak B$.  In other words, the assigment $\elt(-)\colon [\mathfrak B^{op},\cat]\to\mathbf{D2Fib}(\mathfrak B)$ is essentially surjective to within isomorphism.
\end{cor}
\begin{proof}  Since $\elt(F_P)^{op}$ and $\mathfrak E^{op}$ both present the lax comma object, they are canonically isomorphic via the map between them induced by the universal property of $\ast/F_P$.  This shows that the domain 2-category of every discrete 2-fibration occurs, up to isomorphism, as the category of elements of some contravariant category-valued 2-functor, meaning that $\elt(-)$ is essentially surjective up to isomorphism.  \end{proof}

\begin{theo}[Representation Theorem for Discrete 2-Fibrations]\label{theorem duality for disc 2-fibrations} For any 2-category $\mathfrak B$, the assignment $\elt(-)$ extends to a 2-functor making an equivalence of 2-categories
\[ [\mathfrak B^{op},\cat] \simeq \mathbf{D2Fib}(\mathfrak B)
\]
between contravariant category-valued 2-functors on $\mathfrak B$ and discrete 2-fibrations over $\mathfrak B$.
\end{theo}
\begin{proof}  The result is proved if it can be shown that $\elt(-)$ extends to a 2-functor that is locally an isomorphism on ``hom-categories," since it is already known that such a 2-functor is essentially surjective on objects up to isomorphism.  Given a 2-natural transformation $\alpha\colon F\Rightarrow G$ of 2-functors $F,G\colon \mathfrak B^{op}\rightrightarrows\cat$, define $\elt(\alpha)\colon \elt(F)\to\elt(G)$ on 0-, 1-, and 2-cells of $\elt(F)$ by
$$\begin{tikzpicture}
\node(1){$(B,X)$};
\node(2)[node distance=.7in, right of=1]{$\Downarrow \theta$};
\node(3)[node distance=.7in, right of=2]{$(C,Y)$};
\node(4)[node distance=.6in, right of=3]{$\mapsto$};
\node(5)[node distance=.6in, right of=4]{$(B,\alpha_BX)$};
\node(6)[node distance=.7in, right of=5]{$\Downarrow \theta $};
\node(7)[node distance=.7in, right of=6]{$(C,\alpha_CY)$};
\draw[->,bend left](1) to node [above]{$(f,u)$}(3);
\draw[->,bend right](1) to node [below]{$(g,v)$}(3);
\draw[->,bend left](5) to node [above]{$(f,\alpha_Bu)$}(7);
\draw[->,bend right](5) to node [below]{$(g,\alpha_Bv)$}(7);
\end{tikzpicture}$$
These are well-defined and functorial by the construction of the 2-categories of elements and the 2-naturality properties of $\alpha$.  Similarly, given a modification $m\colon \alpha\Rrightarrow \beta$ of 2-natural transformations $\alpha,\beta\colon F\Rightarrow G$, define a transformation $\elt(m)\colon \elt(\alpha)\Rightarrow \elt(\beta)$ by taking as component at $(B,X)$ in $\elt(F)$ the arrow
\[ (1,m_{B,X})\colon (B,\alpha_BX)\to(B,\beta_BX)
\]
of $\elt(G)$.  Of course this is 2-natural and vertical over $\mathfrak B$ by construction and the modification property of $m$ and that each $m_B$ is natural.

These assignments on arrows and 2-cells are thus well-defined and 2-functorial by construction.  Moreover they are the same as in \S 2.2.11 of \cite{Buckley} where it is shown that they are indeed bijections.  The same proofs work in the present context.  \end{proof}

\subsubsection{Further Examples}

As a result of Theorem \ref{theorem duality for disc 2-fibrations}, some light can be thrown on several of the examples encountered so far.

\begin{example}
The canonical representable functor $\mathfrak B(-,X)\colon \mathfrak B^{op}\to\mathbf{CAT}$ corresponds to the domain projection from the slice category $d_0\colon \mathfrak B/X\to \mathfrak B$ of Example \ref{domain projection example}.
\end{example}

\begin{example}

The ``hyperdoctrine"
\[ P\colon \mathbf{Cat}^{op}\to\mathbf{CAT}\qquad \C \mapsto [\C^{op},\mathbf{Set}]
\]
corresponds to the discrete 2-fibration $\mathrm{cod}\colon \mathbf{DFib} \to\mathbf{Cat}$ of Example \ref{codomain example}. For there is an equivalence of 2-categories 
\[ \mathbf{DFib}\simeq \mathbf{Elt}(P) 
\]
commuting with the projection morphisms to $\mathbf{Cat}$.

\end{example}

\begin{example}[Cf. Examples \ref{family 2fibn} and \ref{small fam example}]  Fix a 2-category $\mathfrak B$.  Notice that a 2-functor $\C\to \mathfrak B$ from any 1-category is really a functor $\C\to |\mathfrak B|$ since $\C$ has no nontrivial 2-cells.  Thus, consider the 2-functor taking a category to functors into $\mathfrak B$, that is,
\[  [-,\mathfrak B]\colon \mathbf{Cat}^{op}\to\mathbf{Cat}\qquad \C\mapsto [\C, \mathfrak B]
\]
and extended suitably to functors and transformations.  The image of any such category under $ [-,\mathfrak B]$ is of course a strict 1-category.  Thus, the corresponding 2-category of elements
\[  \Pi \colon \elt( [-,\mathfrak B]) \to \mathbf{Cat}
\]
is a discrete 2-fibration.  The total 2-category $\elt( [-,\mathfrak B])$ is essentially $\mathbf{fam}(\mathfrak B)$ from Example \ref{small fam example}.  
\end{example}

\begin{example}  Let $t$ denote a 2-monad in a 2-category $\K$.  The 2-functor
\[ t\mathbf{Alg}(-)\colon \K^{op}\to\mathbf{Cat} \qquad A \mapsto t\mathbf{Alg}(A)
\]
taking $A\in \K_0$ to the category of $t$-algebras with domain $A$ as in \S 3 of \cite{KS} corresponds under the restricted equivalence of Theorem \ref{restricted equivalence theorem} to the discrete 2-fibration
\[ \Pi\colon t\mathbf{Alg} \to\K
\]
of Example \ref{algebra for monad in K discrete example}.
\end{example}

\subsection{2-Fibrations}

Here is recalled the idea of a 2-fibration, following established developments in the literature.  It will be seen that every discrete 2-fibration is a kind of 2-fibration.  Stating the definition involves first a categorification of the definition of ``cartesian morphism," as follows.

\begin{define}[\S 2 \cite{Hermida} and \S 2.2.1 of \cite{Buckley}]  Let $P\colon \mathfrak E\to\mathfrak B$ denote a 2-functor.  An arrow $f\colon A\to B$ of $\mathfrak E$ is \textbf{cartesian} if it satisfies the following two conditions.
\begin{enumerate}
\item Whenever $g\colon C\to B$ is an arrow of $\mathfrak E$ for which there is a morphism $h\colon PC\to PA$ making a commutative triangle in $\mathfrak B$ as on the right 
$$\begin{tikzpicture}
\node(1){$C$};
\node(2)[node distance=.7in, below of=1]{$A$};
\node(3)[node distance=1in, right of=2]{$B$};
\node(4)[node distance=2.4in, right of=1]{$PC$};
\node(5)[node distance=.7in, below of=4]{$PA$};
\node(6)[node distance=1in, right of=5]{$PB$};
\draw[->,dashed](1) to node[left]{$\hat h$}(2);
\draw[->,bend left](1) to node[above]{$g$}(3);
\draw[->](2) to node[below]{$f$}(3);
\draw[->](4) to node[left]{$h$}(5);
\draw[->,bend left](4) to node[above]{$Pg$}(6);
\draw[->](5) to node[below]{$Pf$}(6);
\end{tikzpicture}$$
it follows that there is a unique $\hat h\colon C\to A$ in $\mathfrak E$ with $P\hat h = h$ making a commutative triangle in $\mathfrak E$ as on the left above.
\item Whenever $\theta\colon g\Rightarrow k$ is a 2-cell of $\mathfrak E$ for which there is a 2-cell $\gamma\colon h \Rightarrow l$ of $\mathfrak B$ making a commutative diagram of composed 2-cells in $\mathfrak B$ as at right below
$$\begin{tikzpicture}
\node(1){$C$};
\node(2)[node distance=.8in, below of=1]{$A$};
\node(3)[node distance=1in, right of=2]{$B$};
\node(4)[node distance=2.4in, right of=1]{$PC$};
\node(5)[node distance=.8in, below of=4]{$PA$};
\node(6)[node distance=1in, right of=5]{$PB$};
\node(7)[node distance=.4in, below of=1]{$\Rightarrow$};
\node(8)[node distance=2.4in, right of=7]{$\Rightarrow$};
\node(9)[node distance=.8in, right of=1]{$$};
\node(10)[node distance=.2in, below of=9]{$\Rightarrow$};
\node(11)[node distance=2.4in, right of=10]{$\Rightarrow$};
\draw[->,bend left,dashed](1) to node[right]{$\hat l$}(2);
\draw[->,bend right,dashed](1) to node[left]{$\hat h$}(2);
\draw[->,bend left](1) to node[below]{$g\;$}(3);
\draw[->,bend left=90](1) to node[above]{$k$}(3);
\draw[->](2) to node[below]{$f$}(3);
\draw[->,bend left](4) to node[right]{$l$}(5);
\draw[->,bend right](4) to node[left]{$h$}(5);
\draw[->,bend left](4) to node[below]{$Pg\;\;$}(6);
\draw[->,bend left=90](4) to node[above]{$Pk$}(6);
\draw[->](5) to node[below]{$Pf$}(6);
\end{tikzpicture}$$
it follows that there is a unique lift 2-cell $\hat \gamma\colon \hat h\Rightarrow \hat l$ in $\mathfrak E$ with $P\hat\gamma=\gamma$ making a commutative diagram of composed 2-cells in $\mathfrak E$ as on the left immediately above.
\end{enumerate}
\end{define}

\begin{define}[Cf. Defn. 2.3 of \cite{Hermida}, \S 2.1.6 of \cite{Buckley}] \label{2fibn defn} A 2-functor $P\colon \mathfrak E\to\mathfrak B$ is a \textbf{2-fibration} if
\begin{enumerate}
\item for each arrow $f\colon B\to PE$ of $\mathfrak B$ there is a cartesian arrow $g\colon A\to E$ of $\mathfrak E$ with $Pg=f$;
\item locally $P$ is an opfibration in that each $P_{A,B}\colon \mathfrak E(A,B)\to\mathfrak B(PA,PB)$ is an opfibration;
\item and finally opcartesian 2-cells are closed under horizontal composition.
\end{enumerate}
Such 2-fibrations will always assumed to be split in the sense of Definition \ref{split 2-fibn defn} below.  The (split) 2-fibrations over $\mathfrak B$ are objects of a 3-category $\mathbf{2Fib}(\mathfrak B)$ whose arrows are splitting-preserving 2-functors over $\mathfrak B$, together with vertical 2-natural transformations, and vertical modifications between them. 
\end{define}

\subsubsection{Spliting Equations}

\label{section splitting eqns}

Let $P\colon \mathfrak E\to\mathfrak B$ denote a 2-fibration as in Definition \ref{2fibn defn}.  A \textbf{cleavage} specifies for each $f\colon B\to C$ of the base and each $X\in \mathfrak E_C$ a chosen cartesian arrow 
\[ \phi(f,X)\colon f^*X \to X
\]
over $f$; and for each 2-cell $\alpha\colon Pf\Rightarrow g$ of the base an opcartesian 2-cell
\[ \phi(\alpha, h) \colon h\Rightarrow \alpha_!h
\]
over $\alpha$, each of $\mathfrak E$.  The splitting equations assert that these choices are functorial in a precise way.  These are adaptations of \S 2.1.10 of \cite{Buckley} suited for our conventions.

For composable arrows $f\colon B\to C$ and $g\colon C\to D$ in $\mathfrak B$ and an object $X\in \mathfrak E_D$, the splitting equation for 1-cells is the usual one, namely,
\begin{equation} \label{splitting 1cell}
\phi(gf, X) = \phi(g,X)\phi(f,g^*X).
\end{equation}
For vertically composable 2-cells $\alpha\colon Pf\Rightarrow g$ and $\beta\colon g\Rightarrow h$, the splitting equation is
\begin{equation}
\label{splitting 2cell vert}
\phi(\beta\alpha,f)=\phi(\beta,\alpha_!f)\phi(\alpha,f).
\end{equation}
For horizontally composable 2-cells $\alpha\colon Pf\Rightarrow h$ and $\beta\colon Pg\Rightarrow k$, the splitting equation is
\begin{equation}
\label{splitting 2cell horiz}
\phi(\beta\ast\alpha,gf)=\phi(\beta,g)\ast\phi(\alpha,f).
\end{equation}
Finally, there are the two splitting equations for identities, namely,
\begin{equation} \label{splitting ident 1cells}
\phi(1_PX,X) = 1_X
\end{equation}
\begin{equation}\label{splitting identi 2cells}
\phi(1_{1_f},f)= 1_f
\end{equation}
for any $X\in \mathfrak E_0$ and any $f\in \mathfrak E_1$.

\begin{define} \label{split 2-fibn defn} A 2-fibration $P\colon \mathfrak E\to\mathfrak B$ is \textbf{split} if it is cloven with cleavage $\phi(-,-)$ satisfying equations \ref{splitting 1cell}, \ref{splitting 2cell vert},\ref{splitting 2cell horiz},\ref{splitting ident 1cells}, and \ref{splitting identi 2cells} above.  Let $\mathbf{2Fib}(\mathfrak B)$ denote the 3-category of split 2-fibrations over $\mathfrak B$, splitting-preserving morphisms, transformations with vertical components and vertical modifications.
\end{define}

\subsubsection{Examples}

The concept of a 2-fibration is a generalization of that of a discrete 2-fibration, just as the concept of a fibration generalizes that of a discrete fibration.

\begin{prop}  Every discrete 2-fibration as above is a 2-fibration as in Definition \ref{2fibn defn}.
\end{prop}
\begin{proof}  That there is a required cartesian 1-cell follows since $|E|\colon |\mathfrak E|\to |\mathfrak B|$ is a split fibration; but in particular the 2-cell lifting condition follows because locally $E$ is a discrete opfibration.  That the horizontal composition of opcartesian 2-cells is again opcartesian also follows from the fact that $E$ is locally a discrete opfibration (in particular the uniqueness aspect of the definition).  \end{proof}

\begin{example} In particular, domain projection from the lax slice of a 2-category
\[ \dom\colon \mathfrak C/X\to\mathfrak C
\]
is a 2-fibration in the sense of Definition \ref{2fibn defn}.  This is the 2-categorical analogue of domain projection from the usual slice $\dom\colon \C/X\to\C$ of a 1-category $\C$.
\end{example}

\begin{example}[Source 2-Fibration]  The source 2-functor from squares in a 2-category
\begin{equation} \label{domain fibn from squares} \src \colon \mathbf{Sq}(\mathfrak C)\to\mathfrak C
\end{equation}
from Example \ref{2cat of squares example} is a 2-fibration in the present sense of Definition \ref{2fibn defn}.  This is a lax 2-categorical analogue of the usual domain fibration $\mathrm{dom}\colon \C^{\mathbf 2}\to\C$ for an ordinary category $\C$ whose fibers are the coslice categories $X/\C$.  For the fibers of \ref{domain fibn from squares} are basically the lax coslice 2-categories $X/\mathfrak C$. 
\end{example}

\begin{example}[Target 2-Fibration]  Let $\mathfrak B$ denote a 2-category.  Again consider $\mathbf{Sq}(\mathfrak B)$ from Example \ref{2cat of squares example}, but this time also the target 2-functor
\begin{equation} \tgt\colon\mathbf{Sq}(\mathfrak B)\to \mathfrak B.
\end{equation}
Suppose that $\mathfrak B$ has (stirct) comma squares (see \S 1 of \cite{StreetFibrations} for example).  The target 2-functor is then a 2-fibration as in Definition \ref{2fibn defn}.  The presence of comma squares suffice for constructing both the cartesian 1-cells and the opcartesian 2-cells locally.  This 2-functor $\tgt$ is the analogue of the codomain fibration $\cod\colon \C^{\mathbf 2}\to\C$ whenever $\C$ is a category with pullbacks.  For the fibers of $\tgt$ are the lax slices $\mathfrak B/X$.
\end{example}

\begin{example}[Category-Indexed Families; Cf. \S 2.3.1 of \cite{Buckley}] \label{family 2fibn} Let $\mathbf{Fam}(\mathfrak B)$ denote the 2-category of ``families" in a 2-category $\mathfrak B$.  That is, the objects are pseudo-functors $F\colon \C\to\mathfrak B$ from small 1-categories $\C$.  Arrows $(\C,F)\to (\D,G)$ are pairs $(H,\alpha)$ where $H\colon \C\to \D$ is a functor and $\alpha$ is a pseudo-natural transformation $\alpha\colon F\Rightarrow GH$.  Finally 2-cells are just appropriate pairs $(\sigma, m)$ of a pseudo-natural transformation $\sigma$ and a modification $m$ as in \S 2.3.1 of \cite{Buckley} adapted for our co- rather than the contravariant families of the reference.  The projection
\[ \Pi\colon \mathbf{Fam}(\mathfrak B)\to\mathbf{Cat}
\]
is then a 2-fibration in the sense of Definition \ref{2fibn defn}.  The proof is essentially the same as in \S 2.3.2 of \cite{Buckley} with suitable adaptations for the fact that families are here covariant rather than covariant pseudo-functors $\C^{op}\to\mathfrak B$ as in the reference. 
\end{example}

\begin{example}  Let $\mathbf{Opf}$ denote the 2-category of opfibrations, appropriate pairs of functors and certain pairs of 2-cells.  The 2-functor
\[ \mathrm{cod}\colon \mathbf{Opf}\to \mathbf{Cat}
\]
is a 2-fibration in the present sense of Definition \ref{2fibn defn}.  For opfibrations are stable under pullback.  This yields the desired cartesian 1-cell with the proper 1- and 2-cell lifting properties.  And locally the 2-functor $\mathrm{cod}$ is an opfibration because the objects of the 2-category are themselves opfibrations, which means that 2-cells lift appropriately to give the required opcartesian 2-cells locally.  Notice that this doesn't appear to require that the top functor of any morphism of $\mathbf{Opf}$ preserves opcartesian morphisms, but we ask for this anyway, so that $\mathrm{cod}$ corresponds under the 3-equivalence to the 2-functor 
\[ \mathbf{Cat}^{op}\to \mathbf{2Cat} \qquad \C\mapsto \mathbf{Opf}(\C)
\]
since the morphisms of $\mathbf{Opf}(\C)$ are generally taken to be opcartesian-morphism-preserving.
\end{example}

\begin{example}[Cf. \S 2.3.12 of \cite{Buckley}] \label{algebras for monads in K example} Take a 2-category $\K$.  Let $\mathbf{Mnd}(\K)$ denote the 2-category of monads in $\K$ and $\mathbf{Alg}_{oplax}(\K)$ the 2-category of pairs $(S,(A,m))$ with $S$ a 2-monad on $\K$ and $(A,m)$ an $S$-algebra; and with morphisms pairs consisting of a 2-monad morphism and an oplax morphisms of algebras; and appropriate 2-cells.  The forgetful 2-functor
\[ \Pi\colon \mathbf{Alg}_{oplax}(\K)\to \mathbf{Mnd}(\K)
\]
is a 2-fibration in the sense of Definition \ref{2fibn defn}.
\end{example}

\begin{example} \label{algebra for a monad in K example} In a related vein, the forgetful 2-functor
\[ \Pi\colon t\mathbf{Alg} \to\K
\]
from $t$-algebras as in Example \ref{algebra for monad in 2cat eg} is a 2-fibration as in Definition \ref{2fibn defn}, since it is a discrete 2-fibration.
\end{example}

The next construction provides more examples and sheds light on both of those above.  Higher 2-categories of elements originated in Bird's thesis \cite{BirdThesis} and appeared later in \S 1,2.5 of \cite{GrayFormalCats}.  The appropriate version for 2-category-valued functors is the following.

\begin{construction}[Adapted from \S 2.2.1 \cite{Buckley}] \label{2cat of elts} For any 3-functor $F\colon \mathfrak B^{op}\to\twocat$ on a 2-category $\mathfrak B$, the \textbf{2-category of elements} of $E$ is the 2-category whose
\begin{enumerate}
\item objects are pairs $(B,X)$ with $B\in \mathfrak B_0$ and $X\in FB$;
\item arrows are pairs $(f,u)\colon (B,X)\to (C,Y)$ with $f\colon B\to C$ in $\mathfrak B$ and $u\colon X\to f^*Y$ in the fiber $FB$;
\item and whose 2-cells $\colon (f,u)\Rightarrow (g,v)$ are those pairs $(\alpha,\sigma)$ where $\alpha\colon f\Rightarrow g$ is in $\mathfrak B$ and $\sigma$ is a 2-cell
$$\begin{tikzpicture}
\node(1){$f^*Y$};
\node(2)[node distance=.4in, below of=1]{$$};
\node(3)[node distance=.4in, below of=2]{$g^*Y$};
\node(4)[node distance=1.2in, left of=2]{$X$};
\node(5)[node distance=.4in, left of=2]{$\Downarrow \sigma$};
\draw[->](4) to node [above]{$u$}(1);
\draw[->](1) to node [right]{$\alpha^*_Y$}(3);
\draw[->](4) to node [below]{$v$}(3);
\end{tikzpicture}$$ 
of the 2-category $FB$.
\end{enumerate}
Denote this 2-category by $\elt(E)$.  There is an evident projection 2-functor $\Pi\colon \elt(E)\to \mathfrak B$.
\end{construction}

The work of \S 2 of \cite{Buckley} consists in showing that the category of elements construction above extends to a 3-functor $\elt(-)\colon \mathbf{2Fib}(\mathfrak B)\to [\mathfrak B^{coop},\mathbf{2Cat}]$ with suitable pseudo-inverse making an equivalence of 3-categories.  The following result makes good on the desiderata from the introduction.

\begin{theo}[Restricted Equivalence] \label{restricted equivalence theorem}  For any 2-category $\mathfrak B$, the 3-equivalence
\[ [\mathfrak B^{op},\mathbf{2Cat}] \simeq \mathbf{2Fib}(\mathfrak B)
\]
restricts to one
\[ [\mathfrak B^{op},\mathbf{Cat}] \simeq \mathbf{D2Fib}(\mathfrak B)
\]
pseudo-naturally in $\mathfrak B$.
\end{theo}
\begin{proof}  Although the 2-cell conditions differ, it can be seen from the proofs of \S 2 in \cite{Buckley} that the 3-equivalence   
\[ [\mathfrak B^{op},\mathbf{2Cat}] \simeq \mathbf{2Fib}(\mathfrak B)
\]
does hold.  Note that the two 2-category of elements constructions for a 2-functor $F\colon \mathfrak B^{op}\to\cat$ are the same whether $F$ is viewed as a 2-functor (Construction \ref{2cat of elts discrete case}) or as a degenerate 3-functor (Construction \ref{2cat of elts}) via the inclusion $\cat\to\mathbf{2Cat}$.  \end{proof}

\begin{remark}  The result is thus that the 3-equivalence restricts to a 2-equivalence making a commutative square
$$\begin{tikzpicture}
\node(1){$\mathbf{2Fib}(\mathfrak B)$};
\node(2)[node distance=1.2in, right of=1]{$[\mathfrak B^{op},\mathbf{2Cat}]$};
\node(3)[node distance=.7in, below of=1]{$\mathbf{D2Fib}(\mathfrak B)$};
\node(4)[node distance=.7in, below of=2]{$[\mathfrak B^{op},\mathbf{Cat}]$};
\node(5)[node distance=.5in, right of=1]{$$};
\node(6)[node distance=.4in, below of=5]{$$};
\draw[->](1) to node [above]{$\simeq$}(2);
\draw[->](3) to node [left]{$\mathrm{incl}$}(1);
\draw[->](4) to node [right]{$\mathrm{incl}$}(2);
\draw[->](3) to node [below]{$\simeq$}(4);
\end{tikzpicture}$$
making good on Desiderata \ref{desiderata restrict equiv} from the introduction.
\end{remark}

\section{Monadicity}

In this section, the other main result of the paper is given in Theorem \ref{MAIN THM2 Monadicity}, namely, that discrete 2-fibrations are monadic over a slice of $\cat$.  This should be seen as a 2-categorical analogue of the well-known result that ordinary discrete fibrations are monadic over a slice of $\mathbf{Set}$.  This result is reviewed in the first subsection, along with its extension to an analogous result for fibrations.  The main theorem appears in the second subsection, while the final subsection is a meditation on the possibility of extending the monadicity result to 2-fibrations.

\label{section: monadicity}

\subsection{(Discrete) Fibrations and Monadicity}

\label{classical case of monadicity}

The theory of ordinary monads is well-known and recounted in \cite{MacLane}, for example.  The main example of a ``monadic" category is, for us, the category of discrete fibrations over a fixed base category. For this, let $\C$ denote an ordinary small category.  Define a functor $T\colon \mathbf{Set}/\C_0\to\mathbf{Set}/\C_0$ by taking $f\colon X\to \C_0$ to the projection from the pullback
$$\begin{tikzpicture}
\node(1){$\C_1\times_{\C_0}X$};
\node(2)[node distance=1in, right of=1]{$X$};
\node(3)[node distance=.7in, below of=1]{$\C_1$};
\node(4)[node distance=.7in, below of=2]{$\C_0$};
\node(5)[node distance=.25in, right of=1]{$$};
\node(6)[node distance=.2in, below of=5]{$\lrcorner$};
\draw[->](1) to node [above]{$\pi_2$}(2);
\draw[->](1) to node [left]{$d_1^*f$}(3);
\draw[->](2) to node [right]{$f$}(4);
\draw[->](3) to node [below]{$d_1$}(4);
\end{tikzpicture}$$
composed with the domain arrow $d_0\colon \C_1\to\C_0$.  Thus, in other words, define $T= d_0\circ d_1^*(f)$.  The arrow assignment is induced by the universal property of the pullback.  So defined, $T$ is an ordinary monad on $\mathbf{Set}/\C_0$.  Now, if $F\colon \F\to\C$ is a discrete fibration, define an action
\[ M\colon \C_1\times_{\C_0} \F_0\to \F_0 \qquad (f,x) \mapsto f^*x
\]
by taking $(f,x)$ with $d_1f=fx$ to the domain of the unique arrow of $\F$ over $f$, denoted by $f^*x$.  This is an action of $\C_1$ on $\F_0$ and makes $F$ into a $T$-algebra.  Universality of the constructions then yields a functor $\mathbf{DFib}(\C) \to T\mathbf{Alg}$ that is one-half of an equivalence.

\begin{theo}  There is an equivalence of categories
\[ \mathbf{DFib}(\C) \simeq T\mathbf{Alg}
\]
for any small category $\C$.
\end{theo}
\begin{proof}  On the other hand, any $T$-algebra $f\colon A\to \C_0$ yields a discrete fibration $F\colon \F\to\C$ by taking $\F_0 = A$ and $\F_1 = TA = \C_1\times_{\C_0}A$.  This extends uniquely to morphisms and gives the pseudo-inverse required for the equivalence.  \end{proof}

The goal is to give an analogous result for discrete 2-fibrations.  To that end it will be helpful to recall the needed preliminaries on 2-monads and some results about ordinary fibrations.  In fact split fibrations have a similar monadicity result but in a ``boosted up" 2-categorical fashion.  The theory of 2-monads probably goes back to Street's \cite{StreetFormalMonads} since it gives a formal theory of such structures internal to any 2-category.  What is needed here for the most part is summarized in \S 3 of \cite{KS}.  Here the material is unpacked for the case $\K= \mathbf{2Cat}$.

\begin{define}[Omnibus Definition] \label{monad} A \textbf{2-monad} on a 2-category $\mathfrak K$ is a 2-functor $T\colon \mathfrak K\to \mathfrak K$ with 2-natural transformations $\eta\colon 1\Rightarrow T$ and $\mu\colon TT\Rightarrow T$ such that $\mu T\mu = \mu \mu T$ and $\mu T\eta =1 = \mu \eta T$ all hold, as usual.  An \textbf{algebra} for such a 2-monad is an object $A\in \K_0$ with a structure map $a\colon TA\to A$ satisfying the usual equations, namely, $a\mu_A = aTa$ and $a\eta_A=1$.  A \textbf{morphism of algebras} $(A,a)$ and $(B,b)$ is a morphism of the 2-category $h\colon A\to B$ that preserves the unit and preserves the action.  A \textbf{2-cell of morphisms of algebras} $h,k\colon A\rightrightarrows B$ is a 2-cell $\theta\colon h\Rightarrow k$ of the 2-category satisfying the compatibility condition $\theta\ast a = b\ast T\theta$.  With these definitions $T\mathbf{Alg}$ denotes the 2-category of $T$-algebras, their morphisms and 2-cells.  A 2-category is \textbf{monadic} over $\mathfrak K$ if it is equivalent to $T\mathbf{Alg}$ for some 2-monad $T$ on $\mathfrak K$.
\end{define}

\begin{example} Consider the 2-monad in the sense of Definition \ref{monad} on $\mathfrak{Cat}/\C$ given by sending a functor $H\colon \X\to\C$ to the pullback $d_1^*H$ as in
$$\begin{tikzpicture}
\node(1){$\C^{\mathbf 2}\times_{\C}\X$};
\node(2)[node distance=1in, right of=1]{$\X$};
\node(3)[node distance=.7in, below of=1]{$\C^{\mathbf 2}$};
\node(4)[node distance=.7in, below of=2]{$\C$};
\node(5)[node distance=.25in, right of=1]{$$};
\node(6)[node distance=.2in, below of=5]{$\lrcorner$};
\draw[->](1) to node [above]{$\pi_2$}(2);
\draw[->](1) to node [left]{$d_1^*H$}(3);
\draw[->](2) to node [right]{$H$}(4);
\draw[->](3) to node [below]{$d_1$}(4);
\end{tikzpicture}$$
composed with the domain functor $d_0\colon \C^{\mathbf 2}\to\C$.  Let $T$ denote this 2-monad.  Split fibrations over $\C$ are precisely the normalized T-algebras as in Definition \ref{monad} for $T$ as above in the sense that there is an equivalence of 2-categories
\[ \mathbf{Fib}(\C)\simeq T\mathbf{Alg}.
\]
If $F\colon \F\to\C$ is a split fibration, the action of $\C^{\mathbf 2}$ on $\F$ is given by 
\[ M\colon \C^{\mathbf 2}\times_{\C}\F \to\F \qquad  (f,X)\mapsto f^*X
\]
on objects and by the dashed arrow solution of the following lifting problem
$$\begin{tikzpicture}
\node(1){$B$};
\node(2)[node distance=1in, right of=1]{$C$};
\node(3)[node distance=.8in, below of=1]{$A$};
\node(4)[node distance=.8in, below of=2]{$D$};
\node(5)[node distance=.5in, right of=1]{$$};
\node(6)[node distance=.4in, below of=5]{$$};
\node(7)[node distance=.6in, right of=2]{$X$};
\node(8)[node distance=.8in, below of=7]{$Y$};
\node(9)[node distance=2in, right of=6]{$\mapsto$};
\node(10)[node distance=2.6in, right of=2]{$f^*X$};
\node(11)[node distance=1in, right of=10]{$X$};
\node(12)[node distance=.8in, below of=10]{$g^*Y$};
\node(13)[node distance=.8in, below of=11]{$Y$};
\draw[->](1) to node [above]{$f$}(2);
\draw[->](1) to node [left]{$h$}(3);
\draw[->](2) to node [right]{$k$}(4);
\draw[->](3) to node [below]{$g$}(4);
\draw[->](7) to node [right]{$u$}(8);
\draw[->](10) to node [above]{$\phi(f,X)$}(11);
\draw[->,dashed](10) to node [left]{$f^*u$}(12);
\draw[->](11) to node [right]{$u$}(13);
\draw[->](12) to node [below]{$\phi(f,Y)$}(13);
\end{tikzpicture}$$
on arrows.  Dually, split opfibrations over $\C$ are precisely the normalized 2-algebras for the 2-monad on $\mathfrak{Cat}/\C$ given by pulling back along $d_0\colon \C^{\mathbf 2}\to\C$ and then composing with $d_1\colon \C^{\mathbf 2}\to\C$.  The correspondence is discussed in \S I,3.5 of \cite{GrayFormalCats}.  A detailed account is in \cite{GrayFiberedCofibered}.  This result led to the definition of fibrations in a 2-category as certain algebras in \S 2 of \cite{StreetFibrations}.
\end{example}

\subsection{Discrete 2-Fibrations}

Now, the main result of the section will be given, namely, that discrete 2-fibrations over a base 2-category $\mathfrak B$ are precisely the algebras for a certain 2-monad on the 2-slice of $\mathbf{Cat}$ by the underlying 1-category $|\mathfrak B|$.  Recall for the following construction that $\mathbf{Sq}(\mathfrak B)$ denotes the 2-category of squares in $\mathfrak B$ from Example \ref{2cat of squares example}.  This was seen to be much like the cotensor with $\mathbf 2$ in the lax 3-category $\mathbf{Lax}$.

\begin{construction} \label{2monad construction}  Define an endo-2-functor 
\begin{equation} \label{2monad 2functor} T\colon \mathbf{Cat}/|\mathfrak B| \to \mathbf{Cat}/|\mathfrak B|
\end{equation}
on the ordinary 2-slice of $\mathbf{Cat}$ by $|\mathfrak B|$.  This is given by pulling back a functor $F\colon \C\to |\mathfrak B|$ along the target map $\mathrm{tgt}\colon |\mathbf{Sq}(\mathfrak B)| \to |\mathfrak B|$ and then composing the resulting projection with the source map $\mathrm{src}\colon |\mathbf{Sq}(\mathfrak B)| \to |\mathfrak B|$.  This makes $T$ a 2-functor in a canonical way since the object assignment is thus defined by pulling back.  There is a 2-natural transformation
\[  \mu\colon T^2\Rightarrow T
\]
given on the component $\mu_F\colon T^2F\to TF$ by horizontal composition of squares crossed with identity or projection from $\C$, as in
\[-\ast-\times 1 \colon |\mathbf{Sq}(\mathfrak B)|\times_{|\mathfrak B|} |\mathbf{Sq}(\mathfrak B)|\times_{|\mathfrak B|} \C \longrightarrow |\mathbf{Sq}(\mathfrak B)|\times_{|\mathfrak B|} \C.
\]
The 2-naturality arguments follow by composition laws.  Similarly, there is a unit 2-natural transformation $\eta\colon 1\Rightarrow T$ given by inserting an identity square for horizontal composition in $\mathbf{Sq}(\mathfrak B)$.
\end{construction}

\begin{remark}  The functor $T$ above could have been defined on the lax slice of $\cat$ over $|\mathfrak B|$ because the pullback $|\sq(\mathfrak B)|\times_{|\mathfrak B}\C$ is the underlying 1-category of the lax comma object $1/F$, which is universal with respect to lax natural, hence 2-natural, transformations.  However, this approach would give the wrong morphisms of algebras, since we have asked for a morphism of discrete 2-fibrations to commute strictly with the projections over the base.
\end{remark}

\begin{lemma}  The 2-functor of Construction \ref{2monad construction} together with $\mu$ and $\eta$ defines a 2-monad.
\end{lemma}
\begin{proof}  The 2-monad axioms are exactly the associativity and unit laws for horizontal composition of squares in $\mathbf{Sq}(\mathfrak B)$ since it suffices to check on components of $\mu$ and $\eta$. \end{proof}

Every discrete 2-fibration gives rise to a $T$ algebra in a canonical and 2-functorial way.

\begin{construction}  Let $P\colon \mathfrak E\to\mathfrak B$ denote a discrete 2-fibration as in Definition \ref{disc 2-fibration defn}.  Define an action functor
\begin{equation} \label{action}
M \colon |\mathbf{Sq}(\mathfrak B)|\times_{|\mathfrak B|} |\mathfrak E| \to |\mathfrak E|
\end{equation}
in the following way.  For an object $(f\colon B\to C,X)$ of the purported domain, by construction of the pullback, $X\in \mathfrak E_C$ holds.  Thus, take
\[ M(f,X):=f^*X
\]
that is, the domain of the chosen cartesian arrow $\phi(f,X)$ over $f$ coming with the splitting.  Now, a morphism of the purported domain is a pair $(\alpha, u)$, which can be displayed as
$$\begin{tikzpicture}
\node(1){$B$};
\node(2)[node distance=1in, right of=1]{$C$};
\node(3)[node distance=.8in, below of=1]{$A$};
\node(4)[node distance=.8in, below of=2]{$D$};
\node(5)[node distance=.5in, right of=1]{$$};
\node(6)[node distance=.4in, below of=5]{$\Downarrow\alpha$};
\node(7)[node distance=1in, right of=2]{$X$};
\node(8)[node distance=.8in, below of=7]{$Y$};
\draw[->](1) to node [above]{$f$}(2);
\draw[->](1) to node [left]{$h$}(3);
\draw[->](2) to node [right]{$k$}(4);
\draw[->](3) to node [below]{$g$}(4);
\draw[->](7) to node [right]{$u$}(8);
\end{tikzpicture}$$
where $Pu = k$, the target of the 2-cell in $\mathfrak B$.  The arrow assignment $M(\alpha, u)$ is then given as a unique lift of $h$ as the dashed arrow in the diagram:
$$\begin{tikzpicture}
\node(1){$f^*X$};
\node(2)[node distance=2in, right of=1]{$X$};
\node(3)[node distance=1in, below of=1]{$g^*Y$};
\node(4)[node distance=1in, below of=2]{$Y$};
\node(5)[node distance=1.3in, right of=1]{$$};
\node(6)[node distance=.4in, below of=5]{$\Downarrow \phi(\alpha,u\phi(f,X))$};
\draw[->](1) to node [above]{$\phi(f,X)$}(2);
\draw[->,dashed](1) to node [left]{$M(\alpha,u)$}(3);
\draw[->](2) to node [right]{$u$}(4);
\draw[->](3) to node [below]{$\phi(g,Y)$}(4);
\draw[->,bend right=20](1) to node [below]{$$}(4);
\end{tikzpicture}$$
The 2-cell $\phi(\alpha,u\phi(f,X))$ is the unique one over $\alpha$ with domain $u\phi(f,X)$.  Thus, its codomain is over $gh$, hence the unique lift giving $M(\alpha,u)$ exists.  That $M$ so defined is functorial follows by uniqueness of the lifted 2-cells since locally $P$ is a discrete opfibration.  One should note, however, that the equations describing the unit and composition laws for $M$ would hold even if $P$ were a 2-fibration as in Definition \ref{2fibn defn}, since the functoriality conditions are precisely the splitting equations for the opcartesian 2-cells from \S \ref{section splitting eqns}.
\end{construction}

\begin{prop}  For any discrete 2-fibration $P\colon \mathfrak E\to\mathfrak B$ as in Definition \ref{disc 2-fibration defn}, the underlying functor of 1-categories $|P|\colon |\mathfrak E|\to |\mathfrak B|$ is an $T$-algebra.  In particular an object assignment
\[ |-|\colon \mathbf{D2Fib}(\mathfrak B) \to T\mathbf{Alg}
\]
is well-defined. 
\end{prop}
\begin{proof}  It suffices to see that the functor in \ref{action} satisfies the algebra laws summarized in Definition \ref{monad}.  But as in the proof of functoriality in the construction of $M$, this results from the uniqueness of lifted 2-cells owing to the fact that $P$ is locally a discrete opfibration.  Again the equations describing this fact would hold even if $P$ were a general 2-fibration.  \end{proof}

The aim now is to prove the following main result, namely, that $|-|$ extends to a 2-functor with suitable pseudo-inverse making a 2-equivalence
\[ \mathbf{D2Fib}(\mathfrak B) \simeq T\mathbf{Alg}
\]
for any 2-category $\mathfrak B$ with $T$ as in Construction \ref{2monad construction}.  In the next subsection, it is shown how to construct an object assignment for the pseudo-inverse.  The arrow and 2-cell assignments are handled in the subsequent subsection.

\subsubsection{Essential Surjectivity on Objects}

This subsection shows that $|-|$ has a pseudo-inverse at the level of objects.  For this, it should first be noted that every $T$-algebra is in fact a split fibration and that $T$-algebra morphisms induce morphisms of split fibrations.

\begin{lemma} \label{action is a fibration} Let $Q\colon \A\to |\mathfrak B|$ denote a $T$-algebra with structure map
\[  M \colon |\mathbf{Sq}(\mathfrak B)|\times_{|\mathfrak B|} \A \to \A
\]
satisfying the appropriate equations.  It follows that $Q$ is a split fibration.
\end{lemma}
\begin{proof}  Take $X\in \A_0$ and a morphism $f\colon B\to QX$ of $\mathfrak B$.  The chosen cartesian arrow above $f$ is
\[ M(=, 1_X)\colon M(f,X)\to X
\]
where the `$=$' symbol denotes the square
$$\begin{tikzpicture}
\node(1){$B$};
\node(2)[node distance=1in, right of=1]{$C$};
\node(3)[node distance=.7in, below of=1]{$C$};
\node(4)[node distance=.7in, below of=2]{$C.$};
\node(5)[node distance=.5in, right of=1]{$$};
\node(6)[node distance=.35in, below of=5]{$=$};
\draw[->](1) to node [above]{$f$}(2);
\draw[->](1) to node [left]{$f$}(3);
\draw[->](2) to node [right]{$1$}(4);
\draw[->](3) to node [below]{$1$}(4);
\end{tikzpicture}$$
Use as notation $f^*X=M(f,X)$ and $\phi(f,X)=M(=,1_X)$.  The lifting property follows from commutativity conditions in $\mathbf{Sq}(\mathfrak B)$.  And that this choice of cartesian arrow above such $f$ provides a splitting for $Q$ follows from the strict algebra equations.  \end{proof}

\begin{lemma}  Any morphism $H\colon \A\to\B$ of $T$-algebras preserves is a morphism of split fibrations.
\end{lemma}
\begin{proof} That $H$ is strictly action preserving and preserves units is equivalent to the statement that $H$ preserves the cartesian arrows as chosen in the previous result.  \end{proof}

\begin{construction}  \label{2functor construction} Let $Q\colon \A\to |\mathfrak B|$ denote a $T$-algebra with structure map
\begin{equation}
M \colon |\mathbf{Sq}(\mathfrak B)|\times_{|\mathfrak B|} \A \to \A
\end{equation}
satisfying the appropriate equations.  Define correspondences $F_Q\colon \mathfrak B^{op}\to\mathbf{Cat}$ in the following way.  On objects take 
\[ F_QB:= \A_B
\]
that is, the fiber of $Q$ over $B$.  Using the fact that $Q$ is a fibration (by Lemma \ref{action is a fibration}), functorial assignments between the fibers $f^*\colon \A_C\to\A_B$ can be given in the usual, way, namely on the one hand by
\[ f^*X:=M(=,1_X)
\]
on objects $X\in \A_C$ where `$=$' is the square in the proof of Lemma \ref{action is a fibration}; and on arrows by the dashed arrow
$$\begin{tikzpicture}
\node(1){$f^*X$};
\node(2)[node distance=1in, right of=1]{$X$};
\node(3)[node distance=.8in, below of=1]{$f^*X$};
\node(4)[node distance=.8in, below of=2]{$Y$};
\node(5)[node distance=.5in, right of=1]{$$};
\node(6)[node distance=.4in, below of=5]{$=$};
\draw[->](1) to node [above]{$\phi(f,X)$}(2);
\draw[->,dashed](1) to node [left]{$f^*u$}(3);
\draw[->](2) to node [right]{$u$}(4);
\draw[->](3) to node [below]{$\phi(f,Y)$}(4);
\end{tikzpicture}$$
arising as a lift of $1$ using the fact that $Q$ is a split fibration.  This is functorial by uniqueness.  For a 2-cell $\alpha\colon f\Rightarrow g$, define components of a transition transformation $\alpha^*\colon f^*\Rightarrow g^*$ by using the action of $M$, namely, the arrows
\[ \alpha^*_X:=M(\alpha,1_X)\colon f^*X\to g^*X
\]
in the category $\A_B$ indexed over $X\in \A_C$.  Naturality follows from the equality of arrows
$$\begin{tikzpicture}
\node(1){$B$};
\node(2)[node distance=1in, right of=1]{$C$};
\node(3)[node distance=.8in, below of=1]{$B$};
\node(4)[node distance=.8in, below of=2]{$C$};
\node(5)[node distance=.5in, right of=1]{$\Downarrow \alpha$};
\node(6)[node distance=.5in, below of=5]{$$};
\node(7)[node distance=.6in, right of=2]{$X$};
\node(8)[node distance=.8in, below of=7]{$Y$};
\node(9)[node distance=2in, right of=2]{$B$};
\node(10)[node distance=1in, right of=9]{$C$};
\node(11)[node distance=.8in, below of=9]{$B$};
\node(12)[node distance=.8in, below of=10]{$C$};
\node(13)[node distance=.5in, right of=9]{$$};
\node(14)[node distance=.8in, below of=13]{$\Downarrow \alpha$};
\node(15)[node distance=.6in, right of=10]{$X$};
\node(16)[node distance=.8in, below of=15]{$Y$};
\node(17)[node distance=1.8in, right of=6]{$=$};
\draw[->,bend left](1) to node [above]{$f$}(2);
\draw[->,bend right](1) to node [below]{$g$}(2);
\draw[->](1) to node [left]{$1$}(3);
\draw[->](2) to node [right]{$1$}(4);
\draw[->](3) to node [below]{$g$}(4);
\draw[->](7) to node [right]{$u$}(8);
\draw[->,bend left](11) to node [above]{$f$}(12);
\draw[->,bend right](11) to node [below]{$g$}(12);
\draw[->](9) to node [above]{$f$}(10);
\draw[->](9) to node [left]{$1$}(11);
\draw[->](10) to node [right]{$1$}(12);
\draw[->](15) to node [right]{$u$}(16);
\end{tikzpicture}$$
in the domain category $|\mathbf{Sq}(\mathfrak B)|\times_{|\mathfrak B|}\A$ using the definition of the components of $\alpha^*$ and the functors $f^*$ and $g^*$ in terms of the action functor $M$.  This $F_Q$ is a well-defined 2-functor.
\end{construction}

\begin{prop}  The assignments for $F_Q\colon \mathfrak B^{op}\to\mathbf{Cat}$ as in Construction \ref{2functor construction} make it a 2-functor.  In particular, this means an object assignment 
\[ \elt(F_{-})\colon T\mathbf{Alg} \to \mathbf{D2Fib}(\mathfrak B)
\]
is well-defined by $Q\mapsto \elt(F_Q)$.
\end{prop}
\begin{proof}  That $F_Q$ is functorial at the level of 1-cells follows since $Q$ has already been seen to be a split fibration.  For vertically composable 2-cells $\alpha\colon f\Rightarrow g$ and $\beta \colon g\Rightarrow h$, that $\beta^*\alpha^* = (\beta\alpha)^*$ holds is the functoriality of $M$ at $X\in \A_C$, that is, the equation
\[ M(\beta,1_X)M(\alpha,1_X) = M(\beta\alpha,1_X),
\]
by definition of the components.  Similarly, horizontal composition of 2-cells is preserved by the algebra associativity axiom.  Preservation of units follow by functoriality and algebra axioms.  \end{proof}

\begin{construction} \label{functor for essen surject} Let $Q\colon \A\to |\mathfrak B|$ denote a $T$-algebra.  The functor $H$ will be one
$$\begin{tikzpicture}
\node(1){$|\elt(F_Q)|$};
\node(2)[node distance=.6in, right of=1]{$$};
\node(3)[node distance=.6in, right of=2]{$\A$};
\node(4)[node distance=.6in, below of=2]{$|\mathfrak B|$};
\draw[->](1) to node [above]{$H$}(3);
\draw[->](1) to node [left]{$\Pi\;\;$}(4);
\draw[->](3) to node [right]{$\;\;Q$}(4);
\end{tikzpicture}$$
with $PH=\Pi$ that respects the actions of $|\mathbf{Sq}(\mathfrak B)|$.  An object of $|\elt(F_Q)|$ is a pair $(B,X)$ with $X\in \A_B$.  So, on objects, take
\[ H(B,X):=X
\]
An arrow is one $(f,u)\colon (B,X)\to (C,Y)$ with $f\colon B\to C$ in $\mathfrak B$ and $u\colon X\to f^*Y$.  Thus, take $H(f,u)$ to be the composite
$$\begin{tikzpicture}
\node(1){$X$};
\node(2)[node distance=.8in, right of=1]{$f^*Y$};
\node(3)[node distance=1in, right of=2]{$Y$};
\draw[->](1) to node [above]{$u$}(2);
\draw[->](2) to node [above]{$\phi(f,Y)$}(3);
\end{tikzpicture}$$
This is a functor by uniqueness properties and splitting equations satisfied by the chosen cartesian arrows $\phi(f,X)$ and the induced lifts.  It needs to be seen that $H$ is an isomorphism of algebras.  That $H$ is a bijection on objects is immediate; and that $H$ is fully faithful is by the existence and uniqueness properties of lifts via chosen cartesian arrows.  It is left to see that $H$ is action-preserving.  The following lemma gives a key computation used in the proof%.
\end{construction}

\begin{lemma}[Technical Preliminary] \label{technical lemma for iso} Let $Q\colon \A\to |\mathfrak B|$ denote a $T$-algebra.  Consider $|\elt(F_Q)|$ as in Construction \ref{functor for essen surject}.  Given a composable arrows $f\colon B\to C$ and $k\colon C\to D$ and a vertical arrow $u\colon X\to k^*Y$ of $\A_C$, the image arrow $f^*u$ under the transition functor as in Construction \ref{2functor construction}, is precisely $f^*u = M(=, \phi(k,Y)u)$ where the `$=$' here denotes the square
$$\begin{tikzpicture}
\node(1){$B$};
\node(2)[node distance=1in, right of=1]{$C$};
\node(3)[node distance=.7in, below of=1]{$B$};
\node(4)[node distance=.7in, below of=2]{$D$};
\node(5)[node distance=.5in, right of=1]{$$};
\node(6)[node distance=.35in, below of=5]{$=$};
\draw[->](1) to node [above]{$f$}(2);
\draw[->](1) to node [left]{$1$}(3);
\draw[->](2) to node [right]{$k$}(4);
\draw[->](3) to node [below]{$kf$}(4);
\end{tikzpicture}$$
and $\phi(k,Y)$ is the chosen cartesian arrow above $k$ from the proof of Lemma \ref{action is a fibration}.
\end{lemma}
\begin{proof}  By its construction, $f^*u$ fits into a commutative square
$$\begin{tikzpicture}
\node(1){$f^*X$};
\node(2)[node distance=1.2in, right of=1]{$X$};
\node(3)[node distance=.8in, below of=1]{$(kf)^*X$};
\node(4)[node distance=.8in, below of=2]{$Y$};
\node(5)[node distance=.6in, right of=1]{$$};
\node(6)[node distance=.4in, below of=5]{$=$};
\draw[->](1) to node [above]{$\phi(f,X)$}(2);
\draw[->](1) to node [left]{$f^*u$}(3);
\draw[->](2) to node [right]{$\phi(k,Y)u$}(4);
\draw[->](3) to node [below]{$\phi(kf,Y)$}(4);
\end{tikzpicture}$$
by composing the square from Construction \ref{2functor construction} in the definition of $f^*u$ with the morphism $\phi(k,Y)$ and using the splitting equation in \S \ref{splitting 1cell}.  Now, the clockwise way around the square above is given by $M$ applied to the arrow
$$\begin{tikzpicture}
\node(1){$B$};
\node(2)[node distance=1in, right of=1]{$C$};
\node(3)[node distance=.7in, below of=1]{$C$};
\node(4)[node distance=.7in, below of=2]{$C$};
\node(5)[node distance=.5in, right of=1]{$$};
\node(6)[node distance=.35in, below of=5]{$=$};
\node(7)[node distance=.7in, below of=3]{$D$};
\node(8)[node distance=.7in, below of=4]{$D$};
\node(9)[node distance=.7in, below of=6]{$=$};
%\node(10)[node distance=.7in, below of=8]{$D$};
\node(11)[node distance=1in, right of=2]{$X$};
\node(12)[node distance=.7in, below of=11]{$X$};
\node(13)[node distance=.7in, below of=12]{$k^*Y$};
\draw[->](1) to node [above]{$f$}(2);
\draw[->](1) to node [left]{$f$}(3);
\draw[->](2) to node [right]{$1$}(4);
\draw[->](3) to node [below]{$1$}(4);
\draw[->](3) to node [left]{$k$}(7);
\draw[->](4) to node [right]{$k$}(8);
\draw[->](7) to node [below]{$1$}(8);
\draw[->](11) to node [right]{$1$}(12);
\draw[->](12) to node [right]{$\phi(k,Y)u$}(13);
\end{tikzpicture}$$
of $|\mathbf{Sq}(\mathfrak B)|\times_{|\mathfrak B|}|\elt(F_Q)|$.  (That is, since $u= M(1_{1_C},u)$ holds and $M$ is a functor).  On the other hand, the counter-clockwise way around the same square is $M$ applied to the arrow
$$\begin{tikzpicture}
\node(1){$B$};
\node(2)[node distance=1in, right of=1]{$C$};
\node(3)[node distance=.7in, below of=1]{$B$};
\node(4)[node distance=.7in, below of=2]{$D$};
\node(5)[node distance=.5in, right of=1]{$$};
\node(6)[node distance=.35in, below of=5]{$=$};
\node(7)[node distance=.7in, below of=3]{$D$};
\node(8)[node distance=.7in, below of=4]{$D$};
\node(9)[node distance=.7in, below of=6]{$=$};
%\node(10)[node distance=.7in, below of=8]{$D$};
\node(11)[node distance=1in, right of=2]{$X$};
\node(12)[node distance=.7in, below of=11]{$Y$};
\node(13)[node distance=.7in, below of=12]{$Y$};
\draw[->](1) to node [above]{$f$}(2);
\draw[->](1) to node [left]{$1$}(3);
\draw[->](2) to node [right]{$k$}(4);
\draw[->](3) to node [below]{$kf$}(4);
\draw[->](3) to node [left]{$kf$}(7);
\draw[->](4) to node [right]{$1$}(8);
\draw[->](7) to node [below]{$1$}(8);
\draw[->](11) to node [right]{$\phi(k,Y)u$}(12);
\draw[->](12) to node [right]{$1$}(13);
\end{tikzpicture}$$
of $|\mathbf{Sq}(\mathfrak B)|\times_{|\mathfrak B|}|\elt(F_Q)|$.  Of course the two arrows upon which $M$ acts in the last two displays are equal.  Hence by uniquness and functoriality of $M$ the desired equation $f^*u=M(=, \phi(k,Y)u)$ does hold.  \end{proof}

\begin{prop} \label{algebras obj assign essen surj} The object assignment
\[ |-|\colon \mathbf{D2Fib} \to T\mathbf{Alg}
\]
is essentially surjective to within isomorphism.
\end{prop}
\begin{proof}  It has now to be seen that $H$ is a morphism of the 2-category, that is, a homomorphism of $T$-algebras.  To show that $H$ commutes with the action, first note that this follows on objects simply by definition the $f^*X:=M(f,X)$ and the construction of $H$.  It must be checked on arrows, however.  Thus, start with
$$\begin{tikzpicture}
\node(1){$B$};
\node(2)[node distance=1in, right of=1]{$C$};
\node(3)[node distance=.8in, below of=1]{$A$};
\node(4)[node distance=.8in, below of=2]{$D$};
\node(5)[node distance=.5in, right of=1]{$$};
\node(6)[node distance=.4in, below of=5]{$\Downarrow\alpha$};
\node(7)[node distance=1in, right of=2]{$(B,X)$};
\node(8)[node distance=.8in, below of=7]{$(D,Y)$};
\draw[->](1) to node [above]{$f$}(2);
\draw[->](1) to node [left]{$h$}(3);
\draw[->](2) to node [right]{$k$}(4);
\draw[->](3) to node [below]{$g$}(4);
\draw[->](7) to node [right]{$(k,u)$}(8);
\end{tikzpicture}$$
of $|\mathbf{Sq}(\mathfrak B)|\times_{|\mathfrak B|}|\elt(F_Q)|$ with $u\colon X\to k^*Y$.  Chase this around each side of the preservation square.  On the one hand, $H$ following the action on $|\elt(F_Q)|$ yields the arrow
$$\begin{tikzpicture}
\node(1){$f^*X$};
\node(2)[node distance=1.2in, right of=1]{$h^*g^*Y$};
\node(3)[node distance=1.2in, right of=2]{$g^*Y.$};
\draw[->](1) to node [above]{$\alpha^*_Yf^*u$}(2);
\draw[->](2) to node [above]{$\phi(h,g^*Y)$}(3);
\end{tikzpicture}$$
On the other hand, the action on $\A$ following $1\times H$ yields the arrow
$$\begin{tikzpicture}
\node(1){$f^*X$};
\node(2)[node distance=1.6in, right of=1]{$g^*Y$};
\draw[->](1) to node [above]{$M(\alpha,\phi(k,Y)u).$}(2);
\end{tikzpicture}$$
Of course the claim is that these arrows are equal in $\A$.  Indeed each is the image under $M$ of the same arrow of $|\mathbf{Sq}(\mathfrak B)|\times_{|\mathfrak B|}\A$.  The first displayed arrow of $\A$ is the image under $M$ of
$$\begin{tikzpicture}
\node(1){$B$};
\node(2)[node distance=1in, right of=1]{$C$};
\node(3)[node distance=.7in, below of=1]{$B$};
\node(4)[node distance=.7in, below of=2]{$D$};
\node(5)[node distance=.5in, right of=1]{$$};
\node(6)[node distance=.35in, below of=5]{$=$};
\node(7)[node distance=.7in, below of=3]{$B$};
\node(8)[node distance=.7in, below of=4]{$D$};
\node(9)[node distance=.7in, below of=7]{$A$};
\node(10)[node distance=.7in, below of=8]{$D$};
\node(11)[node distance=1in, right of=2]{$X$};
\node(12)[node distance=.7in, below of=11]{$Y$};
\node(13)[node distance=.7in, below of=12]{$Y$};
\node(14)[node distance=.7in, below of=13]{$Y$};
\node(15)[node distance=.7in, below of=6]{$\Downarrow\alpha$};
\node(16)[node distance=.7in, below of=15]{$=$};
\draw[->](1) to node [above]{$f$}(2);
\draw[->](1) to node [left]{$1$}(3);
\draw[->](2) to node [right]{$k$}(4);
\draw[->](3) to node [below]{$kf$}(4);
\draw[->](3) to node [left]{$1$}(7);
\draw[->](4) to node [right]{$1$}(8);
\draw[->](7) to node [below]{$gh$}(8);
\draw[->](11) to node [right]{$\phi(k,Y)u$}(12);
\draw[->](12) to node [right]{$1$}(13);
\draw[->](7) to node [left]{$h$}(9);
\draw[->](8) to node [right]{$1$}(10);
\draw[->](9) to node [below]{$g$}(10);
\draw[->](13) to node [right]{$1$}(14);
\end{tikzpicture}$$
by the computation of $f^*u$ in Lemma \ref{technical lemma for iso}, the further equation $\alpha^*_Y=M(\alpha, 1_Y)$ from Construction \ref{2functor construction}, and the fact that $M$ is a functor.  On the other hand, the second arrow of $\A$ is the image under $M$ of the arrow 
$$\begin{tikzpicture}
\node(1){$B$};
\node(2)[node distance=1in, right of=1]{$C$};
\node(3)[node distance=.7in, below of=1]{$B$};
\node(4)[node distance=.7in, below of=2]{$C$};
\node(5)[node distance=.5in, right of=1]{$$};
\node(6)[node distance=.4in, below of=5]{$=$};
\node(7)[node distance=.7in, below of=3]{$A$};
\node(8)[node distance=.7in, below of=4]{$D$};
\node(9)[node distance=.7in, below of=6]{$\Downarrow\alpha$};
%\node(10)[node distance=.7in, below of=8]{$D$};
\node(11)[node distance=1in, right of=2]{$X$};
\node(12)[node distance=.7in, below of=11]{$k^*Y$};
\node(13)[node distance=.7in, below of=12]{$Y$};
\draw[->](1) to node [above]{$f$}(2);
\draw[->](1) to node [left]{$1$}(3);
\draw[->](2) to node [right]{$1$}(4);
\draw[->](3) to node [below]{$f$}(4);
\draw[->](3) to node [left]{$h$}(7);
\draw[->](4) to node [right]{$k$}(8);
\draw[->](7) to node [below]{$g$}(8);
\draw[->](11) to node [right]{$u$}(12);
\draw[->](12) to node [right]{$\phi(k,Y)$}(13);
\end{tikzpicture}$$
since $u= M(1_{1_C},u)$ holds by the identity law for the algebra.  But the arrows of $|\mathbf{Sq}(\mathfrak B)|\times_{|\mathfrak B|}\A$ in the last two displays are evidently equal by composition laws.  Since $M$ is functor, the images are equal.  Hence $H$ respects the action.  The proof for the identity law is similar but easier.\end{proof}

\subsubsection{Remaining Assignments}

Additionally, the object assignments from the previous subsection extend to well-defined 2-functors 
\[ |-|\colon \mathbf{D2Fib}(\mathfrak B) \rightleftarrows T\mathbf{Alg}\colon |\elt(F_{-})|
\]
which turn out to give an equivalence of 2-categories.  Given a morphism $F$ of discrete 2-fibrations $P\colon \mathfrak E\to\mathfrak B$ and $Q\colon \mathfrak G\to\mathfrak B$, the underying functor $|F|$ is a $T$-algebra homomorphism since it is supposed to preserve the splitting and since both $P$ and $Q$ are locally discrete opfibrations.  Each vertical 2-natural transformation $\alpha\colon F\Rightarrow G$ of such morphisms of discrete 2-fibrations clearly induces an ordinary vertical natural transformation of underlying functors $|\alpha|\colon |F|\Rightarrow |G|$.  Thus, $|-|$ is well-defined and functorial.

\begin{prop} \label{algebra isomorphism remaining assignments}  For any discrete 2-fibrations, $P\colon \mathfrak E\to\mathfrak B$ and $Q\colon \mathfrak G\to\mathfrak B$, the underlying functor
\[ |-|\colon \mathbf{D2Fib}(P,Q) \to T\mathbf{Alg}(|P|,|Q|)
\]
is an isomorphism of 1-categories and is natural in $P$ and $Q$.
\end{prop} 
\begin{proof}  First, the functor is a bijection on objects.  Take a morphism $H\colon |P|\to |Q|$ of algebras.  This extends to a 2-functor $H\colon \mathfrak E\to\mathfrak G$ in the following way.  Take a 2-cell $\alpha\colon f\Rightarrow g\colon B\rightrightarrows C$ of $\mathfrak E$.  By the discrete opfibration condition, there is a 2-cell of $\mathfrak G$
$$\begin{tikzpicture}
\node(1){$HX$};
\node(2)[node distance=.7in, right of=1]{$\Downarrow \tilde \alpha$};
\node(3)[node distance=.7in, right of=2]{$HY$};
\draw[->,bend left](1) to node [above]{$Hf$}(3);
\draw[->,bend right,dashed](1) to node [below]{$\alpha_!f$}(3);
\end{tikzpicture}$$
above $P\alpha$ in $\mathfrak B$.  Clearly $H\alpha := \tilde \alpha$ should be the definition.  But the claim is that $Hg = \alpha_!f$ so that the source and target are respected.  Then $H$ will be a 2-functor by uniqueness assumptions.

Start by factoring $f = \phi(Pf,Y)u$ and $g = \phi(Pg,Y)v$ through their respective chosen cartesian arrows.  By the definition of the action on $|\mathfrak E|$, there is a vertical morphism $w\colon P(f)^*Y \to P(g)^*Y$ and a lift of $\alpha$ in $\mathfrak E$ of the form
$$\begin{tikzpicture}
\node(1){$P(f)^*Y$};
\node(2)[node distance=.4in, below of=1]{$$};
\node(3)[node distance=.4in, below of=2]{$P(g)^*Y$};
\node(4)[node distance=1.2in, right of=2]{$Y.$};
\node(5)[node distance=.4in, right of=2]{$\Downarrow \bar \alpha$};
\draw[->](1) to node [above]{$\;\;\;\;\;\phi(Pf,Y)$}(4);
\draw[->](3) to node [below]{$\;\;\;\;\;\phi(Pg,Y)$}(4);
\draw[->](1) to node [left]{$w$}(3);
\end{tikzpicture}$$ 
Note that $wu=v$ holds by the lifting property of $\phi(Pg,Y)$.  Now, by the fact that $H$ respects the actions on $|\mathfrak E|$ and $|\mathfrak G|$, there results a 2-cell as in the diagram
$$\begin{tikzpicture}
\node(1){$P(f)^*HY$};
\node(2)[node distance=.4in, below of=1]{$$};
\node(3)[node distance=.4in, below of=2]{$P(g)^*HY$};
\node(4)[node distance=1.2in, right of=2]{$HY$};
\node(5)[node distance=.4in, right of=2]{$\Downarrow H\bar\alpha$};
\draw[->](1) to node [above]{$\;\;\;\;\;\;\phi(Pf,HY)$}(4);
\draw[->](3) to node [below]{$\;\;\;\;\;\;\phi(Pg,HY)$}(4);
\draw[->](1) to node [left]{$Hw$}(3);
\end{tikzpicture}$$ 
which is a lift of $\alpha$ with target $\phi(Pg,HY)Hw$.  Therefore, pulling back $H\bar\alpha$ by $Hu$ there is a 2-cell $H\bar\alpha\ast Hu$ whose source is $Hf$.  Thus, by uniqueness, $\tilde \alpha$ must have the same target namely, $\phi(Pg,HY)Hw$.  But this arrow is $Hg$ since $H$ respects the splitting and is a functor.

That the functor is a bijection on arrows follows by the discrete opfibration assumption.  For the components of any morphism of the target (i.e. a 2-cell between morphism of algebras) defines a system between the corresponding 2-functors in the source that must indeed satisfy the 2-cell condition for 2-naturality by construction of the 2-cell assignment above and by the uniqueness clause of the discrete opfibration assumption.  \end{proof}

Now the main result of the section can be given.

\begin{theo}[Discrete 2-Fibrations are 2-Monadic] \label{MAIN THM2 Monadicity} There is a 2-equivalence
\[ \mathbf{D2Fib}(\mathfrak B) \simeq T\mathbf{Alg}
\]
for any 2-category $\mathfrak B$ with $T$ as above in Construction \ref{2monad construction}.
\end{theo}
\begin{proof}  Taken together, Propositions \ref{algebras obj assign essen surj} and \ref{algebra isomorphism remaining assignments} give the result.  \end{proof}

\subsection{Remarks on Monadicity of 2-Fibrations}

The development thus far prompts some reflections upon the possibility of showing that 2-fibrations are monadic as well.  This is really a subject for a separat paper, but some preliminary considerations can be given here.  The first point to notice is that a 2-fibration $P\colon \mathfrak E\to\mathfrak B$ admits an action from the whole 2-category $\sq(\mathfrak B)$ and not just the underlying 1-category.

\begin{construction}  Assignments yielding an action $M\colon \sq(\mathfrak B) \times_{\mathfrak B}\mathfrak E\to\mathfrak E$ are given in the following way.  Again assign
\[ (f\colon B\to PX,X)\mapsto f^*X
\]
on objects.  The arrow assignment uses the opcartesian lift as before.  That is, the pair $(\alpha, u)$ is sent to the dashed arrow as at the right in the diagram
$$\begin{tikzpicture}
\node(1){$B$};
\node(2)[node distance=1in, right of=1]{$PX$};
\node(3)[node distance=1in, below of=1]{$A$};
\node(4)[node distance=1in, below of=2]{$PY$};
\node(5)[node distance=.5in, right of=1]{$$};
\node(6)[node distance=.5in, below of=5]{$\Downarrow\alpha$};
\node(7)[node distance=.6in, right of=2]{$X$};
\node(8)[node distance=1in, below of=7]{$Y$};
\draw[->](1) to node [above]{$f$}(2);
\draw[->](1) to node [left]{$h$}(3);
\draw[->](2) to node [right]{$Pu$}(4);
\draw[->](3) to node [below]{$g$}(4);
\draw[->](7) to node [right]{$u$}(8);
\node(9)[node distance=2in, right of=7]{$f^*X$};
\node(10)[node distance=2in, right of=9]{$X$};
\node(11)[node distance=1in, below of=9]{$g^*Y$};
\node(12)[node distance=1in, below of=10]{$Y$};
\node(13)[node distance=1.3in, right of=9]{$$};
\node(14)[node distance=.4in, below of=13]{$\Downarrow \phi(\alpha,u\phi(f,X))$};
\node(15)[node distance=2in, right of=6]{$\mapsto$};
\draw[->](9) to node [above]{$\phi(f,X)$}(10);
\draw[->,dashed](9) to node [left]{$M(\alpha,u)$}(11);
\draw[->](10) to node [right]{$u$}(12);
\draw[->](11) to node [below]{$\phi(g,Y)$}(12);
\draw[->,bend right=20](9) to node [below]{$$}(12);
\end{tikzpicture}$$
The dashed arrow exists because the target of the chose 2-cell above $\alpha$ is over $gh$ and $\phi(g,Y)$ is of course cartesian over $g$.  The 2-cell assignment can be seen from the display above.  That is, a 2-cell of the domain of $M$ is really completely given by a pair of 2-cells $h\Rightarrow h'$ and $u\Rightarrow u'$.  Since the lifts of $\alpha$ corresponding to $u$ and $u'$ are opcartesian, there will be a unique lift of the composite 2-cell $gh\Rightarrow gh'$ between the targets of the lifts of $\alpha$.  The required 2-cell is then uniquely induced $M(\alpha,u)\Rightarrow M(\alpha',u')$ by the 2-cell lifting property of $\phi(g,Y)$.  This $M$ is an action in the required sense by the assumed splitting equations for $P$.
\end{construction}

The question, then, is whether we should expect every 2-functor $Q\colon \mathfrak A\to\mathfrak B$ to be a 2-fibration.  That this is so is confirmed in the proof of the next result.

\begin{prop}  Let $P\colon \mathfrak E\to\mathfrak B$ denote a 2-functor admitting a strict action $M\colon \sq(\mathfrak B)\times_{\mathfrak B}\mathfrak E\to\mathfrak E$.  It follows then that $P$ is a split 2-fibration.
\end{prop}
\begin{proof}  The required cartesian arrow above a morphism $f\colon B\to PX$ of $\mathfrak B$ is given in the same way as in the proof of Lemma \ref{action is a fibration}.  The 2-cell lifting condition follows from the construction of 2-cells in $\sq(\mathfrak B)$.  What needs to be proved is that $P$ is locally a split opfibration.  Take an arrow $u\colon X\to Y$ of $\mathfrak E$ and a 2-cell $\alpha\colon Pu\Rightarrow g$ of $\mathfrak B$.  Consider the 2-cell of $\sq(\mathfrak B)\times_{\mathfrak B}\mathfrak E$ as indicated by the equality
$$\begin{tikzpicture}
\node(1){$PX$};
\node(2)[node distance=1.4in, right of=1]{$PX$};
\node(3)[node distance=1in, below of=1]{$PY$};
\node(4)[node distance=1in, below of=2]{$PY$};
\node(5)[node distance=.5in, right of=1]{$$};
\node(6)[node distance=.25in, above of=5]{$$};
\node(7)[node distance=.5in, below of=5]{$\Downarrow \alpha$};
\node(8)[node distance=.9in, right of=7]{$\mathclap{\substack{1_{Pu} \\ \Leftarrow }}$};
\node(9)[node distance=2.3in, right of=7]{$=$};
\node(10)[node distance=2.5in, right of=2]{$PX$};
\node(11)[node distance=1.4in, right of=10]{$PX$};
\node(12)[node distance=1in, below of=10]{$PY$};
\node(13)[node distance=1in, below of=11]{$PY$};
\node(14)[node distance=4.3in, right of=7]{$=$};
\node(15)[node distance=.75in, below of=14]{$$};
\node(16)[node distance=.9in, left of=14]{$\mathclap{\substack{\alpha \\ \Leftarrow }}$};
\node(17)[node distance=.5in, right of=2]{$X$};
\node(18)[node distance=1in, below of=17]{$Y$};
\node(19)[node distance=.5in, right of=11]{$X$};
\node(20)[node distance=1in, below of=19]{$Y$};
\draw[->](1) to node [above]{$1_{PX}$}(2);
\draw[->](1) to node [left]{$g$}(3);
\draw[->,bend left](2) to node [right]{$Pu$}(4);
\draw[->,bend right](2) to node [left]{$Pu$}(4);
\draw[->](3) to node [below]{$1_{PY}$}(4);
\draw[->](10) to node [above]{$1_{PX}$}(11);
\draw[->,bend left](10) to node [right]{$Pu$}(12);
\draw[->,bend right](10) to node [left]{$g$}(12);
\draw[->](11) to node [right]{$Pu$}(13);
\draw[->](12) to node [below]{$1_{PY}$}(13);
\draw[->,bend left](17) to node [right]{$u$}(18);
\draw[->](19) to node [right]{$u$}(20);
\end{tikzpicture}$$
The image of this arrow under $M$ in $\mathfrak E$ is the chosen opcartesian 2-cell with domain $u$ above $\alpha$.  That it has the required lifting property follows from the definition of vertical composition of 2-cells in $\sq(\mathfrak B)$ and the fact that $M$ is well-defined.  That horizontal composition preserves opcartesian 2-cells follows from the definition.  \end{proof}

\begin{remark}  Notice that the lax structure encoded in $\sq(\mathfrak B)$ is really needed to give the opcartesian 2-cell in showing that the 2-functor $P\colon \mathfrak E\to\mathfrak B$ is a 2-fibration.  That is, if $P$ only admit an action of the cotensor arrow 2-category, having only commutative squares as its morphisms, then it is not clear that $P$ would be a 2-fibration in the present sense.  The result suggests that split 2-fibrations are 3-algebras for some 3-monad on some slice of 2-categories given by pulling by by $\tgt\colon \sq(\mathfrak B)\to\mathfrak B$, but this will have to wait for a separate treatment.
\end{remark}

\begin{remark}  It seems that with different choices of convention on the combination of `$op$' and `$co$' one would obtain actions of different 2-categories of squares associated to $\mathfrak B$.  It seems not-`$op$' vs. `$op$' is the difference between globally an opfibration vs. globally a fibration, which is also the difference between admitting a right vs. a left action of some kind of cotensor object.  The difference between not-`$co$' vs. `$co$' is that $P$ in the former case is locally a split opfibration vs. on the other hand is locally a split fibration, which seems to be the same as admitting an action from the lax comma square $\sq(\mathfrak B)$ vs. admitting one from the oplax comma square associated to $1_{\mathfrak B}$. 
\end{remark}

\section{Prospectus}

The paper closes here with some speculation about further avenues of inquiry.

\subsection{Internalization}

Returning to the discussion of the introduction, we are left with a difficult question about the approach to be taken toward internalizing the notion of a discrete 2-fibration in some higher topos.  Again the idea is to boost the internalization results for flat set-valued functors achieved by Diaconescu in \cite{DiaconescuChangeOfBase} and \cite{DiaconescuThesis} into the next highest dimension, giving an elementary account of flatness for something like 2- and pseudo-functors.  Discrete fibrations were encoded in dimension 1 as algebras for an action of the cotensor of the base category with $\mathbf 2$.  What has been discovered here, however, is that discrete 2-fibrations are algebras for a the action of a structure $\sq(\mathfrak B)$ that is not a cotensor is the expected venue, namely, $\mathbf{2Cat}$, but in the more fiddly 3-dimensional structure $\mathbf{Lax}$.

The approach that was taken in \cite{LambertThesis} was to axiomatize the idea of a 2-category internal to a given 1-category $\E$.  This included the definition of internal ``hom-categories" making sense of how to discuss concepts defined for internal 2-categories ``locally."  A discrete 2-fibration over a fixed base internal 2-category was then defined to be an internal 2-functor (1) whose underlying internal 1-functor is an algebra for the action of the internal arrow category of the underlying 1-category of the base; and (2) that is locally a discrete opfibration in the sense of being an algebra for an internalization of the action as in the opening of \S \label{classical case of monadicity}.  This approach worked to given an elementary version of the desired flatness results but was technically complicated in a way that fundamentally muddied what should have been a clear and elegant picture.  The question, then, is whether the results of this paper give any insight into the possibility of a more straightforward elementary axiomatization of the setting in which the flatness results should be achieved.

It is not clear that this is the case.  At least what one expects is that, just as the 1-dimensional flatness results were axiomatized in the internal category theory of a topos, the 2-dimensional flatness results should appear internally in some kind of 2-topos.  The 2-fibration concepts introduced in \cite{Hermida} and \cite{Buckley} that were considered here are not representable.  The representable concepts appear in \S I,2.9 of \cite{GrayFormalCats}, where a 2-fibration is defined using a certain ``Chevalley condition" via the cotensor with $\mathbf 2$ of the base.  It is asserted there without proof that such 2-fibrations correspond in a strong way with 2-category-valued 2-functors on the base.  The point is that the representable discrete 2-fibration concept would be the discretization of Gray's notion and again given by an action of the ordinary cotensor.  This would perhaps be the correct notion of discrete fibration to internalize in a 2-topos as in \cite{Weber}, which is already a reasonable well-established categorification of the idea of an ordinary elementary topos.  The issue, however, is again that the fibration concepts considered here are not the representable ones and the action of the cotensor object giving the algebra structure is not that of the cotensor in $\mathbf{2Cat}$, but rather $\mathbf{Lax}$, a much less studied 3-dimesional structure.  It is neither clear how ``2-toposy" $\mathbf{Lax}$ really is nor whether the fibration concepts studied here are the end of the story.

\subsection{Fibration Concepts}

An interesting pattern is suggested by the developments of the paper.  The main result, Theorem \ref{theorem duality for disc 2-fibrations}, and all those results summarized in the introduction, generally speaking, take the folllowing form: correspondences between a structured class of ``geometric data" -- that is, (discrete) fibrations -- on the one hand and a structured class of representations of some gadget in a base structure of which the gadget is either (A) of the same status (i.e. presheaves representing a category in the category of sets) or (B) a member (i.e. pseudo-functors representing a category as parameterized categories).  That is, on the one hand, there is a representing structure -- perhaps some kind of $n$-category -- denoted here by $\mathcal K$ and a higher $(n+1)$-structure, denoted by $\cat(\mathcal K)$ of which $\mathcal K$ is a member (think $\mathbf{Set}$ and $\cat$ or $\cat$ and $\twocat$).  There is a represented object of $\cat(\mathcal K)$ -- some $n$-category -- denoted by $\mathbb B$.  There is a ``hom-object" in $\cat(\mathcal K)$ of the same overall structure as $\mathcal K$ of representations of $\mathbb B$, denoted by $[\mathbb B,\mathcal K]$ and a higher class of representations $[\mathbb B,\cat(\mathcal K)]$ of the same overall status and structure as $\cat(\mathcal K)$.  There is an inclusion of representations $[\mathbb B,\mathcal K]\to [\mathbb B,\cat(\mathcal K)]$ where those of the source are thought of as the ``discrete representations" relative to those of the target since $\mathcal K\to \cat(\mathcal K)$ is the inclusion into the ambient $(n+1)$-structure of the $(n+1)$-discrete structures (i.e. precisely the members of $\mathcal K$).  A category of elements construction then establishes correspondences with geometric structures on the other side of $n$ and $(n+1)$-equivalences making the whole following situation commute:
$$\begin{tikzpicture}
\node(1){$\mathbf{Fib}(\mathbb B)$};
\node(2)[node distance=1.2in, right of=1]{$[\mathbb B,\cat(\mathcal K)]$};
\node(3)[node distance=.7in, below of=1]{$\mathbf{DFib}(\mathbb B)$};
\node(4)[node distance=.7in, below of=2]{$[\mathbb B,\mathcal K].$};
\node(5)[node distance=.5in, right of=1]{$$};
\node(6)[node distance=.4in, below of=5]{$$};
\draw[->](1) to node [above]{$\simeq$}(2);
\draw[->](3) to node [left]{$\mathrm{incl}$}(1);
\draw[->](4) to node [right]{$\mathrm{incl}$}(2);
\draw[->](3) to node [below]{$\simeq$}(4);
\end{tikzpicture}$$
The question, then, is given examples of at least one of the representation ``hom-structures" in $[\mathbb B,\mathcal K]\to [\mathbb B,\cat(\mathcal K)]$, what is the (discrete) fibration concept on the other side of some such equivalence?  That is, what is the corresponding lifting property of some structure-preserving $n$-functor into the represented stucture $\mathbb B$.  Insofar as there are functor-category structures other than those considered so far, this is a potentially fruitful area of inquiry.  For there are plenty of 2- and 3-dimensional representing structures such as bicategories of profunctors or of relations; 2-categories with lax functors or lax natural transformations; double categories of sets or of profunctors; higher $n$-categories -- all of which have well-established associated functor categories and thus potentially corresponding fibration concepts waiting to be discovered.  In fact this outline may be interesting even in considering lower-level representations of classical and well-known algebraic gadgets such as groups, rings, modules and their higher-dimensional analogues such as 2-groups and 2-rigs. 

\section{Acknowledgments}

Some of the work of the paper was included in the author's thesis \cite{LambertThesis}.  That research was supported by the NSERC Discovery Grant of Dr. Dorette Pronk at Dalhousie University and by NSGS funding through Dalhousie University.  The author would like to thank Dr. Pronk for supervising the research and for numerous conversations about the much of the content of the paper.  The author would also like to thank Dr. Pieter Hofstra, Dr. Robert Par\'e and Dr. Peter Selinger for comments and suggestions on an earlier version of the work in the paper.

\bibliography{research}
\bibliographystyle{alpha}

\end{document}